\documentclass[a4paper,10pt]{article}

\usepackage[ngerman, english]{babel}
\usepackage[utf8]{inputenc}
\usepackage{amscd,amsfonts,amsmath,amssymb,amsthm,bbm,latexsym,mathrsfs,stmaryrd}
\usepackage[left=2.54cm,right=2.54cm,top=3cm,bottom=3.5cm]{geometry}
\usepackage[T1]{fontenc}
\usepackage{verbatim}
\usepackage{color}
\usepackage[scaled]{uarial}
\usepackage{lmodern} 
\usepackage{mathpazo}

\usepackage[english]{babel}
\usepackage[hidelinks]{hyperref}

\makeatletter \@addtoreset{equation}{section}
\makeatother

\makeatletter \@addtoreset{enunciato}{section}
\makeatother

\newcounter{enunciato}[section]

\newtheorem{ittheorem}{Theorem}
\newtheorem{itlemma}{Lemma}
\newtheorem{itproposition}{Proposition}
\newtheorem{itdefinition}{Definition}
\newtheorem{itremark}{Remark}
\newtheorem{itclaim}{Claim}
\newtheorem{itfact}{Fact}
\newtheorem{itconjecture}{Conjecture}
\newtheorem{itcorollary}{Corollary}

\newenvironment{theorem}{\addtocounter{enunciato}{1}
	\begin{ittheorem}}{\end{ittheorem}}

\newenvironment{lemma}{\addtocounter{enunciato}{1}
	\begin{itlemma}}{\end{itlemma}}

\newenvironment{proposition}{\addtocounter{enunciato}{1}
	\begin{itproposition}}{\end{itproposition}}

\newenvironment{definition}{\addtocounter{enunciato}{1}
	\begin{itdefinition}}{\end{itdefinition}}

\newenvironment{remark}{\addtocounter{enunciato}{1}
	\begin{itremark}}{\end{itremark}}

\newenvironment{conjecture}{\addtocounter{enunciato}{1}
	\begin{itconjecture}}{\end{itconjecture}}

\newenvironment{corollary}{\addtocounter{enunciato}{1}
	\begin{itcorollary}}{\end{itcorollary}}

\newcommand{\be}[1]{\begin{equation}\label{#1}}
\newcommand{\ee}{\end{equation}}

\newcommand{\bl}[1]{\begin{lemma}\label{#1}}
	\newcommand{\el}{\end{lemma}}

\newcommand{\br}[1]{\begin{remark}\label{#1}}
	\newcommand{\er}{\end{remark}}

\newcommand{\bt}[1]{\begin{theorem}\label{#1}}
	\newcommand{\et}{\end{theorem}}

\newcommand{\bd}[1]{\begin{definition}\label{#1}}
	\newcommand{\ed}{\end{definition}}

\newcommand{\bp}[1]{\begin{proposition}\label{#1}}
	\newcommand{\ep}{\end{proposition}}

\newcommand{\bc}[1]{\begin{corollary}\label{#1}}
	\newcommand{\ec}{\end{corollary}}

\newcommand{\bcj}[1]{\begin{conjecture}\label{#1}}
	\newcommand{\ecj}{\end{conjecture}}

\newcommand{\bpr}{\begin{proof}}
	\newcommand{\epr}{\end{proof}}

\newcommand{\overbar}[1]{\mkern 1.5mu\overline{\mkern-1.5mu#1\mkern-1.5mu}\mkern 1.5mu}

\DeclareMathOperator\dist{dist}
\DeclareMathOperator\diam{diam}
\DeclareMathOperator{\mph}{mp}

\def\N{\mathbb{N}}
\def\R{\mathbb{R}}

\def\P{\mathbb{P}}
\def\E{\mathbb{E}}

\def \cA {{\mathcal A}}

\def \cC {{\mathcal C}}

\def \cE {{\mathcal E}}

\def \cH {{\mathcal H}}

\def \cP {{\mathcal P}}

\def \cY {{\mathcal Y}}

\def \cc {{\rm c}}
\def \dd {{\rm d}}
\def \ee {{\rm e}}

\newcommand{\1}[1]{{\mathbbm{1}}_{#1}}

\newcommand{\eps}{\varepsilon}
\newcommand{\Prob}{\mathbf{P}}  

\newcommand{\scrY}{\mathscr{Y}}
\newcommand{\scrYf}{\mathscr{Y}_{\textnormal{f}}}
\newcommand{\scrK}{\mathscr{K}}
\newcommand{\frK}{\mathfrak{K}}

\allowdisplaybreaks[1]

\usepackage{array}
\newcolumntype{e}{>{\displaystyle}r @{\,} >{\displaystyle}c @{\,} >{\displaystyle}l}

\newtheorem{The}{Theorem}[section]
\newtheorem{Lem}[The]{Lemma}
\newtheorem{Pro}[The]{Proposition}
\newtheorem{Cor}[The]{Corollary}
\theoremstyle{remark}
\newtheorem{Rem}[The]{Remark}
\theoremstyle{definition}
\newtheorem{Def}[The]{Definition}

\begin{document}

\title{Brownian motion in attenuated or renormalized inverse-square Poisson potential}
\author{\renewcommand{\thefootnote}{\arabic{footnote}}
Peter Nelson~\footnotemark[1] 
\;\;\; and \, 
\renewcommand{\thefootnote}{\arabic{footnote}}
Renato Soares dos Santos~\footnotemark[2]
}

\footnotetext[1]{
Institut für Mathematik, 
Johannes Gutenberg-Universität,
Staudingerweg 9, D-55099 Mainz
}

\footnotetext[2]{
Department of Mathematics, Federal University of Minas Gerais (UFMG), Av.\ Pres.\ 
Ant\^onio Carlos 6627, Belo Horizonte
}
\maketitle

\vspace{10pt}
\noindent
{\bf Summary. }
We consider the parabolic Anderson problem with random potentials
having inverse-square singularities around the points of a standard Poisson point process in $\R^d$, $d \ge 3$.
The potentials we consider are obtained via superposition of translations over the points of the Poisson point process 
of a kernel $\mathfrak{K}$ behaving as $\mathfrak{K}(x) \approx \theta |x|^{-2}$ near the origin, where $\theta \in (0,(d-2)^2/16]$.
In order to make sense of the corresponding path integrals, 
we require the potential to be either \emph{attenuated} (meaning that $\mathfrak{K}$ is integrable at infinity)
or, when $d=3$, \emph{renormalized}, as introduced by Chen and Kulik in \cite{CK12}.
Our main results include existence and large-time asymptotics of non-negative solutions via Feynman-Kac representation.
In particular, we settle for the renormalized potential in $d=3$ the existence problem with critical parameter $\theta = 1/16$, 
left open by Chen and Rosinski in \cite{CR11}.

\vspace{10pt}


\section{Introduction and main results}

Fix $d \in \N$ and let $W=(W_t)_{t\geq 0}$ be a standard Brownian motion in $\R^d$.
We denote by $\P_x$ its law when started at $x$, and by $\E_x$ the corresponding expectation. 
Let $V:\R^d \to \R$ be a random potential function, which we take independent of $W$. 
The integral $\int_0^tV(W_s)\dd s$ represents the total potential energy along the Brownian path up to time $t$,
and is used to define the \emph{quenched Gibbs measure}
\begin{equation}\label{e:defquelaw}
Q_{t,x}(\cdot):=\frac{1}{Z_{t,x}} \E_x \left[\exp \left\{\int_0^tV(W_s)\dd s \right\} \mathbbm{1} \{W \in \cdot \} \right], \quad \text{ where } \quad Z_{t,x}:=\E_x\left[\exp \int_0^tV(W_s)\dd s \right],
\end{equation}
describing the behaviour of $W$ under the influence of the random potential.

A main feature in the study of Brownian motion in random potential is 
the connection to the (continuous) \emph{parabolic Anderson model}, 
i.e., the initial value problem
\begin{equation}\label{e:PAM}
\begin{aligned}
\partial_t u(t,x)&=\tfrac12 \Delta u(t,x) + V(x)u(t,x),& &(t,x)\in (0,\infty)\times \R^d, \\
u(0,x)&= u_0(x),& &x\in \R^d,
\end{aligned}
\end{equation}
where $\Delta=\sum_{i=1}^{d}\frac{\partial^2}{\partial^2x_i}$ denotes the (weak) Laplacian in $L^2(\R^d)$,
and $u_0 \in L^2_{\textnormal{loc}}(\R^d)$ is some initial data.
When $V$ is e.g.\ in the Kato class (cf.\ \cite[page~8, equation~(2.4)]{Szn98}),
the unique mild solution to \eqref{e:PAM} 
(in the sense of \cite[Definition~6.1.1]{Paz83})
is given by the classical \emph{Feynman-Kac formula}
\begin{align}\label{e:FKrepintro}
u(t,x)=\E_{x}\left[u_0(W_t)\exp\left\{\int_0^t V(W_s)\dd s\right\}\right].
\end{align}
In particular, $u(t,x) = Z_{t,x}$ in \eqref{e:defquelaw} solves \eqref{e:PAM} with $u_0 \equiv 1$.

In this paper, we are interested in \emph{Poisson (or shot noise) potentials}, 
obtained by superposing translations of a fixed function over the points of a Poisson cloud. 
To describe them, let $\omega$ be a standard Poisson point process in $\R^d$, i.e., 
having the Lebesgue measure as its intensity measure.
Denote by $\Prob$ the law of $\omega$, and by $\cP=\{y\in\R^d: \omega(\{y\}) > 0\}$ its support,
which is almost surely discrete.
For a Borel-measurable \emph{shape function} (or \emph{kernel}) $\frK:\R^d\rightarrow\R$, 
we define the Poisson potential
\[
V(x)=V(x,\omega)=\sum_{y\in\cP} \frK (x-y)=\int_{\R^d} \frK(x-y)\omega(\dd y), \quad x\in\R^d.
\]
Since the Gibbs measure in \eqref{e:defquelaw} favours paths with larger energy functional $\int_0^t V(W_s)\dd s$,
the points of $\omega$ will under it either attract the Brownian particle if $\frK$ is positive,
or repel it if $\frK$ is negative.

The study of the model of Brownian motion in a random Poisson potential is motivated by various applications from physics and other fields. 
Think, e.g., of an electron moving in a crystal with impurities, cf.\ \cite{BG90, HB87, Mol91}. 
For an overview on the mathematical treatment of the subject and further references, 
we refer the reader to the monographs~\cite{K16, Szn98}.
In \cite{Szn98}, essentially two types of potentials are considered: the \emph{soft} obstacle potential, 
where $\frK$ is assumed to be negative, bounded and compactly supported, and the \emph{hard} obstacle potential, 
where formally $\frK= - \infty\mathbbm{1}_C$ for some compact, nonpolar set $C\subset\R^d$, 
i.e., the Brownian particle is immediately killed when entering the $C$-neighbourhood of the Poisson cloud and moves freely up to the entrance time. 
The case of $\frK$ positive, bounded and continuous (and satisfying a decay property)
has been considered in \cite{CM95, GKM00}.
The works mentioned identify almost-sure large-time asymptotics for $Z_{t,x}$ in \eqref{e:defquelaw}.

It is of natural concern to study shape functions that are neither bounded nor have compact support. 
Kernels of the form $\frK(x)=|x|^{-p}$ are physically motivated,
e.g.\ $p=d-2$ corresponds to Newton's law of gravitation.
The inverse-square case $p=2$ is of special interest both in mathematics and physics (c.f.\ e.g.\ \cite{BG84a, BG84b, CH02, FMT07, GTV10, Sho31}),
and is related to the inverse-cube central force;
in this case, $\frK$ is \emph{not} in the Kato class (cf.\ \cite[Example~2.3, page~9]{Szn98}).
It turns out however that, when $p \le d$, 
the corresponding Poisson potential almost surely explodes, 
i.e.,
\begin{equation}\label{e:intV}
\int_{\R^d}|x-y|^{-p}\omega(\dd y)=\infty \qquad \Prob\text{-a.s.} \quad \text{for each } x \in \R^d,
\end{equation} 
cf.\ \cite[Proposition~2.1]{CK12}. 
Indeed,  when $p\leq d$, the integrability in \eqref{e:intV} is obstructed by the slow decay of the function $|x|^{-p}$ at infinity. 
To solve this problem, Chen and Kulik have constructed a \emph{renormalized} version $\overbar{V}$ of the Poisson potential $V$, 
formally written as
\begin{equation}\label{e:defoverbarV}
\overbar{V}(x)=\int_{\R^d} |x-y|^{-p}[\omega(\dd y) - \dd y], \quad x\in \R^d \; \text{ (where } |\cdot| \text{ is the Euclidean norm).}
\end{equation}
The mathematical definition of $\overbar{V}$ is as limit in probability of 
the same expression with integrable approximating kernels,
for which both integrals against $\dd y$ and $\omega(\dd y)$ are well defined;
for details, we refer the reader to \cite[Section~2]{CK12}.
This procedure is natural since, at each step of the approximation,
both $V$ and $\overbar{V}$ give rise to the same quenched Gibbs measure.
In \cite[Corollary~1.3]{CK12}, it is shown that \eqref{e:defoverbarV} is well-defined 
whenever $d/2<p<d$, in particular when $p=2,d=3$.

Even if \eqref{e:defoverbarV} is well defined, 
the exponential moment $Z_{t,x}$ in \eqref{e:defquelaw} (with $\overbar{V}$ in place of $V$) may still be infinite.
Indeed, Theorem 1.5 in \cite{CK12} states that, for $d/2<p<d$ and any $\theta,t> 0$,
\begin{equation}\label{e:FinNonCrit}
\E_0\left[\exp\left(\theta\int_0^t\overbar{V}(W_s)\dd s\right)\right]\nonumber\\
\begin{cases}
&<\infty \quad \Prob \mbox{-a.s.\ \ if \ } p<2, \\
&=\infty \quad \Prob \mbox{-a.s.\ \ if \ } p>2.
\end{cases} \quad  
\end{equation}
In the critical case $p=2$ (and necessarily $d=3$), 
the integrability depends on the value of the parameter $\theta$:
according to \cite[Theorem~2.1]{CR11}, for any $t>0$,
\begin{equation}\label{e:theta116}
\E_0\left[\exp\left(\theta\int_0^t\overbar{V}(W_s)\dd s\right)\right]\\
\begin{cases}
&<\infty \quad \Prob \mbox{-a.s.\ \ if \ } \theta<\frac{1}{16}, \\
&=\infty \quad \Prob \mbox{-a.s.\ \ if \ } \theta>\frac{1}{16}.
\end{cases} \quad  
\end{equation}
The boundary case $\theta=\frac{1}{16}$ is not considered in \cite{CR11}, 
and is included in our Theorem~\ref{t:finiteness} below.
The fact that $\theta=\frac{1}{16}$ is critical is related to the celebrated Hardy inequality
(in $d=3$)
\begin{equation}\label{e:Hardy1}
\frac{(d-2)^2}{8} \int_{\R^d}\frac{g^2(x)}{|x|^2}\dd x \leq \frac{1}{2} \| \nabla g \|_{L^2(\R^d)}^2, \qquad g\in H^1(\R^d),
\end{equation}
where $H^1(\R^d)$ is the Sobolev space of $L^2(\R^d)$ functions whose (weak) partial derivatives are also in $L^2(\R^d)$,
and the constant $(d-2)^2/8$ is sharp.

Once finiteness of exponential moments is settled,
our interest turns to large-time asymptotics.
In the non-critical regime $d/2<p<\min(2,d)$, $\theta>0$, 
it is shown in \cite[Theorem~2.2]{Che12} that
\begin{equation}\label{e:QueAsyChe}
\lim_{t\rightarrow\infty}\frac{1}{t}\left(\frac{\log\log t}{\log t}\right)^{\frac{2}{2-p}}\log\E_0\left[\ee^{\theta\int_0^t\overbar{V}(W_s)\dd s}\right]= c(d,p,\theta) \quad \Prob \mbox{-a.s.,}
\end{equation}
where $c(d,p,\theta)$ is an explicit deterministic constant depending only on $d,p,\theta$.
The case $p=2$, $d=3$, already considered in \cite{CR11}, turns out to be rather different: 
after suitable rescaling, the log of the exponential moment does \emph{not} converge to a constant,
but fluctuates randomly, cf.\ Theorem~\ref{t:tightness} below. 
Here we again extend the investigation to the boundary case $\theta = 1/16$.

Finally, we do not restrict our analysis to the renormalized potential $\overbar{V}$,
but also consider integrable versions of the inverse-square kernel. 
For this class of \emph{attenuated potentials}, cf.\ Definition~\ref{def:scrK} below, 
we show similar results as outlined above in all dimensions $d\ge3$;
in fact, our asymptotic results for $\overbar{V}$ in $d=3$ are obtained 
via comparison to attenuated potentials, cf.\ Theorem~\ref{t:limitfrac} below.

\subsection{Main results}
\label{ss:results}

Let $d \ge 3$.
We define next the class $\scrK$ of potential kernels we are after,
whose elements have an inverse-square singularity at the origin and are integrable at infinity.

\begin{Def}\label{def:scrK}
We say that a measurable $\frK:\R^d \to [-\infty, \infty]$ belongs to the class $\scrK$ if and only if
\begin{equation}\label{e:condscrK1}
y \mapsto \sup_{ |x| \leq 1 } |\frK(x-y)| \wedge 1 \quad \text{belongs to} \;\; L^1(\R^d)
\end{equation}
and
\begin{equation}\label{e:condscrK2}
\limsup_{a \downarrow 0} \max \left\{ a^2 \sup_{|x| > a} |\frK(x)|, \; \sup_{|x|\le a } \left| \frK(x) -\frac{1}{|x|^{2}} \right| \right\} < \infty.
\end{equation}
\end{Def}
\noindent
We call $\scrK$ the class of \emph{attenuated inverse-square potential kernels}.

Given $\frK \in \scrK$, we denote the Poisson potential with kernel $\frK$ by
\begin{equation}\label{e:defVK}
V^{(\frK)}(x) := \int_{\R^d} \frK(x-y) \omega( \dd y), \qquad x \in \R^d \setminus \cP.
\end{equation}
By \cite[Proposition~2.1]{CK12}, $V^{(\frK)}$ is a.s.\ well-defined and finite.
Important examples are the truncated kernels $\frK_a(x) := |x|^{-2} \mathbbm{1}_{\{|x| \le a\}}$, $a>0$,
in which case we abbreviate $V^{(a)}:=V^{(\frK_a)}$.

To state our results for $V^{(\frK)}$, denote by
\begin{equation}\label{e:defhd}
h_d := \frac{(d-2)^2}{8}
\end{equation}
as in the form \eqref{e:Hardy1} of Hardy's inequality, and set, for $\theta \in (0,h_d/2]$,
\begin{equation}\label{e:defkatheta}
k_\theta := \left\lfloor \frac{h_d}{\theta} \right\rfloor \ge 2.
\end{equation}
Our first two results show existence of solutions to \eqref{e:PAM} via Feynman-Kac representation.
\begin{The}\label{t:finitenessva}
For all $d\geq 3$, $\frK \in \scrK$ and $\theta\in (0,\frac{h_d}{2}]$,
it holds $\Prob$-almost surely that
\begin{equation}\label{e:defvk}
v^{(\frK)}_\theta(t,x) :=	\E_x\left[\exp\left(\theta\int_0^t |V^{(\frK)}|(W_s)\dd s\right)\right] < \infty
\qquad \forall x\in\R^d\setminus \cP, \, t \geq 0.
\end{equation}
Moreover, $v^{(\frK)}_\theta(t, \cdot) \in L^1_\textnormal{loc}(\R^d)$ for all $t \geq 0$.
As a consequence, for all $u_0 \in L^\infty(\R^d)$ the function
\begin{equation}\label{e:finitenessva}
u^{(\frK, u_0)}_\theta(t,x) :=	\E_x\left[ u_0(W_t) \exp\left(\theta\int_0^t V^{(\frK)}(W_s)\dd s\right)\right]
\; \text{ is well-defined for all } x\in\R^d\setminus \cP, \, t \geq 0,
\end{equation}
and $u^{(\frK, u_0)}_\theta(t, \cdot) \in L^1_{\textnormal{loc}}(\R^d)$ for all $t\geq 0$.
The same is true with $|V^{(\frK)}|$ in place of $V^{(\frK)}$.
\end{The}

\begin{The}\label{t:solva}
The function $u^{(\frK, u_0)}_\theta$ defined in \eqref{e:finitenessva} is a mild solution to \eqref{e:PAM} with $V = \theta V^{(\frK)}$.
The analogous holds for $|V^{(\frK)}|$ in place of $V^{(\frK)}$.
\end{The}

When $\theta>h_d/2$, \eqref{e:defvk} is infinite even with $V^{(\frK)}$ in place of $|V^{(\frK)}|$;
a proof can be obtained as in \cite[Theorem~2.1]{CR11}.
In the following, when $u_0 \equiv 1$ we write $u_\theta^{(\frK)}$ instead of $u_\theta^{(\frK, u_0)}$. 

Our next three results concern large time asymptotics of $u^{(\frK)}_\theta(t,0)$,
starting with tightness.
\begin{The}\label{t:tightnessva} 
Let $d\geq3$, $\frK \in \scrK$ and $\theta\in(0,\frac{h_d}{2}]$.
For any $t \mapsto g(t)>0$ with $g(t)\overset{t\rightarrow\infty}{\longrightarrow} \infty$,
\begin{equation}\label{e:tightnesslbva}
g(t)t^{-\frac{k_\theta+1}{k_\theta-1}}\log\E_0\left[\exp\left(\theta\int_0^tV^{(\frK)}(W_s)\dd s\right)\right]\overset{t\rightarrow\infty}{\longrightarrow} \infty \quad \mbox{in probability}
\end{equation}
and
\begin{equation}\label{e:tightnessubva}
g(t)^{-1}t^{-\frac{k_\theta+1}{k_\theta-1}}\log\E_0\left[\exp\left(\theta\int_0^tV^{(\frK)}(W_s)\dd s\right)\right]\overset{t\rightarrow\infty}{\longrightarrow} 0 \quad \mbox{in probability}.
\end{equation}
In other words, the process $t^{-\frac{k_\theta+1}{k_\theta-1}} \log u^{(\frK)}_\theta(t,0)$, $t>0$ is tight on the open interval $(0,\infty)$.
\end{The}

The following two theorems provide almost-sure $\limsup$ and $\liminf$ asymptotics.
\begin{The}\label{t:limsupva}
Let $d\geq 3$, $\frK \in \scrK$ and $\theta\in (0,\frac{h_d}{2}]$. 
For any slowly varying $\ell\colon(0,\infty)\to(1,\infty)$,
\begin{align}\label{e:limsupva}
\limsup_{t\rightarrow\infty}t^{-\frac{k_\theta+1}{k_\theta-1}} \ell(t)^{-\frac{2}{d(k_\theta-1)}}\log u^{(\frK)}_\theta(t,0) =
\left\{
\begin{array}{ll}
0 & \Prob \mbox{-a.s.} \ \mbox{ if } \ \int_{1}^{\infty}\frac{\dd r}{r \ell(r)}<\infty, \\
\infty & \Prob \mbox{-a.s.} \ \mbox{ if } \ \int_{1}^{\infty}\frac{\dd r}{r \ell(r)}=\infty.
\end{array}\right.
\end{align}
\end{The}

\begin{The}\label{t:liminfva} 
For any $d\geq 3$ and $\theta\in(0,\frac{h_d}{2}]$, there exist $0<C_{\inf}<C^{\inf}<\infty$ such that, for all $\frK \in \scrK$,
\begin{align}\label{e:liminfva}
\liminf_{t\rightarrow\infty}t^{-\frac{k_\theta+1}{k_\theta-1}} ( \log \log t )^{\frac{2}{d(k_\theta-1)}} \log u^{(\frK)}_\theta(t,0) 
\in [C_{\inf},C^{\inf}] \quad \Prob \mbox{-a.s.}
\end{align}
\end{The}

Corresponding results also hold for the renormalized potential $\overbar{V}$ when $d=3$.
We start with the analogues of Theorems~\ref{t:finitenessva}--\ref{t:solva}.
\begin{The}\label{t:finiteness}
Let $d=3$.
For each $\theta\in (0,1/16]$, it holds $\Prob$-almost surely that
\begin{equation}\label{e:defvbar}
\overbar{v}_\theta(t,x) :=	\E_x\left[ \exp\left(\theta\int_0^t|\overbar{V}|(W_s)\dd s\right)\right] < \infty
\qquad \forall x\in\R^3\setminus \cP, \, t\geq 0.
\end{equation}
Moreover, $\overbar{v}_\theta(t, \cdot)$ belongs to $L^1_{\textnormal{loc}}(\R^3)$ for all $t \geq 0$.
As a consequence, for all $u_0 \in L^\infty(\R^3)$ the function
\begin{equation}\label{e:finiteness}
\overbar{u}^{(u_0)}_\theta(t,x) :=	\E_x\left[ u_0(W_t) \exp\left(\theta\int_0^t\overbar{V}(W_s)\dd s\right)\right]
 \; \text{ is well-defined for all } x\in\R^3\setminus \cP, \, t\geq 0,
\end{equation}
and $\overbar{u}^{(u_0)}_\theta(t,\cdot) \in L^1_{\textnormal{loc}}(\R^3)$ for all $t \geq 0$.
The analogous holds for $|\overbar{V}|$ in place of $\overbar{V}$.
\end{The}

\begin{The}\label{t:solvbar}
The function $\overbar{u}^{(u_0)}_\theta$ defined in \eqref{e:finiteness} is a mild solution to \eqref{e:PAM} with $d=3$, and $V = \theta \overbar{V}$.
The analogous holds for $|\overbar{V}|$ in place of $\overbar{V}$.
\end{The}

When $u_0 \equiv 1$, we will write $\overbar{u}_\theta$ instead of $\overbar{u}^{(u_0)}_\theta$. 
Our next theorem provides a convenient comparison between potential kernels, 
allowing us to concentrate on the truncated case $\frK_a(x)=|x|^{-2}\mathbbm{1}_{\{|x|\leq a\}}$.
\begin{The}\label{t:limitfrac}
For any $\theta \in (0,h_d/2]$, any $a \in (0,\infty)$ and any $\frK \in \scrK$,
\begin{equation}\label{e:limitfrac}
\lim_{t \to \infty} \frac{\log u^{(\frK)}_\theta(t,0)}{\log u^{(\frK_a)}_\theta(t,0)} = 1 \qquad \Prob \mbox{-almost surely.}
\end{equation}
The same is true with $v^{(\frK)}_\theta$ in place of $u^{(\frK)}_\theta$.
When $d=3$, \eqref{e:limitfrac} also holds with either $\overbar{v}_\theta$ or $\overbar{u}_\theta$ in place of $u^{(\frK)}_\theta$.
\end{The}

Finally, using Theorem~\ref{t:limitfrac}, we can transfer our results for $V^{(\frK)}$ to $|V^{(\frK)}|$, $\overbar{V}$ and $|\overbar{V}|$:
\begin{The}\label{t:tightness}
The statements of Theorems~\ref{t:tightnessva}, \ref{t:limsupva} and \ref{t:liminfva} are true for $v^{(\frK)}_\theta$ in place of $u^{(\frK)}_\theta$.
When $d=3$, these statements also hold with either $\overbar{v}_\theta$ or $\overbar{u}_\theta$ in place of $u^{(\frK)}_\theta$.
\end{The}

We discuss next our results and provide some heuristics for the scale $t^{(k_\theta +1)/(k_\theta - 1)}$.

\subsection{Discussion and heuristics}
\label{ss:discussion}

\noindent
{\bf 1)}
Theorems~\ref{t:tightnessva}--\ref{t:liminfva}
imply that there is no rescaling under which $\log u^{(\frK)}_\theta(t,0)$
converges almost surely as $t\to\infty$ to a non-trivial deterministic constant (and analogously for $\log \overbar{u}_\theta(t,0)$ by Theorem~\ref{t:tightness}).
We conjecture that, after rescaling by $t^{(k_\theta+1)/(k_\theta-1)}$, 
it converges in distribution to a non-degenerate random variable.
We also conjecture that the $\liminf$ in Theorem~\ref{t:liminfva} is deterministic.

\vspace{10pt}
\noindent
{\bf 2)}
As already mentioned, our main contribution in Theorems~\ref{t:finiteness}, \ref{t:solvbar} and \ref{t:tightness}
is the boundary case $\theta=1/16$, left open in \cite{CR11}.
The proof given in \cite{CR11} that $\overbar{u}_\theta$ is well-defined for $0<\theta<\frac{1}{16}$ 
cannot be extended to the case $\theta=\frac{1}{16}$, as it is based on the following strategy.
Decompose the Brownian path according to 
which of the cubes $Q_n=(-R_n,R_n)^3$ has been exited until time $t$, 
where $(R_n)_{n\in\N}$ is some properly chosen increasing sequence; i.e.,
setting $\tau_0=0$, $\tau_n=\inf\{s\geq 0\colon W_s\notin Q_n\}$, write
\begin{align*}
\E_0\left[\exp\left(\theta\int_0^t\overbar{V}(W_s)\dd s\right)\right]
&	=\sum_{n=1}^{\infty}\E_0\left[\exp\left(\theta\int_0^t\overbar{V}(W_s)\dd s\right)\mathbbm{1}_{\{\tau_{n-1}\leq t< \tau_n\}}\right]\\
&\leq \sum_{n=1}^{\infty}\Prob \left[\tau_{n-1}\leq t\right]^{1/p}\E_0\left[\exp\left(q\theta\int_0^t\overbar{V}(W_s)\dd s\right)\mathbbm{1}_{\{t< \tau_n\}}\right]^{1/q}
\end{align*}
by Hölder's inequality, where $p,q>0$, $p^{-1}+q^{-1}=1$.
The last expectation cannot be controlled if $q \theta > 1/16$, 
and thus $\theta < 1/16$ is required to use this argument.
In order to overcome this, we develop for our proof a more careful decomposition of Brownian paths
according to their excursions to and from certain \emph{islands} whose principal eigenvalues are large,
cf.\ Section~\ref{s:expansions} below.

\vspace{10pt}
\noindent
{\bf 3)} Even though we only prove Theorems~\ref{t:finitenessva}--\ref{t:solva} and \ref{t:finiteness}--\ref{t:solvbar}
for initial data $u_0 \in L^\infty(\R^d)$, a close inspection of our proofs will show that
they may be generalized to $u_0 \in L^\infty_\textnormal{loc}(\R^d)$ 
satisfying e.g.\ 
\[
\limsup_{R \to \infty} \frac{\log \log \| u_0 \|_{L^\infty(B_R)}}{\log R} < \frac{1}{k}.
\]
The key step is to adapt the statement of Lemma~\ref{l:keyUB} taking the above condition into account.

\vspace{10pt}
\noindent
{\bf 4)}
Let $\E^t_{x,y}$ denote expectation under the law of a Brownian bridge between $x,y \in \R^d$ in time $t$,
and write $p_t(x-y) = (2 \pi t)^{-d/2}\ee^{-|x-y|^2/(2t)}$
for the Brownian transition kernel. 
Since the law of the Brownian bridge is a regular conditional probability for the 
law of Brownian motion given $W_t = y$ (cf.\ e.g.\ \cite[Appendix to Part~I]{Szn98}), 
Theorem~\ref{t:finitenessva} implies that, $\Prob$-a.s., for all $R>0$,
\[
\int_{B_R \times \R^d} p_t(x-y) \E^t_{x,y} \left[\exp \left( \theta \int_0^t |V^{(\frK)}|(W_s) \dd s \right) \right] \dd x \, \dd y
= \int_{B_R} v^{(\frK)}_\theta(t, x) \dd x < \infty
\]
and thus, for all $t\geq0$ and almost every $x,y \in \R^d$, the function
\begin{equation}\label{e:deffundsolBB}
(t,x) \mapsto p_t(x-y) \E^t_{x,y} \left[\exp \left( \theta \int_0^t V^{(\frK)}(W_s) \dd s \right) \right]
\end{equation}
is well-defined and, by Theorem~\ref{t:solva}, solves \eqref{e:PAM} with initial condition $u_0 = \delta_y$. 
In $d=3$, the analogous is true for $\overbar V$ by Theorems~\ref{t:finiteness}--\ref{t:solvbar}.
Using the techniques of Section~\ref{s:expansions} and Section~\ref{ss:UB}, we could extend the definition of \eqref{e:deffundsolBB}
to all $x,y \in \R^d \setminus \cP$; in the interest of brevity, we will not pursue this here.

\vspace{10pt}
\noindent
{\bf 5)}
We provide next some heuristics for the scale $t^{(k_\theta+1)/(k_\theta-1)}$ appearing in Theorem~\ref{t:tightnessva}.
The main point is that the logarithmic order of $u^{(\frK)}_\theta(t,0)$ 
is the same when restricting the expectation to Brownian paths that reach by time $s<t$
a region $D\subset \R^d$ containing precisely $k_\theta+1$ Poisson points,
and afterwards stay there until time $t$.
Spectral methods show that the reward for staying in $D$ for time $t-s$
is approximately $\ee^{(t-s) \lambda_{\max}}$,
where $\lambda_{\max}$ is the principal Dirichlet eigenvalue of $\tfrac12 \Delta + V^{(\frK)}$
in $D$.
Asymptotics for this eigenvalue may be estimated with the
help of \emph{multipolar Hardy inequalities} as in \cite{BDE08} (see also Section~\ref{ss:MPHI} below),
yielding that its order roughly equals $\diam(D)^{-2}$.
Now, if $R$ is the distance of $D$ to the origin, Poisson statistics dictate that it may be chosen with $\diam(D) \approx R^{-1/k_\theta}$,
but not much smaller.
On the other hand, the probabilistic cost for Brownian motion to reach $D$ by time $s$ is roughly $\ee^{-R^2/s}$.
The total contribution is thus about $\exp \{(t-s)R^{2/k_\theta} - R^2/s\}$;
optimizing the exponent over $s$ and $R$, we obtain $R=t^{k_\theta/(k_\theta+1)}$
and $\log u^{(\frK)}_{\theta}(t,0) \approx t^{(k_\theta+1)/(k_\theta-1)}$.

\subsection{Outline and notation}
\label{ss:outline}
The rest of the paper is organized as follows.
After introducing some notation,
we develop in Section~\ref{s:DSB} 
upper and lower spectral bounds on the Feynman-Kac functional \eqref{e:FKrepintro} in the setting of deterministic point clouds. 
The upper bounds are extended in Section~\ref{s:expansions} using a path decomposition technique.
Section~\ref{s:SDPC} presents some elementary geometric properties of the standard Poisson point process. 
The proofs of the main theorems are completed in Section~\ref{s:PMT}.

\vspace{10pt}
\noindent
{\bf Notation and terminology.}
We write $B_r(x)=\{y\in \R^d\colon |x-y|<r\}$ for the open ball with radius $r\in(0,\infty)$ 
around $x\in\R^d$ with respect to the Euclidean norm $|\cdot|$; 
when $x=0$ we abbreviate $B_r:=B_r(0)$.
For $D\subset \R^d$, we write $B_r(D)=\{x\in \R^d\colon \exists y\in D, |x-y|<r\}$ for the $r$-neighbourhood of $D$.
We denote by $|D|$ the volume of a Borel measurable subset $D\subset \R^d$,
and by $\tau_D :=\inf\{t\geq 0\colon  W_t\in D\}$ the entrance time of Brownian motion in $D$.
A subset $D \subset \R^d$ is called a \emph{domain} if it is open and connected.
For a real-valued function $f$, a positive function $g$ and $a \neq 0$, we write $f(x) \sim a g(x)$ 
as $x \to \infty$ to denote that $\lim_{x \to \infty} f(x)/g(x) = a$;
when $a = 0$, we write $f = o(g)$ instead, or equivalently $|f| \ll g$ or $g \gg |f|$.
We write $f = \mathcal{O}(g)$ as $x\to\infty$ if there exists a constant $C\in (0,\infty)$ such that $f(x) \le C g(x)$ for all large enough $x$.
We write $\log^+ x := \log (x \vee \ee)$, $x \in \R$.

\section{Deterministic spectral bounds}
\label{s:DSB}

In this section, we consider Brownian motion in $\R^d$, $d \ge 3$, moving among a deterministic point cloud.
Our goal is to obtain lower and upper spectral bounds in $L^1$ and $L^\infty$ for relevant Feynman-Kac formulae.
First we collect some basic tools from the theory of Schrödinger operators (Section \ref{ss:SO}),
which are then applied to derive upper bounds on both time-dependent and stopped Feynman-Kac functionals (Section~\ref{ss:UB}). 
After that, we obtain a lower bound for the time-dependent functional (Section \ref{ss:LB}), 
and conclude the section with a multipolar Hardy inequality (Section \ref{ss:MPHI}).

Define the family of non-empty, locally finite subsets of $\R^d$
\begin{equation}\label{e:locfin}
\scrY=\{ \cY\subset \R^d \colon\, \cY \neq \emptyset, \, \# K \cap \cY < \infty \; \forall \text{ compact } K \subset \R^d\},
\end{equation} 
as well as the family of non-empty, finite subsets
\begin{equation}\label{e:fin}
\scrYf=\{ \cY \in \scrY \colon\, \#\cY<\infty\}.
\end{equation} 
Note that the support $\cP=\{x\in\R^d\colon\omega(\{x\})=1\}$ of the Poisson point process $\omega$ 
belongs almost surely to $\scrY$.
For $\cY\in\scrY$ and $a\in(0,\infty]$
satisfying either $\cY \in \scrYf$ or $a < \infty$, let
\begin{equation}\label{e:deffinpot}
V_\cY^{(a)}(x)=\sum_{y\in\cY}\frac{\mathbbm{1}_{\{|x-y| \leq a\}}}{|x-y|^2}, \quad  x\in\R^d\setminus \cY.
\end{equation}
When $a<\infty$, $V_\cP^{(a)} = V^{(\frK_a)}$ as in \eqref{e:defVK} with $\frK_a(x) = |x|^{-2} \mathbbm{1}_{\{|x|\le a\}}$. 
For $\cY\in\scrYf$, we write $V_\cY=V_\cY^{(\infty)}$.
%
\subsection{Preliminaries on Schrödinger operators and the Feynman-Kac formula}
%
\label{ss:SO}
The content of this section is classical and has been treated by many authors. 
Our major references here are the books \cite{EN91} by Engel and Nagel and \cite{CZ95} by Chung and Zhao.

Let $D\subset\R^d$ be an open subset. 
By $H^1_0(D)$ we denote the Sobolev space on $D$ with zero-boundary condition, i.e. the closure of the space $C_c^\infty(D)$ of smooth, compactly supported functions on $D$ with respect to the Sobolev norm $\|f\|_{H^{1}(D)}=\sum_{1\leq i \le d}\|\partial_i f\|_{L^2(D)}$, where ``$\partial_i$'' denotes differentiation with respect to the $i$-th coordinate.
For a potential $q \in L^1_{\textnormal{loc}}(D)$, we define
\begin{align}\label{e:deflambdamax}
\lambda_{\max}(D,q) := \sup_{g \in H_0^1(D),\, \|g\|_{L^2(D)} =1} \left\{ \int_{D} q(x) g(x)^2 \dd x - \frac12\int_{D} |\nabla g(x)|^2 \dd x \right\} \; \in \R \cup \{+\infty\}.
\end{align}

Note that $\lambda_{\max}(D, q) \ge 0$ if $q \ge 0$; more generally, the following monotonicity property holds.

\begin{Rem}\label{r:evmonoton}
Let $D_1\subset D_2\subset \R^d$ be open and $q_1\in L_{\textnormal{loc}}^1(D_1), q_2\in L^1_{\textnormal{loc}}(D_2)$ with $q_1\leq q_2$ on $D_1$. Then
\begin{equation}\label{e:evmonoton}
\lambda_{\max}(D_1,q_1)\leq \lambda_{\max}(D_2,q_2).
\end{equation}
\end{Rem}

When $q$ has some regularity (e.g.\ when it is in the Kato class), $\lambda_{\max}(D,q)$ is the supremum of the spectrum of
the \emph{Schr\"odinger operator} $\cH_q = \Delta + q$ in $L^2(D)$ with zero Dirichlet boundary conditions,
where $\Delta$ is the weak Laplacian whose domain is dense in $H_0^1(D)$.
This holds in particular when
\begin{equation}\label{e:qbounded}
q \in L^\infty(D),
\end{equation}
in which case $\lambda_{\max}(D,q) < \infty$ and $\cH_q$ is a closed self-adjoint operator
generating a strongly continuous semigroup $(T_t)_{t\geq 0}=(\ee^{t \cH_q})_{t\geq 0}$ on $L^2(D)$
(see e.g.\ \cite[Proposition~3.29]{CZ95}).
We will assume \eqref{e:qbounded} in the remainder of this subsection.

An important fact about $\lambda_{\max}$ is that it controls the growth of $T_t$ via the inequality
\begin{equation}\label{e:FKL2B}
\|T_t f\|_{L^2(D)}\leq \|f\|_{L^2(D)} \exp\{t\lambda_{\max}(D, q)\} \qquad \forall\, t\geq 0,
\end{equation}
cf.~e.g.~\cite[Equation~(30), Section~8.3]{CZ95}. 
From this we get the basic but crucial bound for the resolvent 
of the operator $\cH_q$ (cf.\ e.g.\ \cite[Theorem~II.1.10]{EN91}): for $\gamma>\lambda_{\max}(D,q)$,
\begin{equation}\label{e:resbound}
\|(\cH_q-\gamma)^{-1}\|_{L^2(D)\rightarrow L^2(D)}\leq \frac{1}{\gamma-\lambda_{\max}(D,q)}.
\end{equation}

The semigroup $(T_t)_{t\geq 0}$ can be used to solve the initial boundary value problem
\begin{align}\label{e:PAMD}
\partial_t u(t,x)&=\frac{\Delta}{2}u(t,x) + q(x)u(t,x),& &(t,x)\in [0,\infty)\times D\\
\label{e:PAMBC}u(t,x)&=0,& &(t,x)\in [0,\infty)\times \partial D\\
\label{e:PAMIC}u(0,x)&=u_0(x),& &x\in D
\end{align}
with initial data $u_0\in L^2(D)$, as follows. 
We want to consider solutions to \eqref{e:PAMD}-\eqref{e:PAMIC}
in the \emph{mild} sense (cf.\ \cite[Definition~6.1.1]{Paz83}), 
i.e., we demand that
\begin{equation}\label{e:defmild1}
\int_0^t \int_{D}p_{t-s}(x-y)|q(y)u(s,y)|\dd y\dd s<\infty \qquad \forall \, x\in D, t>0
\end{equation}
and 
\begin{equation}\label{e:defmild2}
u(t,x)= u_0(x) + \int_0^t \int_{D}p_{t-s}(x-y)q(y)u(s,y)\dd y\dd s \qquad \forall\, x\in D, t>0,
\end{equation}
where $p_t(x)$ is the Gaussian density
\begin{equation}\label{e:Gaussiandensity}
p_t (x) := (2\pi t)^{-d/2}\exp\{-|x|^2/(2t)\},
\end{equation}
i.e., the transition density of Brownian motion at time $t$ started from $0$.

The next proposition characterizes the mild solutions to \eqref{e:PAMD}--\eqref{e:PAMIC},
connecting Schrödinger semigroups and Brownian motion via the celebrated Feynman-Kac representation:
\begin{Pro}[Feynman-Kac formula] 
\label{p:FKrep}
Under \eqref{e:qbounded},
the unique mild solution to \eqref{e:PAMD}-\eqref{e:PAMIC} is given by
\begin{align}u(t,x)= T_t u_0 (x) = \E_{x}\left[u_0(W_t)\exp\left( \int_0^t q(W_s)\dd s\right)\mathbbm{1}_{\{\tau_{D^\cc} > t\}}\right].\label{e:FKrep}
\end{align}
\end{Pro}
\begin{proof}
Follows from e.g.\ \cite[Proposition~II.6.4]{EN91}) and \cite[Theorems~3.17 and 3.27]{CZ95}.
\end{proof}

Additionally to the time-dependent Feynman-Kac formula \eqref{e:FKrep},
we will use a \emph{stopped} Feynman-Kac formula as follows.  
Consider the time-independent Schrödinger equation
\begin{equation}\label{e:FK2}
\begin{aligned}
\frac{\Delta}{2}u(x)+ q(x) u(x) &= \gamma u(x), \;\, \quad x \in D,\\
u(x)&=f(x), \quad \quad  x \in \partial D,
\end{aligned}
\end{equation}
with $f\colon \partial D\rightarrow \R$ continuous
and $\gamma  \in \R$.
A function $u\in L^1_{\textnormal{loc}}(D)$ is called a \emph{weak solution} to \eqref{e:FK2} if 
\begin{equation}
\int_D u(x)\Delta\phi(x)\dd x = -2\int_D (q(x)-\gamma) u(x)\phi(x) \dd x 
\end{equation}
for all $\phi\in C_c^\infty(D)$, and $u$ is continuous on $\overbar{D}$ with $u=f$ on $\partial D$.
Recall that $D$ is called \emph{regular} if $\P_x(\tau_{D^\cc}=0)=1$ for all $x\in\partial D$. 
The next result follows from \cite[Theorems~4.7~and~4.19]{CZ95}.
\begin{Pro}
\label{p:FKsol}
Assume \eqref{e:qbounded}. If $D$ is a bounded regular domain and $\gamma > \lambda_{\max}(D,q)$, then
\begin{align}
\label{e:FKsol}
u(x):=\E_{x}\left[\exp\left(\int_0^{\tau_{D^\cc}}(q(W_s) - \gamma) \dd s\right)f(W_{\tau_{D^\cc}})\right]
\end{align}
is the unique weak solution to the boundary value problem \eqref{e:FK2}.
\end{Pro}

\subsection{Upper bounds}
\label{ss:UB}
Let $D\subset \R^d$ be a bounded regular domain. Recall $h_d = (d-2)^2/8$. Fix $\theta \in (0,h_d]$,
$\cY \in \scrYf$ with $\cY\subset D$ and put $M:= \#\cY$.
We give next $L^1$ upper bounds for both stopped and time-dependent Feynman-Kac functionals
with potentials of the form \eqref{e:deffinpot}.
We note that, by \cite[Theorem~1]{BDE08}, $\lambda_{\max}(\R^d, \theta V_\cY) < \infty$;
by Remark~\ref{r:evmonoton}, also $\lambda_{\max}(D, \theta V_\cY^{(a)}) < \infty$ for any $a > 0$.

\begin{Lem}
\label{l:L1res} 
For any measurable $D' \subset \R^d$, 
any $a\in(0,\infty]$ and any $\gamma>\lambda_{\max}(D,\theta V^{(a)}_{\cY})$,
\begin{equation}\label{e:L1timedep}
\int_{D'} \E_x\left[\ee^{\int_0^t (\theta V_{\cY}^{(a)}(W_s)-\gamma) \dd s} \1{\{\tau_{D^\cc} > t\}}\right] \dd x \le 
\sqrt{|D' \cap D| |D|}.
\end{equation}
Moreover, there exists a constant $c=c(d) \in (0,\infty)$ independent of $D$, $\theta$, $\cY$, $\gamma$, $a$, $D'$ 
such that
\begin{equation}\label{e:L1res}
\int_{D'} \E_x\left[\ee^{\int_0^{\tau_{D^{\cc}}} (\theta V_{\cY}^{(a)}(W_s)-\gamma) \dd s} \right] \dd x \le 
|D'| + c \sqrt{|D| |D' \cap D|} \frac{\gamma+(M^2+\theta) \, {\dist(D^{\cc}, \cY)}^{-2}}{\gamma-\lambda_{\max}(D,\theta V^{(a)}_{\cY})}.
\end{equation}
\end{Lem}
\begin{proof}
Fix $D'\subset \R^d$ measurable, $a>0$ and $\gamma>\lambda_{\max}(D,\theta V^{(a)}_{\cY})$. 
Note that, when $x \in D^\cc$, the integrands in \eqref{e:L1timedep} and \eqref{e:L1res} are respectively equal to $0$ and $1$,
and thus we may assume that $D'\subset D$.

We start with \eqref{e:L1timedep}.
For $m \in \N$, let $F_m = \min(V^{(a)}_{\cY},m)$ 
and write $(T_t^{(m)})_{m\in\N}$ for the Schrödinger semigroup associated with the potential $q = \theta F_m$ as in \eqref{e:FKrep}. 
Note that, for all $m\in \N$,
\begin{equation}\begin{split}\label{e:prL1timedep1}
&\int_{D'} \E_x \left[\ee^{\int_0^t (\theta F_m(W_s) -\gamma) \dd s} \mathbbm{1}_{\{\tau_{D^{\cc}}> t\}} \right] \dd x 
= \ee^{-t\gamma}\langle \mathbbm{1}_{D'}, T^{(m)}_t \mathbbm{1}_{D} \rangle_{L^2(D)} \\
& \qquad\qquad\qquad\qquad\le \ee^{-t\gamma}\| \mathbbm{1}_{D'}\|_{L^2(D)} \|T_t^{(m)}\|_{{L^2(D)} \to {L^2(D)}} \|\mathbbm{1}_D \|_{L^2(D)}\\&\qquad\qquad\qquad\qquad\le \ee^{-t\gamma}\sqrt{|D'||D|} \, \ee^{t \lambda_{\max}(D,\theta V^{(a)}_{\cY}\wedge m)}\leq \sqrt{|D'||D|},
\end{split}
\end{equation}
where we used the Cauchy-Schwarz inequality and 
$\lambda_{\max}(D,\theta F_m) \leq \lambda_{\max}(D,\theta V^{(a)}_{\cY}) < \gamma$
by Remark~\ref{r:evmonoton}.
Letting $m\rightarrow\infty$, \eqref{e:L1timedep} follows by monotone convergence.

Consider now \eqref{e:L1res}.
By Proposition~\ref{p:FKsol}, the function
$u_m(x)=\E_x\left[\exp\int_0^{\tau_{D^{\cc}}} (\theta F_m(W_s)-\gamma) \dd s\right]$ 
is the unique weak solution to the boundary value problem
\begin{equation}\label{e:proofFK1}
\begin{split}
\left(\frac{\Delta}{2}+\theta F_m-\gamma\right)u(x)&=0, \quad x \in D\\
u(x)&=1, \quad x \in \partial D.
\end{split}
\end{equation}
Abbreviate $\delta := \dist(D^{\cc}, \cY)$, take $g\colon\R \to [0,1]$ smooth 
with $g(r) = 0$ for $r \le 1/2$ and $g(r)=1$ for $r \ge 1$, and put $\phi(x) := \prod_{y \in \cY} g(|x-y|/\delta)$. 
We may check that $\phi\in C^2(\R^d)$, $0\le\phi\le 1$ on $D$, $\phi\equiv 1$ on $D^{\cc}$, 
and there exists a constant $c=c(d)\in(1,\infty)$, not depending on $D$, $\theta$ or $\cY$,
such that $|\Delta \phi|\le 2cM^2\delta^{-2}$ and $\phi V_{\cY}\le c\delta^{-2}$ uniformly on $\R^d$.
Moreover, $v_m:=u_m-\phi$ solves
\begin{equation}\label{e:proofFK2}
\begin{split}
\left(\frac{\Delta}{2}+\theta F_m-\gamma \right)v_m(x)&=- \left(\frac{\Delta}{2}+\theta F_m-\gamma\right)\phi(x), \quad x \in D,\\
v_m(x)&=0, \qquad \qquad \qquad \qquad \qquad \;\;\,\, x \in \partial D,
\end{split}
\end{equation}
i.e., $v_m=-\mathcal{R}_{\gamma}^{(m)}\left(\frac{\Delta}{2}+\theta F_m-\gamma\right)\phi$ 
where $\mathcal{R}_{\gamma}^{(m)}$ is the resolvent of $\tfrac{1}{2}\Delta+\theta F_m$ at $\gamma$.
Hence
\begin{align}
\|v_m\|_{L^1(D')} & = \left\langle \left| -\mathcal{R}_{\gamma}^{(m)}\left(\frac{\Delta}{2}+\theta F_m-\gamma\right) \phi \right| , \mathbbm{1}_{D'}\right\rangle_{L^2(D)} \nonumber\\
&\le \sqrt{|D'|} \left\|\mathcal{R}_{\gamma}^{(m)} \right\|_{L^2(D)\rightarrow L^2(D)} \left\|\left(\frac{\Delta}{2}+\theta F_m-\gamma\right)\phi \right\|_{L^2(D)} \nonumber\\
&\le \sqrt{|D'|} \frac{\gamma+c(M^2+\theta) \delta^{-2}}{\gamma-\lambda_{\max}(D,\theta F_m)}\sqrt{|D|}\label{e:proofl1res1}
\end{align}
by the bound \eqref{e:resbound} on the resolvent
and the pointwise bounds on $\phi,\Delta\phi$ and $V_{\cY}\phi$.
Noting now that, since $F_m\le V^{(a)}_{\cY}$, $\lambda_{\max}(D,\theta F_m)\le \lambda_{\max}(D,\theta V^{(a)}_{\cY})$ by Remark~\ref{r:evmonoton},
we obtain
\begin{equation}\label{e:proofl1res2}
\|u_m\|_{L^1(D')} \le \|v_m\|_{L^1(D')} + \|\phi\|_{L^1(D')} \le c \sqrt{|D'| |D|} \frac{\gamma+(M^2+\theta) \delta^{-2}}{\gamma-\lambda_{\max}(D,\theta V^{(a)}_{\cY})} + |D'|.
\end{equation}
Now \eqref{e:L1res} follows by monotone convergence since $F_m \uparrow V_\cY^{(a)}$ as $m \to \infty$.
\end{proof}

From the $L^1$-bound above we derive two pointwise estimates that will be useful in Section~\ref{s:expansions}.
\begin{Lem}\label{l:FK}
Fix $x\in D\backslash \cY$ and set $\eps_x = \frac12 \dist(x, \cY)$.
Assume that $0<a<\eps_x$ and $\gamma > \lambda_{\max}(D,\theta V^{(a)}_{\cY})$,
and let $c=c(d)$ be the constant from Lemma \ref{l:L1res}. Then
\begin{equation}\label{e:FKtau}
\E_x\left[\exp \int_0^{\tau_{D^{\cc}}}(\theta V^{(a)}_{\cY}(W_s)-\gamma) \dd s \right] \le 2+ c\sqrt{\frac{|D|}{|B_{\eps_x}|}} \frac{\gamma+ (M^2+\theta) \dist(D^{\cc}, \cY)^{-2}}{\gamma-\lambda_{\max}(D,\theta V^{(a)}_{\cY})}.
\end{equation}
Moreover, for all $t\in(0,\infty)$,
\begin{equation}\label{e:FKt}
\E_x\left[\mathbbm{1}_{\{\tau_{D^{\cc}} > t\}} \exp \int_0^t (\theta V^{(a)}_{\cY}(W_s)-\gamma) \dd s \right] 
\le 2 + \sqrt{\frac{|D|}{|B_{\eps_x}|}} \left(1 + c \frac{\gamma + (M^2+\theta) \dist(D^{\cc}, \cY)^{-2}}{\gamma - \lambda_{\max}(D,\theta V^{(a)}_{\cY})} \right).
\end{equation}
\end{Lem}

\begin{proof}
Fix $0 < r < \eps_x$ and abbreviate $I_s^t := \exp \int_s^t(\theta V^{(a)}_{\cY}(W_u) -\gamma )\dd u$. 
We begin with the proof of \eqref{e:FKtau}. 
Since $V^{(a)}_{\cY}\equiv 0$ on $B_{\eps_x}(x)$, using the strong Markov property we may write
\begin{equation}
\E_x\left[I_0^{\tau_{D^\cc}}\right] \le 1 + \E_x\left[\mathbbm{1}_{\{\tau_{B_r(x)^\cc} < \tau_{D^\cc}\}} I_{\tau_{B_r(x)^\cc}}^{\tau_{D^\cc}}\right]
\le 1 + \E_x\left[\E_{W_{\tau_{\partial B_r(x)}}}\left[I_0^{\tau_{D^\cc}}\right] \right].
\end{equation}
Since $W_{\tau_{\partial B_r(x)}}$ is uniformly distributed on the sphere $\partial B_r(x)$,
\begin{equation}\label{e:surface}
\E_x\left[I_0^{\tau_{D^{\cc}}}\right] \le  1 + \frac{1}{ \sigma_dr^{d-1}}\int_{\partial B_r(x)}\E_z\left[I_0^{\tau_{D^{\cc}}}\right]\sigma(\dd z),
\end{equation}
where $\sigma$ denotes surface measure on $\partial B_r(x)$ and $\sigma_d$ is the area of the $d$-dimensional unit sphere. Multiplying both sides of \eqref{e:surface} by $\sigma_dr^{d-1}$ and integrating over $r$ between $0$ and $\eps_x$ leads to
\begin{equation}\label{e:surface2}
|B_{\eps_x}| \left( \E_x\left[I_0^{\tau_{D^{\cc}}}\right] -1 \right) \le  \int_{B_{\eps_x}(x)}\E_z\left[I_0^{\tau_{D^{\cc}}}\right]\dd z.
\end{equation}
Now apply the $L^1$-bound from Lemma \ref{l:L1res} to the right-hand side with $D'=B_{\eps_x}(x)$, which gives
\begin{equation}\label{e:prlFK6}
\int_{B_{\eps_x}(x)} \E_z \left[I_0^{\tau_{D^{\cc}}}\right] \dd z
\le |B_{\eps_x}| \left\{ 1 + c\sqrt{\frac{|D|}{|B_{\eps_x}|}} 
\left(\frac{\gamma+ (M^2+\theta) \dist(D^{\cc}, \cY)^{-2}}{\gamma - \lambda_{\max}(D,\theta V^{(a)}_{\cY})} \right) \right\}.
\end{equation}
This yields \eqref{e:FKtau}, and we continue with the proof of \eqref{e:FKt}. 
Again, by the strong Markov property and since $V^{(a)}_{\cY}\equiv 0$ on $B_{\eps_x}(x)$,
\begin{align}\label{e:prlFK1}
\E_x \left[I_0^t \mathbbm{1}_{\{\tau_{D^\cc}> t\}} \right]
\le 1 +  \E_x \left[\ee^{-\gamma \tau_{\partial B_r(x)}} \mathbbm{1}_{\{\tau_{\partial B_r(x)} < t\}} \E_{W_{\tau_{\partial B_r(x)}}} \left[I_0^{t-s} \mathbbm{1}_{\{\tau_{D^\cc} > t-s \}} \right]_{s = \tau_{\partial B_r(x)}} \right].
\end{align}
Split the event $\{\tau_{D^{\cc}}>t-s\}$ according to whether $\tau_{D^{\cc}} > t$ or not to write, using $\gamma \ge 0$, $V^{(a)}_\cY \ge 0$,
\begin{align}\label{prlFK2}
I_0^{t-s} \mathbbm{1}_{\{\tau_{D^c}>t-s\}} 
= \ee^{s \gamma} \ee^{\int_0^{t-s} \theta V^{(a)}_{\cY}(W_s) \dd s - t \gamma} \mathbbm{1}_{\{\tau_{D^{\cc}} > t-s \}} \le \ee^{s \gamma} \left\{ I_0^t \mathbbm{1}_{\{\tau_{D^{\cc}} > t\}} + I_0^{\tau_{D^{\cc}}}\right\}.
\end{align}
Substituting this back into \eqref{e:prlFK1}, we obtain
\begin{align}\label{e:prlFK3}
\E_x \left[I_0^t \mathbbm{1}_{\{\tau_{D^{\cc}}> t\}} \right]
& \le 1 + \frac{1}{\sigma_dr^{d-1}} \int_{\partial B_r(x)} \E_z \left[I_0^t \mathbbm{1}_{\{\tau_{D^c}>t\}} +  I_0^{\tau_{D^{\cc}}} \right]\sigma(\dd z),
\end{align}
and the same calculation as between \eqref{e:surface}--\eqref{e:surface2} gives
\begin{align}\label{e:prlFK4}
|B_{\eps_x}| \left( \E_x \left[I_0^t \mathbbm{1}_{\{\tau_{D^{\cc}}> t\}} \right] - 1 \right) \le \int_{B_{\eps_x}(x)} \E_z \left[I_0^t \mathbbm{1}_{\{\tau_{D^c}>t\}}\right] \dd z + \int_{B_{\eps_x}(x)} \E_z \left[I_0^{\tau_{D^{\cc}}}\right] \dd z.
\end{align}
To conclude, apply \eqref{e:L1timedep} with $D' = B_{\eps_x}(x)$ to the first integral above,
and \eqref{e:prlFK6} to the second.
\end{proof}

\subsection{Lower bound}
\label{ss:LB}
%
We derive here an $L^1$ lower bound (cf.\ Lemma~\ref{l:keyLB} below) 
on the Feynman-Kac functional in \eqref{e:FKrep}
with $q = \theta V_\cY$, $\cY \in \scrYf$.
Recall $h_d = (d-2)^2/8$.
Define the truncated potential
\begin{equation}
\widetilde{V}(x):=\begin{cases} 1,  \mbox{\ if \ } |x|\le 1\\
|x|^{-2}, \mbox{\ else.}\end{cases}
\end{equation}

\begin{Lem}\label{l:EVlbVK}
For any $\eps > 0$, there exists $K_\eps \in [1,\infty)$ such that, for all $K \ge K_\eps$,
\begin{equation}\label{e:EVlbVK}
\sup_{g \in H^1_0(B_K), \|g\|_{L^2(B_K)} = 1} \left( h_d + \eps \right) \int_{B_K} g^2(x) \widetilde{V}(x) \dd x - \frac12\| \nabla g\|_{L^2(B_K)}^2 > 0.
\end{equation}
\end{Lem}

\begin{proof}
Taking, for $n \in \N$,
\begin{equation}\label{e:defgtilde}
\tilde{g}_n(x) := \left\{ 
\begin{array}{lcl}
1 & \text{ when } & |x| \le 1,\\
|x|^{-(d-2)/2} & \text{ when } & 1 < |x| \le n,\\
n^{-d/2} (2n - |x|) & \text{ when } & n < |x| \le 2n,\\
0 & \text{ when } & |x| > 2n,\\
\end{array}\right.
\end{equation}
it follows that, for all $K>2n$, $\tilde{g}_n \in H^1_0(B_K)$ and 
\begin{equation}
\left(h_d + \varepsilon \right) \frac{\int_{B_K} \tilde{g}_n^2(x) \widetilde{V}(x) \dd x}{\frac{1}{2} \int_{B_K}|\nabla \tilde{g}_n(x)|^2\dd x} \ge \left(1 + \frac{8 \varepsilon}{(d-2)^2}\right) \left(1-\frac{c}{\log n}\right)
\end{equation}
for some constant $c \in (0,\infty)$.
Letting $g_n := \tilde{g}_n / \|\tilde{g}_n \|_{L^2(B_K)}$, we obtain
\begin{align}\label{e:prpropEVlb3}
\left(h_d+\varepsilon\right) \int_{B_K} g^2_n(x) \widetilde{V}(x) \dd x - \frac12\| {\nabla} g_n \|_{L^2(B_K)}^2
& \ge \frac{2 \varepsilon}{(d-2)^2} \|{\nabla} g_n \|_{L^2(B_K)}^2 >0
\end{align}
for $n$ large and $K>2n$.
\end{proof}

Let $\cY \in \scrYf$ with $M = \# \cY\geq 2$ and fix $\theta\in (\frac{(d-2)^2}{8M},\frac{(d-2)^2}{8}]$.
We define

\begin{equation}\label{e:deltaM}
\delta_\star = \delta_\star(d, M,\theta) := \frac14 \left(1 - \frac{h_d}{\theta M} \right).
\end{equation}

\begin{Lem}\label{l:LBpotential}
If $|y| \le \delta_\star$ for all $y \in \mathcal{Y}$, then
\begin{equation}\label{e:LBpotential}
\theta V_\cY(x) \ge \left( h_d + 2 \theta M \delta_\star \right) \widetilde{V}(x) \quad \forall \, x \in \R^d \setminus \cY.
\end{equation}
\end{Lem}

\begin{proof}
Follows from a simple computation using $|x-y|^2\le|x|^2+2|x||y|+|y|^2$.
\end{proof}

The following is the key lemma to obtain a lower bound on the total mass.
\begin{Lem}\label{l:keyLB}
There exist constants $K > 1$ and $c_1, c_2 > 0$ depending on $d, M, \theta$ such that,
for any $a \in (0, \infty)$ and
any $x \in \R^d \setminus \cY$ such that $\cY \subset B_a(x)$,
\begin{equation}\label{e:keyLB}
\int_{B_{Ka}(x)} \E_z \left[\ee^{\int_0^t \theta V_\cY(W_s) \dd s} \mathbbm{1} \{\tau_{B_{Ka}(x)^\cc}>t \} \right] \dd z \ge c_1 a^{d}  \ee^{c_2 t a^{-2}} \quad \forall \; t \ge 0.
\end{equation}

\begin{proof}
By translation invariance, we may suppose that $x = 0$ and $\cY \subset B_a$.
Set $b = \delta_\star / a$, $K = K_\star / \delta_\star$, where $K_\star$ is given by Lemma~\ref{l:EVlbVK} with $\eps := 2 \theta M \delta_\star$, 
and write
\begin{equation}\label{e:prkeyLB0}
\int_{B_{Ka}} \E_z \left[\ee^{\int_0^t \theta V_\cY(W_s) \dd s} \mathbbm{1} \{\tau_{B^\cc_{Ka}} >t \} \right] \dd z
= b^{-d} \int_{B_{K_\star}} \E_{{z/b}} \left[\ee^{\int_0^t \theta V_\cY(W_s) \dd s} \mathbbm{1} \{\tau_{B^\cc_{Ka}}>t\} \right] \dd z.
\end{equation}
By Brownian scaling, the integrand in the right-hand side of \eqref{e:prkeyLB0} equals
\begin{equation}\label{e:prkeyLB1}
\E_z \left[\ee^{\int_0^t \theta V_\cY(b^{-1} W_{b^2 s}) \dd s} \mathbbm{1} \{\tau_{B^\cc_{K_\star}} > b^2 t\} \right] = \E_z \left[\ee^{\int_0^{b^2 t} \theta V_{b \cY}(W_{s}) \dd s} \mathbbm{1} \{\tau_{B^\cc_{K_\star}} > b^2 t\} \right]
\end{equation}
where $b \cY := \{by \colon\, y\in\cY\}$. Since $|y| \le \delta_\star$ for all $y \in b \cY$,
\eqref{e:prkeyLB1} is at least
\begin{equation}\label{e:prkeyLB2}
\E_z \left[\exp \left\{ \int_0^{b^2 t} \left(h_d + \eps \right) \widetilde{V}(W_{s}) \dd s \right\} \mathbbm{1}\{\tau_{B^\cc_{K_\star}} > b^2 t \} \right]
\end{equation}
by Lemma~\ref{l:EVlbVK}.
Now using a Fourier expansion as in \cite[Equation~(2.33)]{GKM00},
we obtain
\begin{equation}\label{e:prkeyLB3}
\int_{B_{K_\star}} \E_z \left[ \exp \left\{ \int_0^{b^2 t} \left(h_d + \eps \right) \widetilde{V}(W_{s}) \dd s \right\} \mathbbm{1} \{\tau_{B^\cc_{K_\star}} > b^2 t \} \right] \dd z \ge \ee^{b^2 t \, \widetilde{\lambda}_{\max}} \| e_1\|_{L^1(B_{K_\star})}^2
\end{equation}
where $\widetilde{\lambda}_{\max} := \lambda_{\max}(B_{K_\star},(h_d + \eps) \widetilde{V})$ is the principal Dirichlet eigenvalue of $\frac12\Delta + (h_d + \eps) \widetilde{V}$ in $B_{K_\star}$, 
and $e_1$ is the corresponding eigenfunction normalized so that $\|e_1\|_{L^2(B_{K_\star})} = 1$.
Now \eqref{e:keyLB} follows with $c_1 = \delta_\star^{-d} \| e_1 \|_{L^1(B_{K_\star})}^2$ and $c_2 = \delta_\star^2 \widetilde{\lambda}_{\max}$, which is strictly positive by Lemma \ref{l:EVlbVK}.
\end{proof}
\end{Lem}

\subsection{Multipolar Hardy inequality}
\label{ss:MPHI}

We provide in this section upper bounds for $\lambda_{\max}(\R^d, q)$ in \eqref{e:deflambdamax} 
with $q=\theta V_{\cY}$, $\cY \in \scrYf$ and $\theta\in(0,h_d]$ (recall $h_d = (d-2)^2/8$), 
which will be useful to control \eqref{e:FKtau} and \eqref{e:FKt}.

When $\# \cY = 1$, Hardy's inequality \eqref{e:Hardy1} states that
\begin{equation}\label{e:Hardy}
\lambda_{\max}(\R^d, \theta V_\cY) = 0 \quad \text{ if } \;\; 0 \le \theta \le h_d,
\end{equation}
which clearly extends to $\#\cY \ge 2$ in the sense that, with $M = \# \cY$,
\begin{equation}\label{e:MPHI_for_smalltheta}
\lambda_{\max}(\R^d, \theta V_\cY) = 0 \quad \text{ if } \;\;  0 \le \theta \le \tfrac{h_d}{M}.
\end{equation}
More general bounds, known as \emph{multipolar Hardy inequalities}, are considered for example in \cite{BDE08}.
The next proposition is obtained by combining results and methods from \cite{BDE08},
and offers in some cases an improvement of Theorem~1 therein.
\begin{proposition}
\label{p:MPHI}
Fix $\cY \in \scrYf$. Assume that $M := \# \cY \ge 2$ and $\theta \in \left(\tfrac{h_d}{M}, \tfrac{h_d}{(M-1)} \right]$.
Let
\begin{equation}\label{e:defdforMPHI}
\Gamma := \inf \left\{r > 0 \colon\, B_r(\cY) \text{ is connected} \right\}.
\end{equation}
Then
\begin{equation}\label{e:MPHI}
\lambda_{\max}(\R^d, \theta V_\cY) \le \frac{M(\pi^2+ 3 \theta)}{2 \Gamma^2}.
\end{equation}
\end{proposition}
\begin{proof}
Fix $r \in (0,\Gamma)$ and choose $\widehat{\cY} \subset \cY$ such that $\widehat{\cY} \neq \emptyset$, $N := \#\widehat{\cY} \le \lfloor M/2 \rfloor$,
\begin{equation}\label{e:prMPHI0}
B_{r}(\widehat{\cY}) \text{ is connected and } B_{r}(\widehat{\cY}) \cap  B_{r}(\cY\setminus \widehat{\cY}) = \emptyset,
\end{equation}
i.e., $B_r(\widehat{\cY})$ is a connected component of $B_r(\cY)$ containing at most half of the points of $\cY$.
Define a partition of unity (cf.\ definition after Theorem~1 in \cite{BDE08}) with $2$ terms as follows.
Set
\begin{equation}\label{e:defJ}
J(t) := \left\{ \begin{array}{llll}
0, & t \in [0,1/2],\\
-\cos(\pi t), & t \in [1/2, 1],\\
1, & t \ge 1,
\end{array}\right.
\end{equation}
put $J_1(x) := \prod_{y \in \widehat{\cY}} J(|x-y|/r)$ and $J_2(x) := [1 - J_1(x)^2 ]^{1/2}$.
By Lemma~2 in \cite{BDE08},
\begin{equation}\label{e:prMPHI1}
Q[u] := \int_{\R^d} \left\{ \theta V_\cY(x) u(x)^2 - |\nabla u(x)|^2 \right\} \dd x 
= \sum_{i=1}^{2} Q[J_i u] + \int_{\R^d} u(x)^2 \sum_{i=1}^{2} |\nabla J_i(x)|^2 \dd x
\end{equation}
for all $u \in H^1(\R^d)$. Note that, by \eqref{e:prMPHI0} and the definition of $J_1$, $J_2$,
\begin{equation}
V_\cY(x) J_2(x)^2 \le V_{\widehat{\cY}}(x) J_2(x)^2 + \frac{M-N}{r^2} \quad \forall\; x \in \R^d \setminus \widehat{\cY}
\end{equation}
while, for all $x \notin \widetilde{\cY}:= \cY \setminus \widehat{\cY}$,
\begin{align}\label{e:prMPHI2.5}
V_\cY(x) J_1(x)^2
& = \left\{ V_{\widetilde{\cY}}(x) + \sum_{y \in \widehat{\cY}} \frac{\mathbbm{1}_{\{|x-y| \ge  r/2 \}}}{|x-y|^2} \right\} J_1 (x)^2 \nonumber\\
& \le V_{\widetilde{\cY}}(x) J_1(x)^2 + \frac{N}{r^2} \sup_{t \ge 1/2} \frac{J(t)^2}{t^2} 
\le V_{\widetilde{\cY}}(x) J_1(x)^2 + \frac{2 N}{r^2},
\end{align}
where for the last step we used $\sup_{t \ge 1/2} J(t)^2/t^2 = \sup_{t \in [1/2, 1]} \cos(\pi t)^2/t^2 < 2$ (see the proof of Lemma~3 in \cite{BDE08}).
Applying \eqref{e:Hardy}--\eqref{e:MPHI_for_smalltheta}, we obtain
\begin{equation}\label{e:prMPHI3}
\sum_{i=1}^{2} Q[J_i u] \le  \theta\frac{M+N}{r^2} \|u\|^2_{L^2(\R^d)} \;\;\; \forall \; u \in H^1(\R^d).
\end{equation}
Next we claim that
\begin{equation}\label{e:prMPHI4}
\sum_{i=1}^{2} |\nabla J_i(x)|^2 \le N \frac{\pi^2}{r^2} \;\;\; \forall \; x \in \R^d.
\end{equation}
Indeed, we may restrict to $x \in B_r(\widehat{Y})$, in which case we note that
\begin{equation}\label{e:prMPHI4.1}
\sum_{i=1}^{2} |\nabla J_i(x)|^2 = \frac{|\nabla J_1(x)|^2}{1 - J_1(x)^2} \le \frac{\pi^2}{r^2} \sup_{\eta \in [0,\pi/2)^N} F(\eta),
\end{equation}
where, for $\eta = (\eta_1, \ldots, \eta_N) \in [0,\pi/2)^N$,
\begin{equation}\label{e:prMPHI4.2}
F(\eta) := \left(1 - \prod_{i=1}^N \sin(\eta_i)^2 \right)^{-1} \left(\sum_{i=1}^N \cos(\eta_i) \prod_{ j\neq i} \sin(\eta_j) \right)^2.
\end{equation}
Let us show that $\sup_{\eta \in [0,\pi/2)^N}F(\eta) \le N$. 
First note that, if $\min_i \eta_i = 0$, then $F(\eta) \le 1 < N$,
and thus we may restrict to $\eta \in (0,\pi/2)^N$.
In the latter set, $F = f /g$ where
\begin{equation}\label{e:prMPHI4.3}
f(\eta) := \left(\sum_{i=1}^N \cot(\eta_i)\right)^2, \quad g(\eta) := \prod_{i=1}^N \csc(\eta_i)^2 - 1.
\end{equation}
Using $\csc(\eta_i)^2 = 1 + \cot(\eta_i)^2$ and expanding the product in the definition of $g$, 
we obtain $g(\eta) \ge \sum_{i=1}^N \cot(\eta_i)^2$.
On the other hand, by the Cauchy-Schwarz inequality,
$
f(\eta) \le N \sum_{i=1}^N \cot(\eta_i)^2 \le N g(\eta),
$
finishing the proof of \eqref{e:prMPHI4}. As a consequence,
\begin{equation}\label{e:prMPHI5}
\int_{\R^d} u(x)^2 \sum_{i=1}^{2} |\nabla J_i(x)|^2 \dd x \le \frac{\lfloor M/2 \rfloor \pi^2}{r^2} \|u\|^2_{L^2(\R^d)} \;\;\; \forall \; u \in H^1(\R^d).
\end{equation}
Collecting now \eqref{e:prMPHI1}, \eqref{e:prMPHI3}, \eqref{e:prMPHI5} and letting $r \uparrow \Gamma$, 
we conclude \eqref{e:MPHI}.
\end{proof}

\section{Path expansions}
\label{s:expansions}
%
In this section, we provide an upper bound for the contribution to the Feynman-Kac
formula of Brownian paths that leave a large ball. This is achieved
by means of a path expansion technique that splits the Brownian path in
excursions between neighbourhoods of the Poisson points, cf.\ Section~\ref{ss:proofthmupperbound} below.

Recall $h_d = (d-2)^2/8$ and fix $\cY \in \scrYf$ (cf.\ \eqref{e:fin}).
Given $r>0$, we denote by $\mathscr{C}^{(r)}_\cY$ the set of connected components of $B_r(\cY)$.
For $a \in (0,r)$, $\theta \in (0, h_d]$ and $\cC \in \mathscr{C}^{(r)}_\cY$, let
\begin{equation}\label{e:deflambdaC}
N_\cC := \# \cY \cap \cC, \qquad \lambda_\cC := \lambda_{\max}(\cC, \theta V^{(a)}_{\cY}) = \lambda_{\max}(\cC, \theta V^{(a)}_{\cY \cap \cC}) \ge 0,
\end{equation}
where $V^{(a)}_\cY$ is as in \eqref{e:deffinpot} and $\lambda_{\max}(D, V)$ as in \eqref{e:deflambdamax}.
Note that $\lambda_\cC < \infty$ by \cite[Theorem~1]{BDE08}.
Define
\begin{equation}\label{e:numberpointsC}
N^{(r)}_\cY := \max_{\cC \in \mathscr{C}^{(r)}_\cY} N_\cC, \qquad \Lambda^{(\theta, a, r)}_\cY := \max_{\cC \in \mathscr{C}^{(r)}_\cY} \lambda_\cC
\end{equation}
and, for measurable $D' \subset \R^d$,
\begin{equation}\label{e:deffrakND'}
\mathfrak{N}^{(r)}_{\cY}(D') := 1 \vee \sqrt{\# \{\cC \in \mathscr{C}^{(r)}_\cY \colon\, D' \cap \cC \neq \emptyset\}}.
\end{equation}
The following is the main result of this section.
\begin{The}\label{t:upperbound}
There exist constants $K \in [1,\infty)$ and $c, c_* \in (0, \infty)$ such that the following holds.
Let $\cY \in \scrYf$, $\theta \in (0, h_d]$, $a \in (0, 1]$ and $r > 4a$.
For $\gamma >  \Lambda^{(\theta, a, r)}_\cY$, let
\begin{align}\label{e:defLvarrho}
\begin{aligned}
L & = L(\cY, \theta, a, r, \gamma) := K \Big(N_{\cY}^{(r)}\Big)^{5/2}\left(\frac{r}{a} \right)^{\tfrac{d}2} \left(1 +\frac{\gamma+ (1+\theta) r^{-2}}{\gamma - \Lambda^{(\theta, a, r)}_{\cY}} \right),\\
\varrho & = \varrho(\cY,\theta,a,r,\gamma) := L \exp \left\{-a c_* \sqrt{\gamma} \right\}.
\end{aligned}
\end{align}
Assume that $\varrho \leq 1/2$.
Then
\begin{equation}\label{e:upperbound0}
\sup_{z \in B_r(\cY)^\cc} \sup_{t \ge 0} \E_z \left[ \exp \int_0^t \{ \theta V^{(a)}_\cY(W_s) - \gamma\} \dd s \right] \le \frac{1}{1-\varrho} \le 2.
\end{equation}
and, for all measurable $D' \subset \R^d$, 
\begin{equation}\label{e:upperbound0int}
\sup_{t \ge 0}  \int_{D'} \E_z \left[ \exp \int_0^t \{ \theta V^{(a)}_\cY(W_s) - \gamma\} \dd s \right] \dd z \le  \frac{4 L \mathfrak{N}^{(r)}_{\cY}(D')}{1-\varrho}  \left( |D'| \vee \sqrt{|D'|} \right).
\end{equation}
Moreover, for all $R \ge 8 r N^{(r)}_\cY$ and all $t > 0$,
\begin{equation}\label{e:upperbound}
\sup_{z \in B_r(\cY)^\cc} \E_z \left[ \mathbbm{1}_{\{ \tau_{B_R^\cc(z) \le t}\}} \exp \int_0^t \{ \theta V^{(a)}_\cY(W_s) - \gamma\} \dd s\right] 
\le 2 K L \left\{ \frac{R}{r}\ee^{-\frac{ c R^2}{t}} + \varrho^{\frac{R}{4 r N^{(r)}_\cY }} \right\}
\end{equation}
and
\begin{equation}\label{e:upperboundint}
\begin{aligned}
\int_{D'} \E_z \left[ \mathbbm{1}_{\{ \tau_{B_R^\cc(z) \le t}\}} \ee^{\int_0^t \{ \theta V^{(a)}_\cY(W_s) - \gamma\} \dd s}\right] \dd z \le 4 K L \mathfrak{N}^{(r)}_{\cY}(D') \left(|D'| \vee \sqrt{|D'|} \right) \left\{ \frac{R}{r}\ee^{-\frac{ c R^2}{t}} + \varrho^{\frac{R}{4 r N^{(r)}_\cY }} \right\}.
\end{aligned}
\end{equation}
\end{The}

\subsection{Proof of Theorem~\ref{t:upperbound}}
\label{ss:proofthmupperbound}

We start with auxiliary results that will be needed in the following,
and that will allow us to identify the constants in Theorem~\ref{t:upperbound}.
The first lemma concerns standard bounds for Brownian motion.
\begin{Lem}\label{l:tailBM}
There exist $K_* = K_*(d) \in [1,\infty)$ and $c_*=c_*(d) \in (0,\infty)$ such that
\begin{equation}\label{e:tailBM}
\P_0 \left(\sup_{0 \le s \le t} |W_s| > R \right) \le K_* \ee^{-\frac{c_* R^2}{t}} \quad \text{ for all } t, R>0,
\end{equation}
and
\begin{equation}\label{e:exitBM}
\E_0 \left[ \ee^{- u \tau_{B_a^\cc}} \right] \le K_* \, \ee^{-c_* a \sqrt{u}} \quad \text{ for all } a, u > 0.
\end{equation}
\end{Lem}
\begin{proof}
Follows from union bounds and standard estimates for one-dimensional Brownian motion,
e.g.\ Remark~2.22 and Exercise~2.18 in \cite{MP10}.
\end{proof}

The next lemma is a consequence of the bounds in Lemma~\ref{l:L1res} and Lemma\ref{l:FK}.
\begin{Lem}\label{l:resolvcomp}
There exists a constant $K_1 \in [1,\infty)$ such that,
for all $\cY \in \scrYf$,
$\theta \in (0,h_d]$,
$a \in (0, 1]$, $r>2a$, $\cC \in \mathscr{C}^{(r)}_\cY$, 
$\gamma > \lambda_{\cC}$,
and all measurable $D' \subset \R^d$,
\begin{equation}\label{e:intcomptime}
\sup_{t \geq 0} \int_{D'} \E_x \left[\ee^{\int_0^t \left( \theta V^{(a)}_\cY(W_s) - \gamma \right) \dd s} \1{\{ \tau_{\cC^\cc} > t \}} \right] \dd x
\leq  \sqrt{|D'|} K_1 N_\cC^{1/2} \left(\frac{r}{a}\right)^{\tfrac{d}{2}}  
\end{equation}
and
\begin{equation}\label{e:intcomp}
\int_{D'} \E_x \left[\ee^{\int_0^{\tau_{\cC^{\cc}}} \left( \theta V^{(a)}_\cY(W_s) - \gamma \right) \dd s} \right] \dd x
\leq \left( |D'| \vee \sqrt{|D'|} \right) K_1 N_\cC^{5/2} \left(\frac{r}{a} \right)^{\tfrac{d}{2}}\left(1 + \frac{\gamma + (1+\theta)r^{-2}}{ \gamma - \lambda_\cC} \right).
\end{equation}
Moreover, 
for all $x \in \cC \setminus B_{2a}(\cY)$,
\begin{equation}\label{e:resolvcomp}
\E_x \left[\exp{\int_0^{\tau_{\cC^{\cc}}} \left( \theta V^{(a)}_\cY(W_s) - \gamma \right) \dd s} \right] \le 
K_1 N_{\cC}^{5/2}\left(\frac{r}{a} \right)^{\tfrac{d}2} \left(1 + \frac{\gamma + (1+\theta) r^{-2}}{\gamma - \lambda_{\cC}} \right)
\end{equation}
and
\begin{equation}\label{e:semigroupbb}
\sup_{ t \ge 0}\E_x \left[ \ee^{\int_0^t \left(\theta V^{(a)}_\cY(W_s) - \gamma \right) \dd s} \mathbbm{1}_{\{\tau_{\cC^\cc} > t \}}\right] 
\le K_1 N_{\cC}^{5/2} \left(\frac{r}{a} \right)^{\tfrac{d}2} \left( 1 + \frac{\gamma + (1+\theta) r^{-2}}{\gamma - \lambda_{\cC}} \right).
\end{equation}
\end{Lem}
\begin{proof}
By \cite[Proposition~1.22]{CZ95}, each $\cC \in \mathscr{C}^{(r)}_\cY$ is a bounded regular domain.
Noting that $V^{(a)}_\cY(x) = V^{(a)}_{\cY \cap \cC}(x)$ for $x \in \cC \setminus \cY$,
apply Lemmas~\ref{l:L1res}--\ref{l:FK} with $D = \cC$ and use $|\cC| \le |B_{1}| N_\cC r^d$, $N_\cC \ge 1$.
\end{proof}

\begin{Cor}
\label{cor:basicestim}
For any $\cY \in \scrYf$, 
$\theta \in (0,h_d]$,
$a \in (0, \infty)$, $r > 4a$, $\cC \in \mathscr{C}^{(r)}_\cY$, $\gamma > \lambda_\cC$
and $x \in \cC \cap B_{r-a}(\cY) \setminus B_{3a}(\cY)$,
\begin{equation}\label{e:basicestim}
\E_x \left[\exp{\int_0^{\tau_{\cC^{\cc}}} \left( \theta V^{(a)}_\cY(W_s) - \gamma \right) \dd s} \right] \le 
K_* K_1 N_{\cC}^{5/2} \left(\frac{r}{a} \right)^{\tfrac{d}2} \ee^{- c_* a \sqrt{\gamma}} \left(1 +  \frac{\gamma + (1+ \theta) r^{-2}}{\gamma - \lambda_{\cC}}\right)
\end{equation}
where $K_*$, $c_*$ are as in Lemma~\ref{l:tailBM} and $K_1$ as in Lemma~\ref{l:resolvcomp}.
\end{Cor}
\begin{proof}
Use the strong Markov property at the exit time of $B_a(x)$ and apply Lemma~\ref{l:resolvcomp} and \eqref{e:exitBM}.
\end{proof}

With these results in place, we may identify the constants $K, c$ in Theorem~\ref{t:upperbound} as
\begin{equation}\label{e:defK}
K := 2 (K_*)^2 K_1, \qquad c:= \frac{c_*}{16},
\end{equation}
where $K_*, c_*$ are as in Lemma~\ref{l:tailBM} and $K_1$ as in Lemma~\ref{l:resolvcomp}.

Fix now $\cY \in \scrYf$, $\theta \in (0, h_d]$, $a \in (0,1]$, $r > 4a$ and $\gamma > \Lambda^{(\theta, a, r)}_\cY$.
In the following, we fix $K,c$ as in \eqref{e:defK}
and let $L, \varrho$ be defined by \eqref{e:defLvarrho}.

The core of the proof of Theorem~\ref{t:upperbound} is a decomposition of the Brownian path
according to its excursions to and from neighbourhoods of $\cY$,
which are marked by the following stopping times.
Let $\check{\tau}_0 = \hat{\tau}_0 :=0$ and, recursively for $n \ge 0$,
\begin{align}\label{e:checktauhattauarestoptimes}
\begin{aligned}
\check{\tau}_{n+1} & : = \left\{ 
\begin{array}{ll}
\infty & \text{ if } \hat{\tau}_n = \infty,\\
\inf \left\{t > \hat{\tau}_n \colon\, W_t \in \overline{B_{3a}(\cY)} \right\} & \text{ otherwise,}
\end{array}
\right. \\
\hat{\tau}_{n+1} & = \left\{  
\begin{array}{ll}
\infty & \text{ if } \check{\tau}_{n+1} = \infty,\\
\inf \left\{t > \check{\tau}_{n+1} \colon\, W_t \notin B_r(\cY) \right\} & \text{ otherwise.} 
\end{array}
\right.
\end{aligned}
\end{align}
For $t  \ge 0$, define
\begin{equation}\label{e:defEt}
E_t := \inf \{ n \ge 0 \colon\, \check{\tau}_{n+1} > t\}.
\end{equation}
In the following we will abbreviate, for $0\le s_1 \le s_2 \le \infty$,
\begin{equation}\label{e:defI}
I_{s_1}^{s_2} := \exp \left\{\int_{s_1}^{s_2} \left( \theta V^{(a)}_\cY(W_s) - \gamma \right) \dd s \right\}.
\end{equation}

\begin{Lem}\label{l:firstinduction}
For all $n \in \N_0$,
\begin{align}\label{e:firstinduction}
\sup_{x \notin B_r(\cY)} \sup_{t \ge 0} \E_x \left[I_0^t \mathbbm{1}_{\{E_t = n\}} \right] \le  \varrho^{n}.
\end{align}
Moreover, for all measurable $D' \subset \R^d$,
\begin{align}\label{e:firstinductionint}
\sup_{t \ge 0} \int_{D'} \E_x \left[I_0^t \mathbbm{1}_{\{E_t = n\}} \right] \dd z \le  2 \mathfrak{N}^{(r)}_{\cY}(D') \left(|D'| \vee \sqrt{|D'|} \right) L \varrho^{(n-1)^+}.
\end{align}
\end{Lem}

\begin{proof}
Note first that, if $n=0$, both \eqref{e:firstinduction} and \eqref{e:firstinductionint} hold since then $V^{(a)}_\cY(W_s)=0$ for all $0 \le s \le t$.
Let us prove \eqref{e:firstinduction} by induction in $n$.
To treat the case $n=1$, fix $x \notin B_r(\cY)$ and $t > 0$.
There are two cases: either $\check{\tau}_1 \le t < \hat{\tau}_1$,
or $\hat{\tau}_1 \le t < \check{\tau}_2$.
Let $\check{\cC}_1 \in \mathscr{C}^{(r)}_\cY$ such that $W_{\check{\tau}_1} \in \check{\cC}_1$.
Using the Markov property, $\gamma > \Lambda^{(\theta, a, r)}_\cY$ and Lemma~\ref{l:resolvcomp},
we may bound, $\P_x$-a.s.\ on the event $\{\check{\tau}_1 \le t \}$,
\begin{align}\label{e:prfirstind1}
\E_x \left[ I_{\check{\tau}_1}^t \mathbbm{1}_{ \{\check{\tau}_1 \le t < \hat{\tau}_1 \} } \,\middle|\, \check{\tau}_1, (W_s)_{s\le \check{\tau}_1} \right] 
& = \E_{W_{\check{\tau}_1}} \left [I_0^{t-s} \mathbbm{1}_{\{\tau_{\cC^\cc}> t-s\}} \right]_{s = \check{\tau}_1, \cC = \check{\cC}_1}  \nonumber\\
& \le L/(2K_*^2) \le L/(2 K_*)
\end{align}
and, using that $V^{(a)}_\cY(W_s) = 0$ for all $s \in [\hat{\tau}_1,t]$ when $\hat{\tau}_1 \le t < \check{\tau}_2$ and Corollary~\ref{cor:basicestim},
\begin{align}\label{e:prfirstind2}
\E_x \left[ I_{\check{\tau}_1}^t \mathbbm{1}_{ \{\hat{\tau}_1 \le t < \check{\tau}_2 \} } \,\middle|\, \check{\tau}_1, (W_s)_{s\le \check{\tau}_1} \right]
& \le \E_x \left[ I_{\check{\tau}_1}^{\hat{\tau}_1} \,\middle|\, \check{\tau}_1, (W_s)_{s\le \check{\tau}_1} \right] \nonumber\\
& = \E_{W_{\check{\tau}_1}} \left[I_0^{\tau_{\cC^\cc}} \right]_{\cC = \check{\cC}_1} \nonumber\\
& \le \varrho/(2K_*) <  L/(2K_*).
\end{align}
Since $r>4a$ and $x \notin B_r(\cY)$, $\check{\tau}_1 \ge \tau_{B_a^\cc(x)}$ and thus
\begin{equation}\label{e:prfirstind3}
\E_x \left[ I_0^{\check{\tau}_1} \mathbbm{1}_{\{\check{\tau}_1 \le t\}} \right] \le \E_0 \left[ \ee^{- \gamma \tau_{B_a^\cc}}\right] \le K_* \ee^{-c_* a \sqrt{\gamma}}
\end{equation}
by Lemma~\ref{l:tailBM}.
This together with \eqref{e:prfirstind1}--\eqref{e:prfirstind2} gives
\begin{align}\label{e:prfirstind4}
\E_x \left[I_0^t \mathbbm{1}_{\{E_t = 1\}} \right]
& = \E_x \left[ I_0^{\check{\tau}_1} \mathbbm{1}_{\{\check{\tau}_1 \le t\}} \E_x \left[ I_{\check{\tau}_1}^t \mathbbm{1}_{\{E_t = 1\}} \,\middle|\, \check{\tau}_1, (W_s)_{s \le \check{\tau}_1} \right] \right] \nonumber\\
& \le L \ee^{-c_* a \sqrt{\gamma}} = \varrho
\end{align}
by \eqref{e:defLvarrho}, concluding the case $n=1$.
Suppose now by induction that \eqref{e:firstinduction} has been shown for some $n \ge 1$.
If $E_t = n+1$, then $\hat{\tau}_1 \le t$ and we can write
\begin{align}\label{e:prfirstind5}
\E_x \left[I_0^t \mathbbm{1}_{\{ E_t = n+1 \}} \right]
& = \E_x \left[I_0^{\hat{\tau}_1} \mathbbm{1}_{\{\hat{\tau}_1 \le t\}}
\E_{W_{\hat{\tau}_1}} \left[I_0^{t-s} \mathbbm{1}_{\{E_{t-s} = n\}} \right]_{s = \hat{\tau}_1} \right] 
\leq \varrho^n \E_x \left[I_0^{\hat{\tau}_1} \mathbbm{1}_{\{ \hat{\tau}_1 \le t\}} \right]
\nonumber\\
& \le \varrho^n \E_x \left[I_{\check{\tau}_1}^{\hat{\tau}_1} \mathbbm{1}_{\{ \hat{\tau}_1 \le t\}} \right] \le \varrho^{n+1}/(2 K_*)
\end{align}
by the induction hypothesis, \eqref{e:prfirstind2} and \eqref{e:defLvarrho}.
This concludes the proof of \eqref{e:firstinduction}.

We turn next to \eqref{e:firstinductionint}. Assume first that $n=1$. 
There are again two cases: either $\check{\tau}_1 \le t < \hat{\tau}_1$,
or $\hat{\tau}_1 \le t < \check{\tau}_2$.
In the first case,
\begin{align}\label{e:prfirstind6}
\int_{D' \cap B_r(\cY)} \E_x \left[I_0^t \1{\{E_t =1, \check{\tau}_1 \le t < \hat{\tau}_1 \}} \right] \dd x
& = \sum_{\cC \in \mathscr{C}^{(r)}_\cY} \int_{D' \cap \cC} \E_x \left[I_0^t \1{\{\tau_{\cC^\cc} > t\}}\right] \dd x \nonumber\\
& \leq \sqrt{\# \{\cC \in \mathscr{C}^{(r)}_\cY \colon D \cap \cC \neq \emptyset \}} \sqrt{|D' \cap B_r(\cY)|}L/(2 K_*^2)
\end{align}
by \eqref{e:intcomptime} and the Cauchy-Schwarz inequality. In the second case,
\begin{align}\label{e:prfirstind7}
& \int_{D' \cap B_r(\cY)} \E_x \left[I_0^t \1{\{E_t =1, \hat{\tau}_1 \le t < \check{\tau}_2 \}} \right] \dd x
 \leq  \int_{D' \cap B_r(\cY)} \E_x \left[I_0^{\hat{\tau}_1} \right] \dd x 
 = \sum_{\cC \in \mathscr{C}^{(r)}_\cY} \int_{D' \cap \cC } \E_x \left[I_0^{\tau_{\cC^\cc}} \right] \dd x \nonumber\\
& \leq \left(|D'\cap B_r(\cY)| + \sqrt{\#\{\cC \in \mathscr{C}^{(r)}_\cY \colon \cC \cap D'\neq \emptyset\}} \sqrt{|D'\cap B_r(\cY)|} \right) L/(2 K_*^2)
\end{align}
by \eqref{e:intcomp}. 
Combining \eqref{e:prfirstind6}, \eqref{e:prfirstind7} and \eqref{e:firstinduction} we get
\begin{align}\label{e:prfirstind8}
\int_{D'} \E_x \left[I_0^t \1{\{ E_t = 1 \}} \right] \dd x
& \leq |D'\setminus B_r(\cY)| \varrho + \left(|D'\cap B_r(\cY)| + 2 \mathfrak{N}^{(r)}_\cY(D') \sqrt{|D' \cap B_r(\cY)|} \right) L / 2 \nonumber\\
& \leq 2 \mathfrak{N}^{(r)}_\cY(D') \left(|D'| \vee \sqrt{|D'|} \right) L.
\end{align}
This concludes the case $n=1$.
To deal with $n+1$, $n\geq 1$,
note that the first line of \eqref{e:prfirstind5} is valid for any $x \in D'$.
Then we may write
\begin{align}\label{e:prfirstind10}
\int_{D'} \E_x \left[I_0^t \mathbbm{1}_{\{ E_t = n+1 \}} \right] \dd x
& \leq \varrho^{n+1} |D'\setminus B_r(\cY)| + \varrho^n \sum_{\cC \in \mathscr{C}^{(r)}_\cY} \int_{D'\cap \cC} \E_x \left[I_0^{\tau_{\cC^\cc}} \right] \dd x \nonumber\\
& \leq 2 \mathfrak{N}^{(r)}_\cY(D') \left( |D'| \vee \sqrt{|D'|} \right) L \varrho^n.
\end{align}
This finishes the proof \eqref{e:firstinductionint}.
\end{proof}

The next result is the key lemma for the proof of Theorem~\ref{t:upperbound}.
\begin{Lem}\label{l:secondinduction}
For all $R>0$, $n \in \N_0$, and $z \in \R^d$,
\begin{align}\label{e:secondinduction}
\E_x \left[I_0^t \mathbbm{1}_{\{E_t = n, \tau_{B_R^\cc(z)} \le t \}} \right]
\le  2 L \varrho^{(n-1)^+} \, \P_x\left( \sup_{0 \le s \le t} |W_s-z| > R - 2 N^{(r)}_\cY n r \right) \;\,\forall \, x \notin B_r(\cY), t \ge 0.
\end{align}
If moreover $R>2r N_\cY$, then, for any measurable $D'\subset \R^d$,
\begin{align}\label{e:secondinductionint}
& \sup_{t \geq 0} \int_{D'} \E_x \left[I_0^t \mathbbm{1}_{\{E_t = n, \tau_{B_R^\cc(x)} \le t \}} \right] \dd x \nonumber\\
& \le 4 L \mathfrak{N}^{(r)}_\cY(D') \left( |D'| \vee \sqrt{|D'|} \right)  \varrho^{(n-1)^+} \, \P_0 \left( \sup_{0 \le s \le t} |W_s| > R - 2 N^{(r)}_\cY n r \right).
\end{align}
\end{Lem}
\begin{proof}
Fix $z \in \R^d$ and $R>0$. 
Note that, when $n=0$, both inequalities hold since then $V^{(a)}_\cY(W_s) = 0$ for all $0 \le s \le t$.
Let us prove \eqref{e:secondinduction} by induction in $n$.
Define the events
\begin{equation}\label{e:prsecind0}
\cE^n_u(z) := \left\{\sup_{0 \le s \le u}|W_s -z| \ge R - 2 n r N^{(r)}_\cY \right\}, \;\;\; n \in \N_0, u \ge 0.
\end{equation}
For the case $n=1$, fix $x \notin B_r(\cY)$ and $t > 0$.
Consider first the case $\tau_{B_R^\cc(z)} \le \hat{\tau}_1$.
We claim that, on this event,
$\cE^1_{\check{\tau}_1}(z)$ occurs.
Indeed, if $\tau_{B_R^\cc(z)} \le \check{\tau}_1$ this is clear,
and if $\check{\tau}_1 < \tau_{B_R^\cc(z)} \le \hat{\tau}_1$ then
$|W_{\check{\tau}_1} - z | > R - 2 r N_\cY$ as the diameter of any component $\cC \in \mathscr{C}^{(r)}_\cY$
is bounded by $2 r N^{(r)}_\cY $.
Thus
\begin{align}\label{e:prsecind1}
\E_x \left[I_0^t \mathbbm{1}_{\{\tau_{B_R^\cc(z)} \le \hat{\tau}_1, E_t = 1\}} \right]
& \le \E_x \left[\mathbbm{1}_{\cE^1_{\check{\tau}_1}(z)\cap \{ \check{\tau}_1 \le t \}} 
\E_x \left[I_{\check{\tau}_1}^t \mathbbm{1}_{\{E_t = 1\}} \,\middle|\, \check{\tau}_1, (W_s)_{s \le \check{\tau}_1}\right] \right] \nonumber\\
& \le L  \P_x \left( \cE^1_t(z) \right)
\end{align}
by \eqref{e:prfirstind1}--\eqref{e:prfirstind2} above.
If $t \ge \tau_{B_R^\cc(z)} > \hat{\tau}_1$, then $\hat{\tau}_1 < t < \check{\tau}_2$,
and thus
\begin{equation}\label{e:prsecind2}
\E_x \left[I_0^t \mathbbm{1}_{\{\hat{\tau}_1 < \tau_{B_R^\cc(z)} \le t, E_t =1 \}} \right]
\le \E_x \left[ I_0^{\hat{\tau}_1} \mathbbm{1}_{\{\hat{\tau}_1 \le t\}} \P_{W_{\hat{\tau}_1}} \left( \cE^{0}_{t-s}(z) \right)_{s = \hat{\tau}_1} \right].
\end{equation}
Note now that, since $\check{\tau}_1 \le \hat{\tau}_1$ and $|W_{\hat{\tau}_1}-W_{\check{\tau}_1}| \le 2 r  N_\cY$,
\begin{equation}\label{e:prsecind3}
\P_{W_{\hat{\tau}_1}} \left( \cE^0_{t-s}(z) \right)_{s=\hat{\tau}_1}
\le \P_{W_{\check{\tau}_1}} \left( \cE^1_{t-s}(z) \right)_{s=\check{\tau}_1}
\end{equation}
and thus \eqref{e:prsecind2} is at most
\begin{align}\label{e:prsecind4}
& \E_x \left[ \mathbbm{1}_{\{\check{\tau}_1 \le t\}} \P_{W_{\check{\tau}_1}} \left( \cE^1_{t-s}(z) \right)_{s = \check{\tau}_1} \E_{W_{\check{\tau}_1}} \left[I_0^{\tau_{\cC^\cc}} \right]_{\cC = \check{\cC}_1} \right] \le \frac{\varrho}{2} \, \P_x \left( \cE^1_t(z) \right)
< L \, \P_x \left( \cE^1_t(z) \right)
\end{align}
by Corollary~\ref{cor:basicestim} and \eqref{e:defLvarrho}. Collecting \eqref{e:prsecind1}--\eqref{e:prsecind4}, we conclude the case $n=1$.

Assume now by induction that \eqref{e:secondinduction} holds for some $n \ge 1$.
There are two possible cases:
either $\tau_{B_R^\cc(z)} \le \hat{\tau}_{1}$ or not.
In the first case, we conclude as before that $\cE^1_{\check{\tau}_1}(z)$ occurs.
Then we may write
\begin{align}\label{e:prsecind5}
\E_x \left[ I_0^t \mathbbm{1}_{\{E_t = n+1, \tau_{B_R^\cc(z)} \le \hat{\tau}_1 \}}\right]
& \le \E_x \left[ I_0^{\hat{\tau}_1} \mathbbm{1}_{\cE^1_{\check{\tau}_1}(z) \cap \{\hat{\tau}_1 \le t \}} 
\E_{W_{\hat{\tau}_1}} \left[ I_0^{t-s} \mathbbm{1}_{\{E_{t-s} = n\}} \right]_{s = \hat{\tau}_1} \right] \nonumber\\
& \le \varrho^n \E_x \left[ \mathbbm{1}_{\cE^1_{\check{\tau}_1}(z) \cap \{\check{\tau}_1 \le t \}} \E_{W_{\check{\tau}_1}} \left[  I_{0}^{\tau_{\cC^\cc}} \right]_{\cC = \check{\cC}_1} \right] \nonumber\\
& \le \varrho^{n} \frac{\varrho}{2} \, \P_x \left( \cE^1_t(z) \right)
< L \varrho^n \, \P_x \left( \cE^1_t(z) \right)
\end{align}
by Lemma~\ref{l:firstinduction}, Corollary~\ref{cor:basicestim} and \eqref{e:defLvarrho}.
Consider now the case $\hat{\tau}_1 < \tau_{B_R^\cc(z)}$ and write
\begin{align}\label{e:prsecind6}
\E_x \left[ I_0^t \mathbbm{1}_{\{E_t = n+1, \hat{\tau}_1 < \tau_{B_R^\cc(z)} \le t \}}\right]
& = \E_x \left[I_0^{\hat{\tau}_1} \mathbbm{1}_{\{\hat{\tau}_1 \le t\}} \E_{W_{\hat{\tau}_1}} \left[I_0^{t-s} \mathbbm{1}_{\{E_{t-s} = n, \tau_{B_R^\cc(z)} \le t-s\}} \right]_{s = \hat{\tau}_1} \right] \nonumber\\
& \le 2 L \varrho^{n-1} \E_x \left[I_0^{\hat{\tau}_1} \mathbbm{1}_{\{\hat{\tau}_1 \le t\}} \P_{W_{\hat{\tau}_1}}
\left( \cE^n_{t-s}(z) \right)_{s = \hat{\tau}_1} \right]
\end{align}
by the induction hypothesis. 
Reasoning as for \eqref{e:prsecind3}, we see that 
\[
\P_{W_{\hat{\tau}_1}} \left(\cE^n_{t-s}(z) \right)_{s=\hat{\tau}_1} \le \P_{W_{\check{\tau}_1}} \left(\cE^{n+1}_{t-s}(z)\right)_{s=\check{\tau}_1},
\]
and hence \eqref{e:prsecind6} is at most
\begin{align}\label{e:prsecind8}
2 L \varrho^{n-1} \E_x \left[ \mathbbm{1}_{\{\check{\tau}_1 \le t\}} 
\P_{W_{\check{\tau}_1}} \left( \cE^{n+1}_{t-s}(z) \right)_{s = \check{\tau}_1} \E_{\check{\tau}_1} \left[ I_0^{\tau_{\cC^\cc}} \right]_{\cC = \check{\cC}_1} \right]
\le  2 L \varrho^{n-1} (\varrho/2) \P_x \left( \cE^{n+1}_t(z) \right)
\end{align}
by Corollary~\ref{cor:basicestim}.
Combining \eqref{e:prsecind5} and \eqref{e:prsecind8} we conclude the induction step 
and the proof of \eqref{e:secondinduction}. 

Fix now a measurable $D' \subset \R^d$ and consider \eqref{e:secondinductionint}.
Let $n=1$ and $x \in D' \cap B_r(\cY)$. Since $R> 2 r N_\cY $, 
we must have $\hat{\tau}_1 < \tau_{B^\cc_R(D')}$ $\P_x$-almost surely, and if $E_t = 1$ then $\hat{\tau_1} < t < \check{\tau}_2$ as well.
Note that \eqref{e:prsecind2} still holds in this case (with $z=x$).
Since $|W_{\hat{\tau}_1}- x| < 2 r N_\cY$ $\P_x$-a.s.,
\[
\P_{W_{\hat{\tau_1}}} \left( \cE^{0}_{t-s}(x) \right) \leq \P_x \left(\cE^{1}_t(x) \right) = \P_0 \left( \cE_t^1(0) \right)
\]
and thus
\begin{align}\label{e:prsecind9}
&  \int_{D' \cap B_r(\cY)} \E_x \left[ I_0^t \1{\{E_t = 1, \tau_{B_R^\cc(D')} \leq t \}} \right] \dd x
\leq \P_0 \left( \cE^1_t(0) \right) 
\sum_{\cC \in \mathscr{C}^{(r)}_\cY } \int_{D'\cap \cC} \E_x \left[I_0^{\tau_{\cC^\cc}}\right]\dd x \nonumber\\
& \leq  
\P_0 \left( \cE^1_t(0) \right) \left( |D'\cap B_r(\cY)| + \sqrt{\#\{\cC \in \mathscr{C}^{(r)}_\cY \colon D' \cap \cC \neq \emptyset\}} \sqrt{|D' \cap B_r(\cY)|}\right) L/(2K_*^2)
\end{align}
by \eqref{e:intcomp} and the Cauchy-Schwarz inequality.
Now \eqref{e:prsecind9} and \eqref{e:secondinduction} imply
\begin{align}\label{e:prsecind10}
& \int_{D'} \E_x \left[I_0^t \1{\{E_t = 1, \tau_{B_R^\cc(D')} \leq t \}} \right] \dd x \nonumber\\
& \leq 2 L \P_0 \left( \cE_t^1(0) \right) 
\left\{ |D' \setminus B_r(\cY)| + |D'\cap B_r(\cY)|/4 +  \mathfrak{N}^{(r)}_\cY(D') \sqrt{|D' \cap B_r(\cY)|}/4
\right\} \nonumber\\
& \leq 4 L \P_0 \left( \cE_t^1(0) \right)  \mathfrak{N}^{(r)}_\cY(D') \left( |D'| \vee \sqrt{|D'|}\right),
\end{align}
finishing the proof of the case $n=1$.
The general case is analogous, using \eqref{e:prsecind6} instead of \eqref{e:prsecind2}.
\end{proof}

We are now ready to finish the:
\begin{proof}[Proof of Theorem~\ref{t:upperbound}]
Items \eqref{e:upperbound0}--\eqref{e:upperbound0int} follow from Lemma~\ref{l:firstinduction}.
To show \eqref{e:upperbound}, fix $z \in B_r(\cY)^\cc$ and write
\begin{align}\label{e:prthmuppderbound1}
\E_z \left[ I_0^t \mathbbm{1}_{\{\tau_{B_R^\cc(z)} \le t \}}\right] 
& = \sum_{n = 0}^\infty \E_z \left[ I_0^t \mathbbm{1}_{\{ \tau_{B_R^\cc(z)} \le t, E_t = n \}}\right] \nonumber\\
& \le 2 L \sum_{n=0}^\infty \varrho^{(n-1)^+} \P_0 \left(\sup_{0 \le s \le t} |W_s| \ge R - 2 N^{(r)}_\cY n r \right)
\end{align}
by Lemma~\ref{l:secondinduction} and the translation invariance of Brownian motion.
Split the sum in \eqref{e:prthmuppderbound1} according to whether
$ 4 N^{(r)}_\cY (n -1 ) r > R  $ or not to obtain
\begin{align}\label{e:prthmuppderbound2}
\frac{1}{2L} \E_z \left[ I_0^t \mathbbm{1}_{\{\tau_{B_R^\cc(z)} \le t \}}\right] 
& \le 2 \varrho^{\frac{R}{4 r N^{(r)}_\cY}} + \left(\frac{R}{4  r N^{(r)}_\cY  }+2 \right) \P_0 \left( \sup_{0 \le s \le t} |W_s| \ge \tfrac14 R \right) \nonumber\\
& \le K \left\{ \varrho^{\frac{R}{4 r N_\cY}} + \frac{R}{r} \ee^{-\frac{c R^2}{t} }\right\}
\end{align}
using $\varrho \le 1/2$, $R \ge 8 r N^{(r)}_\cY$, Lemma~\ref{l:tailBM} and \eqref{e:defK}.
This concludes the proof of \eqref{e:upperbound}.

For \eqref{e:upperboundint}, we have instead
\[
\begin{aligned}
\int_{D'} \E_z \left[ I_0^t \mathbbm{1}_{\{\tau_{B_R^\cc(z)} \le t \}}\right] \dd z
& = \sum_{n = 0}^\infty \int_{D'} \E_z \left[ I_0^t \mathbbm{1}_{\{ \tau_{B_R^\cc(z)} \le t, E_t = n \}}\right] \dd z\nonumber\\
& \le 4 L \mathfrak{N}^{(r)}_\cY(D') \left( |D'| \vee \sqrt{|D'|} \right) \sum_{n=0}^\infty \varrho^{(n-1)^+} \P_0 \left(\sup_{0 \le s \le t} |W_s| \ge R - 2 N^{(r)}_\cY n r \right)
\end{aligned}
\]
so we may conclude as before.
\end{proof}

\section{Small distances in Poisson clouds}
\label{s:SDPC}

We collect some elementary facts
concerning the probability to find Poisson points close to each other.
With the help of Proposition \ref{p:MPHI}, this will allow us to control in Section~\ref{ss:taileigenvalue} the growth of the 
maximal principal eigenvalue $\Lambda^{(\theta, a, r)}_\cY$ appearing in Theorem \ref{t:upperbound}
with $\cY = \cP \cap B_R$.

\begin{Lem}\label{l:probupp}
For any measurable $D \subset \R^d$, any $r \in (0,\infty)$ and any $k \in \N_0$,
\begin{equation}\label{e:PPPgeom1}
\Prob \left( \exists \text{ distinct } y_0, \ldots, y_k \in \cP \colon\, y_0 \in D, \max_{1 \le i \le k} |y_i - y_{i-1} | \le r \right) \le |D|\frac{|B_r|^k}{(k+1)!}.
\end{equation}
Moreover,
\begin{equation}\label{e:PPPgeom2}
\Prob \left( \sup_{x \in D} \omega(B_r(x)) \ge k+1 \right) \leq |B_{r}(D)|\frac{|B_{2 r}|^k}{(k+1)!}.
\end{equation}
\end{Lem}
\begin{proof}
We start with \eqref{e:PPPgeom1}.
We may assume that $|D|<\infty$.
First note that, if $y_0 \in D$ and $|y_i-y_{i-1}| < r$ for $1 \le i \le k$,
then $\{y_0, \ldots, y_k\} \subset D_k := B_{kr}(D)$.
Let $(X_i)_{i \ge 0}$ be i.i.d.\ random vectors, each uniformly distributed in $D_k$.
Note that, for any fixed $N \in \N$, $D_k \cap \cP$ has under its conditional law given that $\omega(D_k) = N$ 
the same distribution as $\{X_1, \ldots, X_N\}$.
For $N \ge k+1,$ estimate with a union bound
\begin{align}\label{e:prgeom1}
& \Prob \left( \exists \text{ distinct } j_0, \ldots, j_k \in \{1, \ldots, N\}  \colon\, X_{j_0} \in D, \max_{1 \le i \le k}|X_{j_i} - X_{j_{i-1}}| \le r \right) \nonumber\\
& \le \binom{N}{k+1} \Prob \left( X_0 \in D, \max_{1 \le i \le k} |X_{i}-X_{i-1}| \le r\right) \nonumber\\
& = \binom{N}{k+1} \frac{1}{|D_k|^{k+1}} \int_D \dd x_0 \int_{B_r(x_0)} \dd x_1 \cdots \int_{B_r(x_{k-1})} \dd x_k  
= \binom{N}{k+1} \frac{|D| |B_r|^k}{|D_k|^{k+1}}.
\end{align}
Since $|\cP \cap D_k|$ has distribution Poisson($|D_k|$),
splitting the left-hand side of \eqref{e:PPPgeom1} according to whether $|\cP \cap D_k| = N \ge k+1$
and using \eqref{e:prgeom1}, we get the bound
\begin{equation}\label{e:prgeom2}
\sum_{N=k+1}^{\infty} \binom{N}{k+1}\frac{|D| |B_r|^k}{|D_k|^{k+1}} \frac{|D_k|^N}{N!} \ee^{-|D_k|} = |D| \frac{|B_r|^k}{(k+1)!}
\end{equation}
as advertised.
Now \eqref{e:PPPgeom2} follows from \eqref{e:PPPgeom1} with $D$, $r$ substituted by $B_{r}(D)$, $2r$.
\end{proof}

Next we provide a lower bound on the probability to have close Poisson points.

\begin{Lem}\label{l:problbppD}
For all measurable $D\subset \R^d$, all $k\in \N_0$ and all $r\in (0,\infty)$,
\begin{equation}\label{e:problbppD}
\Prob \left(\exists x\in D\colon \omega(B_r(x))=  k+1\right) \geq  1-\exp\left\{- |D| \frac{r^{kd} \ee^{-|B_r|}}{2^d (k+1)!} \right\}
\end{equation}
\end{Lem}
\begin{proof}
Note that there exists a finite $F \subset D$ such that $B_r(x) \cap B_r(y) = \emptyset$ for all distinct $x,y \in F$
and $\#F \ge \lceil |D|/|B_{2 r}| \rceil$, which can be proved e.g.\ by induction on $\lceil |D|/|B_{2r}|\rceil$.
Then the family $\omega\left(B_r(x)\right)$, $x \in F$, is i.i.d., and we may estimate
\begin{align*}
& \Prob \left(\forall x\in D\colon \omega\left( B_r(x)\right)\neq k+1\right)
\leq \left(1-\Prob \left(\omega\left(B_r\right)= k+1\right)\right)^{\#F} 
\leq \exp \left\{ - |D|\frac{\Prob \left(\omega\left(B_r\right)= k+1\right)}{|B_{2r}|} \right\},
\end{align*}
where we also used $1-x\le \ee^{-x}$. 
Since $\omega(B_r)$ has distribution Poisson($|B_r|$), \eqref{e:problbppD} follows.
\end{proof}

We now apply the bounds in Lemmas~\ref{l:probupp}--\ref{l:problbppD} to derive several asymptotic results.
As a first consequence of Lemma~\ref{l:probupp}, we can show that,
for fixed $a>0$, the maximal number of Poisson points in $a$-neighbourhoods of points in $B_R$ 
grows at most logarithmically in $R$:

\begin{Cor}\label{cor:pppfixed}
For any $a \in (0,\infty)$,
\begin{equation}\label{e:pppfixed}
\lim_{R\rightarrow \infty} (\log R)^{-1}\sup_{x\in B_R}\omega(B_{a}(x))  = 0 \quad \Prob \mbox{-a.s.}
\end{equation}
\end{Cor}
\begin{proof}
Fix $a\in (0,\infty)$ and $K>1$. By \eqref{e:PPPgeom2}, there exists a constant $c\in(0,\infty)$ such that
\[
\Prob \left(\sup_{x\in B_{K^n}}\omega(B_{a}(x)) \geq n \right) \le c \frac{(K a)^n}{n!}.
\]
Since this is summable in $n$, the Borel-Cantelli lemma yields 
$\sup_{x\in B_{K^n}} \omega(B_{a}(x)) \le n$ a.s.\ eventually.
For $R\in(1,\infty)$, take $n_R \in \N$ such that $K^{n_R-1} < R \le K^{n_R}$. 
Then $\lim_{R \to \infty} n_R = \infty$ and
\begin{align*}
\limsup_{R\to\infty} (\log R)^{-1}\sup_{x\in B_R}\omega(B_{a}(x)) 
\le\lim_{R \to\infty} \frac{n_R}{\log R} =  (\log K)^{-1} \quad \Prob \mbox{-a.s.,} 
\end{align*}
and we complete the proof letting $K \to \infty$.
\end{proof}

Next we show that the number of points in neighborhoods with radii decreasing sufficiently fast
to $0$ are bounded by a constant.
Recall the notation $\cP = \{x \in \R^d \colon\, \omega(\{x\}) = 1\}$ for the support of $\omega$.

\begin{Lem}\label{l:asppp}
Fix $k \in \N$ and a function $g:(0,\infty) \to (0,\infty)$.
Let $R(t), r(t) \in (0,\infty)$ satisfy 
\begin{equation}\label{e:condaspppRr}
R(t) \to \infty, \quad r(t) \to 0 \quad \text{ and } \quad R(t) r(t)^k \sim g(t)^{-\frac{1}{d}} \quad \text{ as } t \to \infty.
\end{equation}
Assume that $R(t)$ is eventually non-decreasing, $r(t)$ is eventually non-increasing and
$\sum_{n=1}^{\infty} g(2^n)^{-1} < \infty$.
Then, for $N^{(r)}_\cY$ as in \eqref{e:numberpointsC} and $\cP_R = \cP \cap B_R$,
\begin{equation}\label{e:asppp}
\limsup_{t \to \infty} \sup_{|x| \le R(t)} \omega(B_{r(t)}(x)) \leq k
\qquad \text{ and } \qquad
\limsup_{t \to \infty} N^{(r(t))}_{\cP_{R(t)}} \le k \qquad \Prob \mbox{-a.s.}
\end{equation}
\end{Lem}

\begin{proof}
Applying \eqref{e:PPPgeom2} and our assumptions we get, for some constant $c>0$ and all $n$ large enough,
\begin{align}
\Prob \Bigg( \sup_{x\in B_{R(2^{n+1})}}\omega\big(B_{r(2^n)}(x)\big) \geq k+1\Bigg)
\le c g(2^n)^{-1}.
\end{align}
Now the Borel-Cantelli lemma implies that $\sup_{x\in B_{R(2^{n+1})}} \omega(B_{r(2^n)}(x))\le k$
almost surely for all large enough $n$, and the first inequality in \eqref{e:asppp} follows by interpolation and monotonicity.
To see that the second inequality follows from the first, note that, for any $R, r>0$,
\[
\{\exists \, \cC \in \mathscr{C}_{\cP_R}^{(r)} \colon\, N_\cC \ge k + 1 \} \subset \{\exists\, x \in B_R \colon\, \omega(B_{2k r}(x)) \ge k+1  \}.
\qedhere
\]
\end{proof}

The following corollary is immediate from \eqref{e:problbppD}.
\begin{Cor}\label{c:lbppp}
Fix $n\in\N$ and let $R(t)$, $r(t) \in (0,\infty)$ satisfy $r(t) \to 0$, 
$R(t) r(t)^k \to \infty$ as $t \to \infty$.
Then
\begin{equation}\label{e:lbppp}
\lim_{t \to \infty}\Prob \left(\exists\, x \in B_{R(t)} \colon \omega\left( B_{r(t)}(x)\right)= k+1\right)= 1.
\end{equation}
\end{Cor}

The next lemma is needed for the results on the $\limsup$-asymptotic.
\begin{Lem}\label{l:pplimsup}
Fix $k \in \N$. Let $R(t), r(t), g(t) \in (0,\infty)$ satisfy \eqref{e:condaspppRr} and $\sum_{n\ge1} g(2^n)^{-1} = \infty$.
Then
\begin{equation}\label{e:pplimsup}
\limsup_{t \to \infty} \sup_{|x| \le R(t)} \omega(B_{r(t)}(x)) \ge k+1 \qquad \Prob \mbox{-a.s.}
\end{equation}
\end{Lem}
\begin{proof}
Let $A_n:=B_{R(2^{n})-r(2^n)}\setminus B_{R(2^{n-1})+r(2^n)}$, $n\in\N$.
Using \eqref{e:problbppD}, our assumptions on $R(t), r(t)$ and $1 - \ee^{-x} \sim x$ as $x \to 0$, we find a constant $c>0$ such that
\begin{align*}
\sum_{n \in \N} \Prob \left(\sup_{x\in A_n}\omega(B_{r(2^n)}(x))\geq k+1\right) \geq c \sum_{n = 1}^\infty g(2^n)^{-1} = \infty.
\end{align*}
Noting that $\sup_{x\in A_n}\omega(B_{r(\beta^n)}(x))$, $n\in\N$, are independent random variables, 
the second Borel-Cantelli lemma yields the result.
\end{proof}

In the remaining lemmata we investigate the $\liminf$ behaviour.
\begin{Lem}\label{l:PPPliminf1}
Fix $k\in\N$. Let $R(t)$, $r(t) \in (0,\infty)$ such that $R(t)\to\infty$ and  $r(t) \to 0$ eventually monotonically as $t\to \infty$.
Assume that
\begin{equation}\label{e:assumpPPPliminf1}
c := \liminf_{t \to \infty} \frac{(R(t) r(t)^k)^d}{\log \log t} > \tfrac{2^d(k+1)!}{|B_1|},
\end{equation}
and furthermore that, for any $\eps>0$, there exists $\beta \in(1,\infty)$ such that $\inf_{n\in\N}\tfrac{R(\beta^n)}{R(\beta^{n+1})}\geq 1-\eps$.
Then
\begin{equation}\label{e:PPPliminf1}
\liminf_{t\rightarrow\infty} \sup_{x\in B_{R(t)}}\omega\Big(B_{r(t)}(x)\Big)
\geq k+1 \quad \Prob \mbox{-a.s.}
\end{equation}
\end{Lem}

\begin{proof}
For $\eps>0$ satisfying $\delta := c(1-\eps)^d |B_1| / (2^d(k+1)!) - 1 > 0$, 
choose $\beta\in(1,\infty)$ as in the statement.
By \eqref{e:problbppD} and our assumptions on $R$, $r$, 
for all $n$ large enough,
\begin{align*}
\begin{aligned}
\P \left(\sup_{x\in B_{R(\beta^n})}\omega\Big(B_{r(\beta^{n+1})}(x) \Big) \leq k \right)
& \leq \exp \left\{-\frac{|B_1| (1-\eps)^d}{2^d(k+1)!} R(\beta^{n+1})^d r(\beta^{n+1})^{kd} \ee^{-|B_r(\beta^{n+1})|} \right\} \leq n^{1+\delta/2}.
\end{aligned}
\end{align*}
Now \eqref{e:PPPliminf1} follows by the Borel-Cantelli lemma,
interpolation and monotonicity.
\end{proof}

We state next an improvement of \eqref{e:PPPgeom2}.
For $D \subset \R^d$ and $r>0$, we denote by
\begin{equation}\label{e:defvarthetaD}
\vartheta_r(D) := \min \left\{n \in \N \colon\, \exists\, z_1, \ldots, z_n \in \R^d, D \subset \cup_{i=1}^n (z_i + [0,r]^d)\right\}
\end{equation}
the minimum number of boxes of side-length $r$ needed to cover $D$.

\begin{Lem}\label{l:genPPPgeom}
For any $k, m \in \N$, any measurable $D_1, \ldots, D_m \subset \R^d$,
and any $r_1, \ldots, r_m \in (0,\infty)$,
\begin{equation}\label{e:genPPPgeom}
\Prob \left(\sup_{1 \le i \le m} \sup_{x \in D_i} \omega(B_{r_i}(x)) \le k \right) 
\ge \prod_{i=1}^m \left(1- \frac{(2r_i)^d |B_{2r_i}|^k}{(k+1)!} \right)^{\vartheta_{r_i}(D)}.
\end{equation}
\end{Lem}

\begin{proof}
We first note that $\sup_{x \in D_i} \omega(B_{r_i}(x))$, $1 \le i \le m$, is a family of associated random variables
(cf.\ \cite[Proposition~4]{Res88}, see also \cite[Theorem~5.1]{EPW67}), i.e.,
\begin{equation}\label{e:prgenPPP1}
\Prob \left(\sup_{1 \le i \le m} \sup_{x \in D_i} \omega(B_{r_i}(x)) \le k \right) \ge \prod_{i=1}^m \Prob \left(\sup_{x \in D_i} \omega(B_{r_i}(x)) \le k \right).
\end{equation}
Consider the case $m=1$, and write $D=D_1$, $r= r_1$.
Then, with $z_1, \ldots, z_{\widehat{m}} \in \R^d$
as in \eqref{e:defvarthetaD}, 
\begin{equation}\label{e:prgenPPP2}
\Prob \left(\sup_{x \in D} \omega(B_{r}(x)) \le k \right) 
\ge \Prob \left( \sup_{i=1}^{\vartheta_r(D)} \sup_{x \in z_i + [0,r]^d} \omega(B_r(x)) \le k \right)
\ge \left( 1 - (2r)^d\frac{|B_{2r}|^k}{(k+1)!}\right)^{\vartheta_r(D)}
\end{equation}
by \eqref{e:prgenPPP1} and \eqref{e:PPPgeom2}.
Now \eqref{e:genPPPgeom} follows from \eqref{e:prgenPPP1}--\eqref{e:prgenPPP2}.
\end{proof}

The following lemma uses ideas from \cite[Lemma~5.2]{CR11}.
\begin{Lem}\label{l:ppliminf}
Let $k \geq 2$ and $R(t), r(t)>0$ satisfy
\begin{equation}\label{e:defrkn}
R(t) \sim t^\frac{k}{k-1} (\log \log t)^{-\frac{1}{d(k-1)}}, \qquad r(t) \sim t^{-\frac{1}{k-1}} (\log \log t)^\frac{1}{d(k-1)} \qquad \text{ as } t \to \infty.
\end{equation}
Let $b_n > 0$, $n \in \N$, such that
\begin{equation}\label{e:condck}
\sum_{n=1}^{\infty}(2^{n-1} b_n^k)^d < \frac{(k+1)!}{(2^d |B_1|)^{k+1}}.
\end{equation}
Let $\rho > 0$ and $z(t):= \lfloor \rho \log \log t \rfloor$. Then
\begin{equation}\label{e:ppliminf1}
\liminf_{t\rightarrow\infty} \sup_{n=1}^{z(t)} \sup_{x\in B_{2^{n-1} R(t)}}\omega\Big(B_{b_n r(t)}(x)\Big)\leq k 
\quad \Prob \mbox{-a.s.}
\end{equation}
\end{Lem}

\begin{proof}
We may assume that $\rho > 1$. Abbreviate $\ell(t) := \log \log t$.
Take $t_0\in(1,\infty)$ large enough such that $\ell(t_0)>1$,
and define a growing sequence $(t_j)_{j\in\N_0}$ recursively by
\begin{equation}\label{e:defRn}
t_{j}=t_{j-1} \exp\{\rho \ell(t_{j-1})\}, \quad j\in\N.
\end{equation}
For $j\in\N$ and $n\in\N$, set
\begin{align*}
A_{j,n}&:=B_{2^{n-1} R(t_{j})}\setminus B_{R(t_{j-1})},& 
X_j&:=\sup_{n=1}^{z(t_j)} \sup_{x\in B_{2^{n-1} R(t_{j})}}\omega(B_{b_n r(t_{j})}(x)),\\
\hat{X}_j&:=\sup_{n=1}^{z(t_j)} \sup_{x\in A_{j,n}}\omega(B_{b_n r(t_{j})}(x)),&
\check{X}_j&:=\sup_{x\in B_{R(t_{j-1})}} \omega(B_{b_1 r(t_{j})}(x)).
\end{align*}
Note that $X_j=\max(\check{X}_j,\hat{X}_j).$ Thus it will be sufficient to show that $\Prob$-a.s.\ both 
\begin{equation}\label{e:liminfsup}
\limsup_{j\rightarrow\infty}\check{X}_j\leq k \quad\mbox{and}\quad \liminf_{j\rightarrow\infty}\hat{X}_j\leq k.
\end{equation}
To obtain the first inequality, note that by \eqref{e:PPPgeom2} there exists a constant $c\in(0,\infty)$ such that
\begin{align}\label{e:proofliminf1}
\Prob \left(\check{X}_j\geq k+1 \right)
\leq c \left(R(t_{j-1})r(t_j)^k\right)^d
& \leq 2 c \ee^{-\tfrac{k}{k-1}\left\{ d \rho \ell(t_{j-1}) - \log \ell(t_{j-1}) \right\}} 
\leq 2c \ee^{-\frac{d k \rho (1-\eps_j)}{k-1}\ell(t_{j-1})}
\end{align}
for all large enough $j$,
where we used $\ell(t_j) \le 2 \ell(t_{j-1})$, and $\eps_j \to 0$ as $j \to \infty$.
To conclude with the Borel-Cantelli lemma,
note that \eqref{e:proofliminf1} is summable in $j$ since, for any $\alpha>1$,
\begin{align*}
\infty >\int_{t_0}^\infty\frac{1}{t}\ee^{-\alpha \ell(t)} \dd t=\sum_{j=0}^\infty\int_{t_j}^{t_{j+1}}\frac{1}{t}\ee^{-\alpha \ell(t)}\dd t
 \geq \sum_{j=0}^\infty \log ( t_{j+1}/t_j) \ee^{-\alpha \ell(t_{j+1})} > \sum_{j=0}^\infty\ee^{-\alpha \ell(t_{j+1})}.
\end{align*}

Consider now the second inequality in \eqref{e:liminfsup}.
By \eqref{e:genPPPgeom}, for all $j \in \N$,
\begin{align*}
\log\Prob \left(\hat{X}_j\leq k \right) \geq \log \Prob \left(X_j\leq k \right) 
\geq \sum_{n=1}^{z(t_j)} \vartheta_{b_n r(t_j)}\left(B_{2^{n-1} R(t_j)}\right) \log \left( 1 - \frac{(2 b_n r(t_j))^d |B_{2 b_n r(t_j)}|^k}{(k+1)!} \right).
\end{align*}
Using \eqref{e:defrkn}, $\log(1-x) \sim -x$ as $x \to 0$ and $\vartheta_r(B_R) \sim |B_R|/r^d$ as $r\downarrow 0, R\uparrow\infty$, 
we obtain
\begin{align*}
\log\Prob \left(\hat{X}_j\leq k \right) \geq - (1+ \eps_j) \frac{(2^d |B_1|)^{k+1}}{(k+1)!} \ell(t_j) \sum_{n=1}^{\infty} (2^{n-1} b_n^k)^d
 > -(1-\delta) \ell(t_j)
\end{align*}
for large $j$ by \eqref{e:condck}, where $\eps_j \to 0$ as $j \to \infty$ and $\delta \in (0,1)$.
Since, for some $c \in (0,\infty)$,
\[
\infty =\int_{t_0}^\infty\frac{1}{t}\ee^{-\ell(t)} \dd t 
\leq \sum_{j=0}^\infty \rho \ell(t_j) \ee^{-\ell(t_{j})} \leq c \sum_{j=0}^\infty\ee^{-(1-\delta) \ell(t_j)},
\]
we deduce $\sum_{j=0}^\infty \Prob (\hat{X}_j\leq k)=\infty$.
Note now that, since $R(t_{j+1}) \gg 2^{z(t_j)} R(t_j)$ as $j \to \infty$,
there exists a $j_0 \in \N$ such that both $(\hat{X}_{2j})_{j \ge j_0}$ and $(\hat{X}_{2j+1})_{j \ge j_0}$ are families of independent random variables,
allowing us to conclude the proof with an application of the second Borel-Cantelli lemma.
\end{proof}

\section{Proof of the main theorems}
\label{s:PMT}
%
Throughout this section, we fix $d \ge 3$ arbitrary in general,
but $d=3$ whenever we treat the renormalized potential $\overbar{V}$.
We also fix $\theta \in (0,\frac{h_d}{2}]$ and set $k=k_\theta=\lfloor \frac{h_d}{\theta}\rfloor$, where $h_d = (d-2)^2/8$ .

The section is organized as follows. In Sections~\ref{ss:taileigenvalue}--\ref{ss:truncation} below,
we provide some preparatory results concerning respectively bounds for principal eigenvalues 
and estimates of the error introduced when substituting either $V^{(\frK)}$ or $\overbar{V}$ by a truncated potential $V^{(a)}$.
Section~\ref{ss:proofsupperbounds} contains the proofs of Theorems~\ref{t:finiteness} and \ref{t:finitenessva}
as well as of the upper bounds for Theorems~\ref{t:tightnessva}, \ref{t:limsupva} and \ref{t:liminfva}.
Corresponding lower bounds are proved first in the special case of truncated potentials in Section~\ref{ss:proofslowerbounds}.
The proofs of Theorem~\ref{t:limitfrac} is given in Section~Section~\ref{ss:prooflimitfrac}, 
as well as the completion of the proofs of Theorems~\ref{t:tightnessva}, \ref{t:limsupva}, \ref{t:liminfva} and \ref{t:tightness}.
Finally, Theorems~\ref{t:solvbar} and \ref{t:solva} are proved in Section~\ref{ss:proofssol}.

\subsection{Bounds for principal eigenvalues}
\label{ss:taileigenvalue}
%
In order to make use of the upper bound given in Theorem~\ref{t:upperbound}, 
we study the almost-sure asymptotics as $R\to\infty$ of $\Lambda^{(\theta, a, r)}_\cY$ defined in \eqref{e:numberpointsC}
with $\cY=\cP_R = \cP \cap B_R$. 
To this end, we will combine the multipolar Hardy inequality from Section~\ref{ss:MPHI} and the Poissonian asymptotics stated in Section~\ref{s:SDPC}.

Fix $0<a<r<\infty$ and recall \eqref{e:deflambdaC}--\eqref{e:numberpointsC}. For $s>0$, write
\begin{align}\label{e:proofev}
\{\Lambda^{(\theta, a, r)}_{\cP_R}>s\} \subset \{\Lambda^{(\theta, a, r)}_{\cP_R}>s, N_{\cP_R}^{(r)} \le k+1\} \cup \{N_{\cP_R}^{(r)} \geq k+2\}.
\end{align}
The second event in \eqref{e:proofev} can be controlled by
\[
\{N_{\cP_R}^{(r)} \geq k+2\}\subset \big\{\exists x\in B_R\colon\omega\big(B_{(k+1)r}(x)\big)\geq k+2\big\}.
\]
To control the first event in \eqref{e:proofev}, write, for $\cC\in\mathscr{C}^{(r)}_{\cP_R}$,
\[\Gamma(\cC):=\inf \left\{s > 0 \colon\, B_{s}(\cP_R\cap \cC) \text{ is connected} \right\}\] 
and set 
\begin{equation}\label{e:defcmp}
c_{\mph} := (k+1)\frac{\pi^2 + 3 \theta}{2}.
\end{equation}
Note that $\lambda_\cC=0$ for each $\cC\in\mathscr{C}_{\cP_R}^{(r)}$ with $N_\cC\leq k$ due to the Hardy inequality 
(cf.\ \eqref{e:MPHI_for_smalltheta}) and Remark~\ref{r:evmonoton}.
Then, by the multipolar Hardy inequality \eqref{e:MPHI},
\begin{align*}
&\{\Lambda^{(\theta, a, r)}_{\cP_R}>s, N_{\cP_R}^{(r)} \le k+1\} \subset\{\exists \cC\in\mathscr{C}^{(r)}_{\cP_R}\colon\lambda_\cC>s, \,  N_\cC=k+1\} \\
& \subset\Big\{\exists \cC\in\mathscr{C}^{(r)}_{\cP_R}\colon \Gamma(\cC)^2 < c_{\textnormal{mp}}/s, \, N_\cC=k+1 \Big\} \\
&\subset\Big\{\exists \ \mbox{distinct} \ y_1,\ldots, y_{k+1}\in \cP_R \colon\, \cup_{i=1}^{k+1}B_{ (c_{\textnormal{mp}}/s)^{1/2}}(y_i) \text{ is connected }\Big\}\\
&\subset\big\{\exists x \in B_R\colon \omega\big(B_{2 k (c_{\mph}/s)^{1/2}}(x)\big)\geq k+1\big\}.
\end{align*}
Combining these results, we get
\begin{align}\label{e:evin}
\{\Lambda^{(\theta, a, r)}_{\cP_R}>s\} \subset 
\big\{\sup_{x \in B_R} \omega\big(B_{2 k(c_{\mph}/s)^{1/2}}(x)\big)\geq k+1\big\}
\cup\big\{\sup_{x\in B_R} \omega\big(B_{(k+1)r}(x)\big)\geq k+2\big\}.
\end{align}
With this inclusion at hand, we derive next several consequences of the results from Section~\ref{s:SDPC}.

\begin{Lem}\label{l:eva}
Let $0<a<r<R<\infty$ and $\theta\in(0,\frac{h_d}{2}]$. 
There exists a constant $c\in(0,\infty)$ depending only on $\theta$ and $d$ such that, for all $s > c R^{-2}$,
\begin{equation}\label{e:evbp}
\Prob \left( \Lambda^{(\theta, a, r)}_{\cP_R}>s \right) \le cR^d\left(s^{-\frac{d}2k}+r^{d(k+1)}\right).
\end{equation} 
\end{Lem}
\begin{proof}
We can assume $c \ge 4 k^2 c_{\mph}$.
Using \eqref{e:evin}, \eqref{e:PPPgeom2} and $2 k(c_{\mph}/s)^{1/2} < R$, we get
\begin{align*}
\Prob \left( \Lambda^{(\theta, a, r)}_{\cP_R}>s \right) 
& \leq \Prob \left(\sup_{x \in B_R} \omega\Big(B_{2k (c_{\mph}/s)^{1/2}}(x)\Big)\geq k+1\right)
+\Prob \left(\sup_{x \in B_R}\omega\Big(B_{(k+1)r}(x)\Big)\geq k+2\right)\\
& \leq |B_1|(2R)^d \frac{\big(|B_1| (4 k (c_{\mph}/s)^{1/2})^d \big)^k }{(k+1)!} + |B_1| (2(k+1)R)^d \frac{\big(|B_1|(2(k+1)r)^d\big)^{k+1}}{(k+2)!}.
\end{align*}
This shows \eqref{e:evbp}.
\end{proof}

\begin{Lem}\label{l:evp}
Fix $\alpha>(k+1)^{-1}$ and let $R(t)\to \infty, g(t)\to\infty$ as $t\to\infty$.
For any $c_1, c_2 \in (0,\infty)$,
\begin{equation}\label{e:evp}
\lim_{t \to \infty}  \frac{\Lambda_{R(t)}}{g(t) R(t)^{2/k}} = 0 \quad \mbox{ in probability, } \;\;\; \text{ where } \quad \Lambda_R := \Lambda_{\cP_{R}}^{(\theta, c_1 R^{-\alpha}, c_2 R^{-\alpha})}.
\end{equation}
If moreover $\sum_{n=1}^{\infty} g(2^n)^{-dk/2} < \infty$, 
$R$ is regularly varying with positive index,
and $g$ is either eventually non-decreasing or slowly varying,
then \eqref{e:evp} holds almost surely.
\end{Lem}

\begin{proof}
\eqref{e:evp} follows directly from \eqref{e:evbp}. 
For the second statement, note that, for $n\in\N$, \eqref{e:evin} yields
\begin{align}\nonumber
\{\Lambda_{R(t)} > n^{-1} g(t) R(t)^{2/k} \} \subset
\{\sup_{|x| \leq R(t)}\omega(B_{r(t)}(x))\geq k+1\} \cup\{\sup_{|x| \leq R(t)} \omega(B_{c_2 (k+1)R(t)^{-\alpha}}(x))\geq k+2\},
\end{align}
where $r(t) = 2 k \sqrt{c_{\mph} n} g(t)^{-1/2} R(t)^{-1/k}$.
By \cite[Theorem~1.5.3]{BGT02}, we may assume that $R(t)$ and $r(t)$ are eventually monotone.
By \eqref{e:asppp}, $\limsup_{t\to\infty} \Lambda_{R(t)}/(g(t) R(t)^{2/k}) \le 1/n$ almost surely,
and to conclude we let $n\uparrow\infty$.
\end{proof}

The following lemma will be used in the proof of Theorem~\ref{t:liminfva}.
\begin{Lem}\label{l:liminfev}
Let $R(t)$ as in \eqref{e:defrkn} and $\alpha >(k+1)^{-1}$. 
For $n \ge 1$, let $a_n(t) := (2^{n-1} R(t))^{-\alpha}$
and, for $A>0$,
\[
\Lambda_{t,n}:=\Lambda^{(\theta, a_{n}(t), 5a_{n}(t))}_{\cP_{2^{n-1} R(t)+1}}, \qquad  \Theta_{t,n}(A) := \Lambda_{t,n}- A \1{\{n \ge 2\}} 4^{n-1} t^{\frac{2}{k-1}}(\log \log t)^{-\frac{2}{d(k-1)}}.
\]
Let $\rho > 0$ and $z(t):= \lfloor \rho \log \log t \rfloor$. For any $A>0$, there exists a $C=C(A,k,d) \in (0,\infty)$ such that
\begin{equation}\label{e:liminfev}
\liminf_{t\rightarrow\infty}t^{-\frac{2}{k-1}}(\log \log t)^{\frac{2}{d(k-1)}}\max_{n=1}^{z(t)}\Theta_{t,n}(A) \leq C \quad \Prob \mbox{-almost surely.}
\end{equation}
\end{Lem}

\begin{proof}
Fix $A, \rho>0$. Let $\chi_k:=(k+1)!/(|B_2|^{k+1})$ as in \eqref{e:condck} and $c_{\mph}$ as in \eqref{e:defcmp},
and pick
\begin{equation}\label{e:previnf1}
C > (4 A) \vee \left(\frac{(4 k \sqrt{c_{\mph}})^{k}}{ \sqrt{A} \chi_k^{1/d}} \right)^{2/(k-1)}. 
\end{equation}
Define $b_n > 0$, $n \in \N$ by setting
\[
b_1= 2k (c_{\mph}/C)^{1/2} \quad  \mbox{and} \quad  b_n= 2k (c_{\mph}/C)^{1/2} \left(1+ (A/C) 4^{n-1}\right)^{-\tfrac12},\quad n\geq 2.
\]
Let us verify that $b_n$ satisfies \eqref{e:condck}.
Indeed, setting $n_0 := \lfloor \log_4(C/A) \rfloor \ge 1$, we may write
\begin{align*}
& (2 k (c_{\mph}/C)^{1/2})^{-kd} \sum_{n=1}^{\infty}\big(2^{n-1} b_n^k\big)^d  \le \sum_{n=1}^{n_0+1} 2^{(n-1)d} + \left(C/A \right)^{kd/2} \sum_{n=n_0+2}^{\infty} 2^{-(k-1)d (n-1)} \\
& \le 2^{n_0 d+1} + 2 \left(C/A\right)^{kd/2} 2^{-(k-1) d n_0 } 
 \le  2^{kd} (C/A)^{d/2} \\
 & = (2 k (c_{\mph}/C)^{1/2})^{-kd} \left( (4 k \sqrt{c_{\mph}})^{k} A^{-1/2} C^{-(k-1)/2} \right)^d
<  (2 k (c_{\mph}/C)^{1/2})^{-kd} \chi_k
\end{align*}
by our choice of $C$. This shows \eqref{e:condck}.
Let now $r(t):= t^{-\frac{1}{k-1}} (\log \log t)^{\frac{1}{d(k-1)}}$ and use \eqref{e:evin} to write
\begin{align*}
&\Big\{\max_{n=1}^{z(t)}\Theta_{t,n}(A) \leq C t^{\frac{2}{k-1}} (\log \log t)^{-\frac{2}{d(k-1)}}\Big\} 
=\bigcap_{n=1}^{z(t)}\Big\{\Lambda_{t,n}\leq r(t)^{-2} (2k \sqrt{ c_{\mph}})^{2} b_n^{-2}\Big\} 
\\ &\supset\bigcap_{n=1}^{z(t)} \{\sup_{ |x| \leq 2^{n-1}R(t)+1 }  \omega(B_{b_n r(t)}(x)) \leq k \}\cap\{\sup_{|x| \leq 2^{n-1}R(t)+1}  \omega(B_{5(1+k)a_n(t)}(x))\leq k+1\}\big)\\
&\supset\Big\{\max_{n=1}^{z(t)}\sup_{|x| \leq 2^{n-1}(R(t)+1)} \omega(B_{b_n r(t)}(x))\leq k\Big\}\cap \Big\{\sup_{|x| \leq 2^{z(t)}(R(t)+1)} \omega(B_{5(1+k)a_0(t)}(x))\leq k+1\Big\}.
\end{align*}
The first event on the right-hand side above occurs a.s.\ infinitely often by \eqref{e:ppliminf1}, 
and the second event occurs eventually by \eqref{e:asppp}. 
This yields \eqref{e:liminfev}.
\end{proof}

\subsection{Truncation of Poisson potentials}
\label{ss:truncation}

In this section, we control the error that occurs when replacing either an attenuated potential $V^{(\frK)}$ as in  \eqref{e:defVK}
or the renormalized potential $\overbar{V}$ by a truncated potential $V^{(a)} = V^{(\frK_a)}$, 
where $\frK_a(x) = |x|^{-2}\mathbbm{1}_{\{|x| \le a\}}$.
We first state an auxiliary result.

\begin{Lem}\label{l:decayPoissonInt}
Let $R \mapsto \frK(R) \in L^1(\R^d) \cap L^\infty(\R^d)$ satisfy
\begin{equation}\label{e:conddecayPI}
C := \limsup_{R\to\infty} \|\frK(R)\|_{L^\infty(\R^d)} < \infty 
\qquad \text{ and } \qquad
\limsup_{R \to \infty} \int_{\R^d} \sup_{|x|\le 1} |\frK(R)(x-y)| \dd y < \infty.
\end{equation}
Then
\begin{equation}\label{e:decayPoissonInt}
\limsup_{R \to \infty} \frac{\log \log R}{\log R}  \max_{|x| \le R} |V^{(\frK(R))}(x)| \le d C \qquad \Prob \mbox{-a.s.}
\end{equation}
\end{Lem}
\begin{proof}
Using \eqref{e:conddecayPI} and \cite[Proposition~2.7]{CK12},
one can follow the proof of \cite[Lemma~2.6]{GKM00}.
\end{proof}

Our comparison lemma reads as follows.

\begin{Lem}\label{l:errortruncation}
Let $d \ge 3$, $\frK \in \scrK$ and $a \in (0,\infty)$. 
Then, $\Prob$-almost surely for all bounded $D \subset \R^d$,
\begin{equation}\label{e:truncationfinite}
\sup_{x\in D \setminus \cP}|V^{(\frK)}(x) - V^{(a)}(x)|<\infty.
\end{equation}
Moreover, for any $R \mapsto a_R > 0$ such that $\limsup_{R\to\infty} a_R < \infty$,
\begin{equation}\label{e:errortruncation}
\lim_{R\rightarrow\infty} \frac{a_R^2}{\log R}\sup_{|x|\le R \colon x \notin \cP} \left|V^{(\frK)}(x)-V^{(a_R)}(x) \right|= 0 \quad \Prob \mbox{-a.s.}
\end{equation}
When $d=3$, \eqref{e:truncationfinite}--\eqref{e:errortruncation} hold with either $\overbar{V}$ or $|\overbar{V}|$ in place of $V^{(\frK)}$.
\end{Lem}
\begin{proof}
Note that, for all $x \in \R^d \setminus \cP$ and all $a>0$,
\[
\left|V^{(\frK)}(x) - V^{(a)}(x) \right| \le \omega(B_a(x)) \sup_{|z|\le a}\left|\frK(z) - \frac{1}{|z|^2}\right| + a^{-2} \int a^2 |\frK(x-y)|\mathbbm{1}_{\{|x-y|>a\}} \omega(\dd y),   
\]
proving \eqref{e:truncationfinite}.
With $a = a_R$ as in the statement, \eqref{e:errortruncation} follows
by \eqref{e:condscrK2}, Corollary~\ref{cor:pppfixed} and Lemma~\ref{l:decayPoissonInt}.

Consider now $d=3$. Fix $a>0$ and let $\alpha:\R^+\rightarrow [0,1]$ be a smooth truncation function with $\alpha(\lambda)=1$ on $[0,1]$,  
$\alpha(\lambda)=0$ for $\lambda\geq 3$ and $-1\le \alpha'(\lambda)\le 0$.
Decompose $\overbar{V} = \overbar{V}_1 + \overbar{V}_2$ by setting
\begin{equation}\overbar{V}_1(x):=\int_{\R^3}\frac{1-\alpha(a^{-1}|x-y|)}{|x-y|^2}[\omega(\dd y)-\dd y],
\qquad \overbar{V}_2(x):=\int_{\R^3}\frac{\alpha(a^{-1}|x-y|)}{|x-y|^2}[\omega(\dd y)-\dd y].
\end{equation}
Note that $\overbar{V}_1$ exactly matches $\overbar{V}_{a,\eps}$ in \cite[Eq.~(3.5)]{CR11} with $\eps=1$. Thus, by \cite[Eq.~(3.6)]{CR11},
\begin{equation} \label{e:V1finite}
\sup_{x\in D}\overbar{V}_1(x)<\infty \quad \Prob \mbox{-a.s.}
\end{equation}
for any bounded $D \subset \R^3$ while, by Lemma~3.3 in the same reference,
\begin{equation}\label{e:lbt1}
\lim_{R\rightarrow\infty}(\log R)^{-1}\sup_{|x|\le R}|\overbar{V}_1(x)|=0 \quad \Prob \mbox{-a.s.}
\end{equation}
Furthermore, since the integrand in the definition of $\overbar{V}_2$ is in $L^1(\R^3)$, we may separate the integration in terms of $\omega(\dd y)$ and $\dd y$ using \cite[Proposition~2.5]{CK12}, i.e.,
\begin{equation}\label{e:sepV2}
\overbar{V}_2(x)=\int_{\R^3}\frac{\alpha(a^{-1}|x-y|)}{|x-y|^2}\omega(\dd y)-\int_{\R^3}\frac{\alpha(a^{-1}|x-y|)}{|x-y|^2}\dd y.
\end{equation}
The second integral above is a finite constant independent of $x$. For the first integral, we get
\begin{align}
\int_{\R^3}\frac{\alpha(a^{-1}|x-y|)}{|x-y|^2}\omega(\dd y)=V^{(b)}(x)+\int_{\R^3}\frac{\alpha(a^{-1}|x-y|)}{|x-y|^2}\mathbbm{1}_{\{|x-y| \geq b \}}\omega(\dd y)
\end{align}
for any $b \in (0,a]$. Now note that, since $\alpha(\lambda)=0$ for $\lambda\geq 3$,
\begin{equation} \label{e:V2finite}
\sup_{x\in D}\int_{\R^3}\frac{\alpha(a^{-1}|x-y|)}{|x-y|^2}\mathbbm{1}_{\{|x-y|\geq b \}}\omega(\dd y)\le b^{-2} \sup_{x\in D}\omega(B_{3a}(x))<\infty 
\quad \Prob \mbox{-a.s.}
\end{equation}
Combining \eqref{e:V1finite} and \eqref{e:sepV2}--\eqref{e:V2finite} with $b =a$, 
we obtain \eqref{e:truncationfinite} with $\overbar{V}$ in place of $V^{(\frK)}$.
To obtain \eqref{e:errortruncation}, take $a > \limsup_{R\to\infty} a_R$, $b=a_R$, $D=B_R$
and apply additionally \eqref{e:lbt1} and Corollary~\ref{cor:pppfixed}.
The statement for $|\overbar{V}|$ is obtained with the inequality $||x| - |y|| \le |x-y|$, $x,y \in \R$.
\end{proof}

\subsection{The upper bounds}
\label{ss:proofsupperbounds}

We introduce next some notation and a key result that will be used in the following proofs of the upper bounds.
Fix $\alpha \in (\frac{1}{k+1}, \frac{1}{k})$ and recall \eqref{e:numberpointsC}.
Throughout the section, we will use the notation
\begin{equation}\label{e:defLambdaRUB}
\Lambda_R := \Lambda_{\cP\cap B_{R+2}}^{(\theta, R^{-\alpha}, 5 R^{-\alpha})}, \qquad R>0.
\end{equation}
Note that, for any $a \in (0,1]$ and $z \in B_{R+1}$,
$V^{(a)}(z) = V^{(a)}_{\cP \cap B_{R+2}}(z)$.

In the proofs below, we will work with certain radii sequences $R_n(t) \in [1, \infty)$,
$n \in \N$, $t>0$, which we keep arbitrary for now.
According to the choice of $R_n(t)$, we introduce
\begin{align}\label{e:defscalesUpbd}
a_n(t)&=R_{n}(t)^{-\alpha},& r_{n}(t)&=5a_n(t), & R_0(t)&= 8(k+1) r_1(t),
\end{align}
as well as the hitting times
\begin{equation}\label{e:defhattaun}
\hat{\tau}_n(x) = \hat{\tau}_n(t,x) :=\tau_{B^\cc_{R_{n}(t)}(x)}=\inf\{s\geq 0\colon W_s\notin B_{R_{n}(t)}(x)\}, \quad n\in\N_0, x \in \R^d.
\end{equation}
Fix $\frK \in \scrK$ and define the error terms
\begin{equation}\label{e:defsnt}
S_n(t):=\sup_{z\in B_{R_{n}(t)+1}}|V^{(\frK)}(z)-V^{(a_n(t))}(z)|, \qquad \overbar{S}_n(t):=\sup_{z\in B_{R_{n}(t)+1}}|\overbar{V}(z)-V^{(a_n(t))}(z)|.
\end{equation}
Recall \eqref{e:numberpointsC} and define, for $t >0$,
\begin{equation}\label{e:defzetacirc}
\zeta^\circ_t:= \inf\Big\{n\in\N\colon \ N^{(r_n(t))}_{\mathcal{P} \cap B_{R_n(t)+2}}\leq k+1
\ \mbox{and} \ R_{n-1}(t) \ge 8 r_n(t) (k+1) \Big\}
\end{equation}
and,  for $x \in \R^d \setminus \cP$,  
\begin{equation}\label{e:defzeta}
\zeta_t(x):= \zeta^\circ_t \vee \inf\Big\{n\in\N\colon x \notin B_{r_n(t)}(\cP) \Big\}.
\end{equation}
When $x=0$, we write $\hat{\tau}_n$, $\zeta_t$ instead of $\hat{\tau}_n(0)$, $\zeta_t(0)$.
The next lemma provides conditions on $R_n(t)$ guaranteeing the finiteness of $\zeta^\circ_t$, $\zeta_t(x)$.
\begin{Lem}\label{l:zetafinite}
Let $R_n(t) \ge 1$, $n\in \N$, $t>0$ satisfy
\begin{equation}\label{e:condRn(t)}
\forall t_2 > t_1 > 0 \colon\, \quad \lim_{n \to \infty} R_n(t_1) = \infty \quad \text{ and } \quad \liminf_{n\to\infty} \inf_{t,s \in [t_1, t_2]}\frac{R_n(t)}{R_n(s)} > 0.
\end{equation}
Then, $\Prob$-almost surely for all $x \in \R^d \setminus \cP$ and $t>0$, 
$1 \leq \zeta^\circ_t \leq \zeta_t(x) < \infty$ 
and there exist $0 \leq t^\circ_0 \leq  t_0(x) <\infty$ 
such that $\zeta^\circ_t = 1$ for all $t \geq t^\circ_0$ and 
$\zeta_t(x) = 1$ for all $t \ge t_0(x)$.
\end{Lem}

\begin{proof}
If \eqref{e:condRn(t)} holds, then, for any $\epsilon>0$,
\[
\lim_{n\to \infty} \sup_{\epsilon \le t \le \epsilon^{-1}} r_n(t) = 0, \quad 
\lim_{n\to\infty}\inf_{\epsilon \le t \le \epsilon^{-1}} R_{n-1}(t) = \infty
\quad \text{and} \quad
\limsup_{n\to\infty} \sup_{\epsilon \le t \le \epsilon^{-1}} N^{(r_n(t))}_{\cP \cap B_{2R_n(t)}} \le k
\]
almost surely by \eqref{e:asppp} (with $R(t) = t$). Similar estimates hold when $n=1$, $t \to \infty$.
\end{proof}

We are now ready to state the key estimate of the section.
\begin{Lem}\label{l:keyUB}
There exist deterministic constants $\chi \in [1,\infty)$ and $c_1, c_2 \in (0,\infty)$
such that, for any $R_n(t)\ge1$ satisfying \eqref{e:condRn(t)},
the following holds $\Prob$-almost surely for all $t\geq 0$.
Let
\begin{equation}\label{e:condgamman}
\gamma_n(t) \geq \max \left\{ 2 \Lambda_{R_n(t)}, \chi R_n(t)^{2\alpha} \right\}, \qquad n \in \N.
\end{equation}
Then, for all $\frK \in \scrK$ and all $0\le A_1<A_2 \le \infty$, 
\begin{align}\label{e:keyUBL1}
& \int_{B_1}\E_x \left[ \ee^{\int_0^t \theta V^{(\frK)}(W_s) \dd s } \mathbbm{1}_{\{\tau_{B^\cc_{A_1}(x)} \le t < \tau_{B^\cc_{A_2}(x)}\}} \right] \dd x
 \le \int_{B_1} \E_x \left[ \ee^{\int_0^t \theta V^{(\frK)}(W_s) \dd s } \mathbbm{1}_{\{\tau_{B^\cc_{A_1}(x)} \le t < \hat{\tau}_{\zeta_t^\circ-1}(x) \}} \right]  \dd x \\
& + \sqrt{1 \vee |\cP \cap B_1|} \sum_{\substack{n \ge \zeta_t^\circ \colon \\ A_1 < R_n(t) < A_2}} 
c_1 \ee^{
t \theta S_n(t) + t \gamma_n(t) + \log^+(\sqrt{t} R_n(t)^{\alpha}) - c_2 R_{n-1}(t) \min \left\{t^{-1} R_{n-1}(t), \sqrt{\gamma_n(t)}  \right\}.
}.\nonumber
\end{align}
Moreover,
for all $x \in B_1 \setminus \cP$,
\begin{align}\label{e:keyUB}
& \E_x \left[ \ee^{\int_0^t \theta V^{(\frK)}(W_s) \dd s } \mathbbm{1}_{\{\tau_{B^\cc_{A_1}(x)} \le t < \tau_{B^\cc_{A_2}(x)}\}} \right]
 \le \E_x \left[ \ee^{\int_0^t \theta V^{(\frK)}(W_s) \dd s } \mathbbm{1}_{\{\tau_{B^\cc_{A_1}(x)} \le t < \hat{\tau}_{\zeta_t(x)-1}(x) \}} \right] \\
& +  \sum_{\substack{n \ge \zeta_t(x) \colon \\ A_1 < R_n(t) < A_2}} 
c_1 \ee^{
t \theta S_n(t) + t \gamma_n(t) + \log^+(\sqrt{t} R_n(t)^{\alpha}) - c_2 R_{n-1}(t) \min \left\{t^{-1} R_{n-1}(t), \sqrt{\gamma_n(t)} \right\}
}.\nonumber
\end{align}
The same bounds hold with $|V^{(\frK)}|$ in place of $V^{\frK}$ and, when $d=3$, with $|\overbar{V}|$, $\overbar{S}_n(t)$ in place of $V^{(\frK)}$, $S_n(t)$. 
\end{Lem}

\begin{proof}
Fix $x \in B_1$. 
Splitting according to whether $t \geq \hat{\tau}_{\zeta_t^\circ-1}(x)$ or not and, if so,
according to which $n \ge \zeta_t^\circ$ satisfies $\hat{\tau}_{n-1}(x) \le t < \hat{\tau}_{n}(x)$,
we may decompose
\begin{align}\label{e:prkeyUB1}
& \E_x\left[\ee^{\int_0^t \theta V^{(\frK)}(W_s)\dd s} \mathbbm{1}_{\{\tau_{B^\cc_{A_1}(x)} \le t < \tau_{B^\cc_{A_2}(x)} \}}\right] 
\le \E_x\left[\ee^{ \int_0^t \theta V^{(\frK)}(W_s)\dd s} \mathbbm{1}_{\{\tau_{B^\cc_{A_1}(x)} \le t < \hat{\tau}_{\zeta_t^\circ-1}(x) \}}\right] \nonumber\\
& + \sum_{\substack{n\ge \zeta_t^\circ \colon \\ A_1 < R_n(t) < A_2 }} \E_x\left[\exp\left( \int_0^t \theta V^{(\frK)}(W_s)\dd s\right)\mathbbm{1}_{\{\hat{\tau}_{n-1}(x)\le t< \hat{\tau}_{n}(x)\}}\right].
\end{align}
Set $\cY_n(t) := \cP \cap B_{R_n(t)+2}$ and note that,
if $t < \hat{\tau}_n(x)$, then $V^{(a_n(t))}(W_s) = V^{(a_n(t))}_{\cY_n(t)}(W_s)$ for all $s \in [0,t]$.
Recalling \eqref{e:defsnt}, we see that the integral over $B_1$ of the series in \eqref{e:prkeyUB1} is bounded by
\begin{equation}\label{e:prkeyUB2}
\sum_{\substack{n\ge \zeta_t^\circ \colon\, A_1 < R_n(t) < A_2 }}^{\infty}\ee^{\theta t S_n(t)}	
\int_{B_1} \E_x\left[\exp\left(\int_0^t \theta V_{\cY_n(t)}^{(a_n(t))}(W_s)\dd s\right)\mathbbm{1}\{\tau_{B^\cc_{R_{n-1}(t)}(x)} \le t\}\right] \dd x.
\end{equation}
We wish to apply the bound \eqref{e:upperboundint} to the terms of \eqref{e:prkeyUB2}, 
with parameters chosen as follows:
\begin{align*}
\cY&=\cY_n(t),& R&=R_{n-1}(t),
& a& =a_{n}(t), &r&=r_{n}(t), & \gamma&=\gamma_n(t).
\end{align*} 
It is straightforward to verify that
we may (deterministically) choose $\chi \in [1,\infty)$ large enough such that,
with this choice of parameters,
whenever $\gamma_n(t)$ satisfies \eqref{e:condgamman} and $n \ge \zeta_t^\circ$,
the function $L = L(\cY_n(t), \theta, a_n(t), r_n(t), \gamma_n(t))$ in \eqref{e:defLvarrho} is uniformly bounded by a deterministic constant, and
\begin{equation}\label{e:prkeyUB3}
c_* a_n(t) \sqrt{\gamma_n(t)} > 2 \log (2 L) \qquad \text{(in particular, $\varrho < 1/2$)}.
\end{equation}
We may thus apply \eqref{e:upperboundint} to the terms in \eqref{e:prkeyUB2}, 
obtaining
\begin{align*}
& \int_{B_1} \E_x\left[\ee^{ \int_0^t\theta V_{\cY_n(t)}^{(a_n(t))}(W_s)-\gamma_n(t)\dd s} \mathbbm{1}\{\tau_{B^\cc_{R_{n-1}(t)}(x)} \le t\}\right] \dd x \\
& 
\le c_1 \mathfrak{N}_{\cY_n(t)}^{(r_n(t))}(B_1) \left\{ \sqrt{t} R_n(t)^\alpha \ee^{-\frac{ c_2 R_{n-1}(t)^2}{t}} + \ee^{-c_2 \sqrt{\gamma_n(t)} R_{n-1}(t)} \right\}
\end{align*}
for some deterministic constants $c_1,c_2\in(0,\infty)$,
where we also used $\sup_{x > 0} x \ee^{-x^2/b}\le \sqrt{b/2}$ for any $b>0$.
Together with the bound \eqref{e:prkeyUB2} and $\mathfrak{N}_{\cY_n(t)}^{(r_n(t))}(B_1) \leq \sqrt{1\vee|\cP \cap B_1|}$, 
this shows \eqref{e:keyUBL1}.

To show \eqref{e:keyUB}, we split instead according to whether $\hat{\tau}_{\zeta_t(x)-1} \leq t$ or not
and, if so, according to which $n \geq \zeta_t(x)$ satisfies $\hat{\tau}_{n-1}(x) \leq t < \hat{\tau}_{n}(x)$.
Arguing analogously as before, we obtain
\begin{align*}\label{e:prkeyUB4}
& \E_x\left[\ee^{\int_0^t \theta V^{(\frK)}(W_s)\dd s} \mathbbm{1}_{\{\tau_{B^\cc_{A_1}(x)} \le t < \tau_{B^\cc_{A_2}(x)} \}}\right] 
\le \E_x\left[\ee^{ \int_0^t \theta V^{(\frK)}(W_s)\dd s} \mathbbm{1}_{\{\tau_{B^\cc_{A_1}(x)} \le t < \hat{\tau}_{\zeta_t(x)-1}(x) \}}\right] \nonumber\\
& + \sum_{\substack{n\ge \zeta_t(x) \colon \\ A_1 < R_n(t) < A_2 }} \ee^{\theta t S_n(t)}\E_x\left[\exp\left( \int_0^t \theta V^{(a_n(t))}_{\cY_n(t)}(W_s)\dd s\right)\mathbbm{1}\{\tau_{B^\cc_{R_{n-1}(t)}(x)} \le t\}\right].
\end{align*}
Note that, when $n \ge \zeta_t(x)$,
$x\notin B_{r_n(t)}(\cY_n(t))$ and $R_{n-1}(t) \geq 8 r_n(t) N^{(r_n(t))}_{\cY_n(t)}$.
Applying \eqref{e:upperbound} to the terms in \eqref{e:prkeyUB2}, 
\eqref{e:keyUB} follows analogously as for \eqref{e:keyUBL1}.
The proofs for $|V^{(\frK)}|$, $|\overbar{V}|$ are identical.
\end{proof}

\subsubsection{Proof of Theorems \ref{t:finitenessva} and \ref{t:finiteness}}
\label{sss:prooffiniteness}

\begin{proof}
We start with Theorem~\ref{t:finitenessva}. 
Fix $\frK \in \scrK$.
Note that $v_\theta^{(\frK)}(t,x)$ is non-decreasing in $t$,
and thus it will be sufficient to show that,
for each $y\in\R^d$ and each $t > 0$, 
$\Prob$-almost surely,
\[ \int_{B_1(y)} v^{(\frK)}_\theta(t,x) \dd x < \infty \qquad \text{ and } \qquad
v^{(\frK)}_\theta(t,x) < \infty \;\;\forall\, x \in B_1(y) \setminus \cP.
\]
By the homogeneity of $\omega$, it is enough to consider $y=0$.
To this end, we will apply Lemma~\ref{l:keyUB} with
\begin{equation}\label{e:defRngammanfiniteness}
R_n(t) = 1 \vee (2^{n-1}t)^{\frac{k+1}{k-1}}, \qquad \gamma_n(t) = \max\left\{ 2 \Lambda_{R_n(t)}, \chi R^{2/k}_n \right\}, \qquad A_1 = 0, A_2 = \infty.
\end{equation}
Note that $R_n(t)$ satisfies \eqref{e:condRn(t)} and, 
by \eqref{e:defscalesUpbd} and the choice of $\alpha$, 
\eqref{e:condgamman} is fulfilled.

Let us first control the first term in the right-hand side of \eqref{e:keyUBL1}.
Recall Lemma~\ref{l:zetafinite}, \eqref{e:defsnt} and write
\begin{align}\label{e:prooffiniteness1}
\int_{B_1} \E_x \left[\ee^{\int_0^t\theta |V^{(\frK)}|(W_s) \dd s} \mathbbm{1}_{\{t < \hat{\tau}_{\zeta_t^\circ - 1}(x) \}} \right] \dd x 
\leq \ee^{ \theta t S_{\zeta_t^\circ-1}(t)} 
\int_{B_1} \E_x \left[\ee^{\int_0^t\theta V_{\cY_t}^{(a_t)}(W_s) \dd s} \mathbbm{1}_{\{t < \hat{\tau}_{\zeta_t^\circ - 1}(x) \}} \right] \dd x,
\end{align}
where $\cY_t := \cP \cap B_{R_{\zeta_t^\circ -1}(t)+2}$, $a_t = a_{\zeta_t^\circ - 1}(t)$
and we used that, 
if $t< \hat{\tau}_{\zeta_t^\circ-1}(x)$,
then $V^{(a_t)}(W_s)=V^{(a_t)}_{\cY_t}(W_s)$ for $0\leq s\leq t$.
Since $\cP \in \scrY$ a.s., the multipolar Hardy inequality in \cite{BDE08}[Theorem~1] implies that
$\Prob \big( \forall\, R>0 \colon\, \lambda_{\max} (\R^d, \theta V_{\cP \cap B_R} ) < \infty\big) = 1$,
and thus \eqref{e:prooffiniteness1} is finite by \eqref{e:truncationfinite} and \eqref{e:L1timedep}.
For the first term in the right-hand side of \eqref{e:keyUB},
let $\eps_x := \tfrac12 \dist(x,\cP)$ and fix $\hat{a}_x \in (0,\eps_x)$ to write
\[
\begin{aligned}
& \E_x \left[\exp \left\{\int_0^t\theta |V^{(\frK)}|(W_s) \dd s \right\} \mathbbm{1}{\{t < \hat{\tau}_{\zeta_t(x) - 1}(x) \}} \right] \\
\leq \, & \exp \left\{ \theta t \sup_{z \in B_{R_{\zeta_t -1}(t)}(x)} \left|V^{(\frK)}(z) - V^{(\hat{a}_x)}(z) \right| \right\} 
\E_x \left[\exp\left\{ \int_0^t\theta V_{\cY_t(x)}^{(\hat{a}_x)}(W_s) \dd s \right\} \mathbbm{1}\{t < \hat{\tau}_{\zeta_t(x) - 1}(x) \} \right],
\end{aligned}
\]
where $\cY_t(x) := \cP \cap B_{R_{\zeta_t-1}(t)+\hat{a}_x}(x)$,
so the latter is again finite by \eqref{e:truncationfinite} and \eqref{e:FKt}.

Consider now the series in \eqref{e:keyUBL1} and \eqref{e:keyUB}.
The term for $n=1$ is bounded by
\begin{equation}\label{e:prooffiniteness3}
c_1 \exp\{\theta t S_1(t) + t \gamma_1(t) + \log^+(\sqrt{t} R_1(t)^\alpha)\}.
\end{equation}
For some constants $c_3, c_4>0$ and using $\gamma_n(t) > R_n(t)^{2/k}$, we bound the terms for $n\ge2$ by
\begin{align}\label{e:prooffiniteness4}
c_3 \exp \left\{\theta t S_n(t) + t \gamma_n(t) - c_4 (2^n t)^{\frac{k+1}{k-1}} \right\}
\end{align}
Note that, $\Prob$-almost surely and as $n\to\infty$, $S_n(t) = o(R_n(t)^{2\alpha} \log R_n(t))$ by \eqref{e:errortruncation}
and, for any $\beta > 1/k$, $\gamma_n(t) = o(R_n(t)^{2\beta})$ by Lemma~\ref{l:evp} (applied with $R(t) = t$).
Therefore, the sum over $n \geq 2$ of \eqref{e:prooffiniteness4} is finite.
This finishes the proof of Theorem~\ref{t:finitenessva}.
The proof of Theorem~\ref{t:finiteness} is completely analogous.
\end{proof}

Lemma~\ref{l:keyUB} and the estimates in the proof above also allow us to show the following.
\begin{Lem}\label{l:macrobox}
For  any $\frK \in \scrK$ and any $\gamma \in (0,\infty)$ such that $ \gamma (k-1) > 2/d$,
\begin{equation}\label{e:macrobox}
\lim_{t \to \infty} \E_0 \left[ \exp \left\{\int_0^t |V^{(\frK)}|(W_s) \dd s \right\}\mathbbm{1}\left\{\sup_{0 \le s \le t}|W_s| \geq (\log t)^\gamma t^{\tfrac{k}{k - 1}} \right\}\right] = 0 \quad \Prob \mbox{-a.s.}
\end{equation}
When $d=3$, the same holds with $\overbar{V}$ in place of $V^{(\frK)}$.
\end{Lem}
\begin{proof}
Take $R_n(t)$, $\gamma_n(t)$ as in \eqref{e:defRngammanfiniteness}.
Using the bound \eqref{e:prooffiniteness4} for the $n$-th term of the series in \eqref{e:keyUB}, 
we bound the expectation in \eqref{e:macrobox} by
\begin{equation}\label{e:prmacrobox1}
c_3 \sum_{n=n_t+1}^{\infty} \exp\left\{ t \theta S_n(t)+ t \gamma_n(t)-c_4(2^nt)^{\frac{k+1}{k-1}}\right\}, 
\quad \text{ where } n_t := \left\lfloor \frac{ \gamma(k-1) \log_2 \log t}{k} \right\rfloor.
\end{equation}
Let $\beta \in (0,1)$ with $2/d < k \beta < \gamma(k-1)$.
Then $g(t) := (\log t)^\beta$ satisfies the conditions of Lemma~\ref{l:evp},
implying that $\gamma_n(t) = o( g(2^{n-1}t) R_n(t)^{2/k})$.
Using additionally the bound $S_n(t) = o( R_n(t)^{2 \alpha} \log R_n(t) )$ 
which holds almost surely by Lemma~\ref{l:errortruncation},
we may check that, when $t$ is large enough,
the exponents of the summands
in \eqref{e:prmacrobox1} are smaller than $- c_3 (2^n t)^{(k + 1)/(k -1)}$ for some constant $c_3 >0$,
from which \eqref{e:macrobox} follows.
The statement for $\overbar{V}$ is obtained analogously, considering $\overbar{S}_n(t)$.
\end{proof}

\subsubsection{Upper bound in Theorem~\ref{t:tightnessva}}
\label{sss:prooftightness}

\begin{proof}[Proof of \eqref{e:tightnessubva}]
Let $R_n(t)$, $\gamma_n(t)$ as in \eqref{e:defRngammanfiniteness}.
Recall \eqref{e:defzeta} and that, by Lemma~\ref{l:zetafinite},
$\zeta_t = 1$ a.s.\ for all large enough $t$.
Using Lemma~\ref{l:keyUB}, Lemma~\ref{l:macrobox}
and the estimates \eqref{e:prooffiniteness3}--\eqref{e:prooffiniteness4} for the terms of the series in \eqref{e:keyUB},
we see that it is enough to show that, as $t \to \infty$,
\begin{equation}\label{e:prooftightness1}
\log \E_0\left[\exp\left(\theta\int_0^tV^{(\frK)}(W_s)\dd s\right)\mathbbm{1}_{ \{t<\hat{\tau}_0\} }\right] = \mathcal{O}(t) \qquad \Prob \mbox{-a.s.,}
\end{equation}
\begin{align}\label{e:prooftightness2}
g(t)^{-1} t^{-\frac{2}{k-1}} \left\{\theta S_1(t) + \gamma_1(t) \right\} \to 0 \;\; \text{ in probability,}
\end{align}
and that, for any $\rho>0$,
\begin{align}\label{e:prooftightness3}
g(t)^{-1}t^{-\frac{k+1}{k-1}} \max_{2 \le n \le \lfloor \rho \log \log t  \rfloor} \left\{ t\theta S_n(t) +t \gamma_n(t) -c_4 (2^nt)^{\frac{k+1}{k-1}}\right\} \to 0 \;\; \text{ in probability.}
\end{align}
We start with \eqref{e:prooftightness1}.
When $t$ is large, $R_0(t) < 1$.
Let $\eps_0 := \tfrac12 \dist(0,\cP)$ and set $\hat{a}_0 := \tfrac12 (\eps_0 \wedge 1)$.
Note that, when $t < \tau_{B^\cc_1}$, $V^{(\hat{a}_0)}(W_s) = V^{(\hat{a}_0)}_{\cP_2}(W_s)$ for $0 \le s \le t$.
Moreover, $\lambda_{\max}(B_2, \theta V^{(\hat{a}_0)}_{\cP_2}) < \infty$ and $\dist(B_2^\cc, \cP_1)>0$.
Applying Lemma~\ref{l:errortruncation} and \eqref{e:FKt},
we find (random) constants $C_1, C_2 \in (0,\infty)$ such that, a.s.\ for all large enough $t$,
the expectation in \eqref{e:prooftightness1} is at most
\[
\ee^{\theta t \sup_{x \in B_2}\left|V^{(\frK)}(x) - V^{(\hat{a}_0)}(x)  \right|} 
\E_0 \left[ \exp\left\{ \int_0^t \theta V_{\cP_2}^{(\hat{a}_0)}(W_s) \dd s \right\} \mathbbm{1}_{\{ t<\tau_{B^\cc_2} \}}\right] \le C_1 \ee^{C_2 t}.
\]
Thus \eqref{e:prooftightness1} follows, and we move to \eqref{e:prooftightness2}--\eqref{e:prooftightness3}. 
Note first that, by \eqref{e:errortruncation},
\begin{equation}\label{e:convprob1}
\lim_{t \to \infty} \max_{1\le n\le\lfloor \rho \log \log t\rfloor} t^{-\frac{2}{k-1}} S_n(t) = 0 \quad \Prob \mbox{-a.s.,}
\end{equation}
and \eqref{e:prooftightness2} follows by Lemma~\ref{l:evp}.
To control the remaining term in \eqref{e:prooftightness3}, 
fix $\eps>0$ and estimate with a union bound
\begin{align}\label{e:proofseries}
\Prob \left(\max_{n\ge1} \frac{t \gamma_n(t)-c_4(2^{n}t)^{\frac{k+1}{k-1}}}{g(t) t^{\frac{k+1}{k-1}}}>\eps \right)
\le\sum_{n=1}^{\infty} \Prob \left(\gamma_n(t)>t^{\frac{2}{k-1}}\left(\eps g(t)+c_4(2^{n})^{\frac{k+1}{k-1}}\right)\right).
\end{align}
Now note that, since $g(t) \to \infty$, when $t$ is large enough, 
it is impossible to have $\gamma_n(t) = \chi R_n(t)^{2/k}$ if $\gamma_n(t)$ satisfies the inequality in \eqref{e:proofseries};
thus in this case $\gamma_n(t) = 2 \Lambda_{R_n(t)}$.
Applying \eqref{e:evbp}, we obtain deterministic constants $c_5,c_6 \in(0,\infty)$ such that \eqref{e:proofseries} is at most
\begin{align*}
& c_5 \sum_{n=1}^{\infty}R_n(t)^d\Big \{ t^{-\frac{dk}{k-1}} \Big(\eps g(t)+c_4(2^n)^{\frac{k+1}{k-1}}\Big)^{-\frac{dk}{2}}+ R_n(t)^{-\alpha d(k+1)}\Big\} \\
\leq \, & c_5 \sum_{n=1}^{\infty}\left(\frac{(2^n)^{\frac{2}{k-1}}}{\eps g(t)+c_4(2^n)^{1+\frac{2}{k-1}}}\right)^{\frac{dk}{2}}+c_6 \sum_{n=1}^{\infty} (2^n t)^{-\tfrac{d k}{k-1}(\alpha(k+1)-1)} \ \overset{t\rightarrow\infty}{\longrightarrow}0
\end{align*}
since $\alpha>(k+1)^{-1}$. Together with \eqref{e:convprob1}, this shows \eqref{e:prooftightness3}, completing the proof of \eqref{e:tightnessubva}.
\end{proof}

\subsubsection{Upper bound in Theorem~\ref{t:limsupva}}
\label{sss:proofUBlimsup}

Before we proceed to the proof,
we recall that, when $\ell$ is slowly varying,
$\ell(\lambda r) \sim \ell(r)$ as $r \to \infty$
uniformly over $\lambda$ in compact subsets (cf.\ \cite[Theorem~1.2.1]{BGT02}).
It is then straightforward to translate the integrability condition 
in \eqref{e:limsupva} into a summability condition, namely,
\begin{align}\label{e:integSV}
\int_{1}^{\infty}\frac{\dd t}{t \ell(t)} < \infty \qquad \text{ if and only if } \qquad \sum_{n=0}^\infty \ell(2^n)^{-1} < \infty.
\end{align}

\begin{proof}[Proof of the upper bound in \eqref{e:limsupva}]
Fix $t \mapsto \ell(t)$ slowly varying with 
$\int_{1}^{\infty}\frac{\dd r}{r\ell(r)}<\infty$,
and set
\begin{equation}\label{e:defRnUBlimpsup}
R_n(t) := (2^{n-1}t)^\frac{k}{k-1}\ell(2^{n-1}t)^{\frac{1}{d(k-1)}},
\quad \gamma_n(t) :=  \max\left(2 \Lambda_{R_n(t)}, R_n(t)^{\frac{2}{k}} \ell(2^{n-1} t)^{\frac{2}{dk}}\right).
\end{equation}
When $t$ is large, $R_n(t) \ge 1$, $\zeta_t =1$ and \eqref{e:condgamman} holds,
so we may apply Lemma~\ref{l:keyUB}.

Note that \eqref{e:prooftightness1} still holds as $R_0(t)$ is given by \eqref{e:defscalesUpbd},
and thus the first term in the right-hand side of \eqref{e:keyUB} is controlled.
For the term with $n=1$, note that, by Lemma~\ref{l:evp} (with $g(t) = \ell(t)^{\frac{2}{dk}}$), 
\eqref{e:integSV} and \eqref{e:errortruncation}, \eqref{e:prooftightness2} holds almost surely with $g(t) = \ell(t)^{\frac{2}{d(k-1)}}$.
It is thus enough to show that, $\Prob$-a.s.,
\begin{equation}\label{e:prUBlimsup1}
\limsup_{t\to\infty} \sum_{n\ge2} \exp \left( 
\theta t S_n(t) + t \gamma_n(t) + \log (t R_n(t)^\alpha) - c_2 R_{n-1}(t) \min \left\{ \frac{R_{n-1}(t)}{t}, \sqrt{\gamma_n(t)}\right\} 
\right) < \infty.
\end{equation}
To this end, use the slow variation of $\ell$ to find a constant $c>0$ such that, for $t$ large enough,
\begin{equation}\label{e:prUBlimpsup2}
R_{n-1}(t) \min \left\{ \frac{R_{n-1}(t)}{t}, \sqrt{\gamma_n(t)}\right\} \ge c 2^{n-1} t (2^{n-1} t)^{\frac{2}{k-1}} \ell(2^{n-1} t)^{\frac{2}{d(k-1)}}, \quad n\ge2.
\end{equation}
Applying Lemma~\ref{l:evp} (with $t$ substituted by $2^{n-1} t$), we obtain $c' \in (0,\infty)$ such that
\[
t \gamma_n(t) \le t c' R_n(t)^{\frac{2}{k}} \ell(2^{n-1} t)^{\frac{2}{dk}} = c' t (2^{n-1} t)^{\frac{2}{k-1}} \ell(2^{n-1}t)^{\frac{2}{d(k-1)}}, \quad n \ge 2.
\]
Noting that, by \eqref{e:errortruncation}, the remaining terms are of lower order,
we can choose $n_0 = n_0(c,c')$ sufficiently large so that, for any $n \geq n_0$,
the $n$-th term in the series in \eqref{e:prUBlimsup1} is bounded by the exponential of $-c_3 t (2^{n-1} t)^{\frac{2}{k-1}} \ell(2^{n-1}t)^{\frac{2}{d(k-1)}}$
for some constant $c_3>0$, showing \eqref{e:prUBlimsup1}.
This finishes the proof.
\end{proof}

\subsubsection{Upper bound in Theorem~\ref{t:liminfva}}
\label{sss:proofUBliminf}
%
\begin{proof}
Let $A:=c_2 \wedge 1$ with $c_2$ as in Lemma~\ref{l:keyUB}, and set
\[
R_n(t) := 2^{n-1} t^{\frac{k}{k-1}} (\log \log t)^{-\frac{1}{d(k-1)}}, \qquad \gamma_n(t) := \max \left\{ 2 \Lambda_{R_n(t)}, \frac{A^2}{4} t^{-2}R_{n-1}(t)^2\right\}.
\]
Applying Lemma~\ref{l:keyUB}, Lemma~\ref{l:macrobox}, Lemma~\ref{l:errortruncation} and \eqref{e:prooftightness1},
we see that we only need to find a constant $C^{\inf} \in (0,\infty)$ such that, for all $\rho>0$,
\begin{equation}\label{e:prUBliminf1}
\liminf_{t \to \infty} \max_{1 \le n \le \lfloor \rho \log \log t \rfloor} 
\frac{\gamma_{n}(t) - A \mathbbm{1}_{\{n \ge 2\}} t^{-1} R_{n-1}(t) \min\left\{ t^{-1}R_{n-1}(t), \sqrt{\gamma_n(t)} \right\}}{(\log \log t)^{-\frac{2}{d(k-1)}}t^{\frac{2}{k-1}}} \le C^{\inf}.
\end{equation}
Abbreviate $s_t:= t^{\frac{2}{k-1}} (\log \log t)^{-\frac{2}{d(k-1)}}$.
Note that $t^{-1}R_{n-1}(t) = 2^{n-1} \sqrt{s_t}$ and, since $A\le 1$,
\begin{equation}\label{e:prUBliminf2}
\min\left\{t^{-1}R_{n-1}(t), \sqrt{\gamma_n(t)} \right\} \geq \frac{A}{2} t^{-1}R_{n-1}(t) = \frac{A}{2} 2^{n-1} \sqrt{s_t}.
\end{equation}
Hence, as $t\to\infty$,
\begin{equation}\label{e:prUBliminf3}
\frac{A^2}{4} t^{-2} R_{n-1}(t)^2 - A \mathbbm{1}_{\{n \ge 2\}} t^{-1} R_{n-1}(t) \min\left\{ t^{-1}R_{n-1}(t), \sqrt{\gamma_n(t)} \right\}
\ll s_t.
\end{equation}
On the other hand, Lemma~\ref{l:liminfev} provides a constant $C \in (1,\infty)$
and a subsequence $t_j \to \infty$ as $j\to\infty$ such that, for all $j \in \N$ and all $1 \le n \le \lfloor \rho \log \log t_j \rfloor$,
\begin{equation}\label{e:prUBliminf4}
\Lambda_{R_n(t_j)} \le \frac{A^2}{4} \mathbbm{1}_{\{n \ge 2\}} t_j^{-2}R_{n-1}(t_j)^2 + C s_{t_j}.
\end{equation}
Now \eqref{e:prUBliminf2}--\eqref{e:prUBliminf4} and the definition of $\gamma_n(t)$ imply \eqref{e:prUBliminf1} with $C^{\inf} = 2C$.
\end{proof}

\subsection{The lower bounds}
\label{ss:proofslowerbounds}

In this section, we will prove the lower bounds in Theorems~\ref{t:tightnessva}, \ref{t:limsupva} and \ref{t:liminfva}
for the truncated potentials $V^{(a)} = V^{(\frK_a)}$ where $\frK_a(x) = |x|^{-2}\mathbbm{1}_{\{|x| \leq a\}}$.
The proof of the theorems will be finished in Section~\ref{ss:prooflimitfrac}
after the proof of Theorem~\ref{t:limitfrac}.
The following lemma will be used in all the proofs of this section:
\begin{Lem}\label{l:lemlb}
There exists $c \in (0,1]$ such that the following holds $\Prob$-almost surely.
Fix $a \in (0,\infty)$ and let $R(t)$, $r(t)$ satisfy 
$\ee^{-t} \ll r(t) \ll 1 \ll R(t)$ as $t \to \infty$, 
and $R(t) r(t) \leq \sqrt{c} t$ for all $t$ large enough.
Define
\begin{equation}\label{e:defAt}
\cA_t:=\left\{\exists x\in B_{R(t)}\colon \omega\left(B_{r(t)}(x)\right) \geq k+1\right\}.
\end{equation}
Then, for all large enough $t$, on $\cA_t$,
\begin{equation}\label{e:lemlb}
\log\E_0\left[\exp\left(\theta\int_0^tV^{(a)}(W_s)\dd s\right)\right] \geq \frac{c t}{r(t)^2} - 2 \sqrt{c}\frac{R(t)}{r(t)}-\mathcal{O}(t).
\end{equation}
\end{Lem}

\begin{proof}
On $\cA_t$, we pick $x_t\in B_{R(t)}$ such that $\omega\left(B_{r(t)}(x_t)\right) \geq k+1$.
Choose distinct $y_1, \ldots, y_{k+1} \in \cP \cap B_{r(t)}(x_t)$
(for example, according to lexicographical order) and set $\cY_t:=\{y_1, \ldots, y_k\}$.
Clearly $V^{(a)}\geq V^{(a)}_{\cY_t}$, and $V_{\cY_t}-V^{(a)}_{\cY_t}\le \#\cY_t a^{-2} =(k+1)a^{-2} =: c_0$.
Thus
\begin{align}\label{e:prooflb1}
\log\E_0\left[\exp\left(\theta\int_0^tV^{(a)}(W_s)\dd s\right)\right] \geq -\theta c_0 t + \log\E_0\left[\exp\left(\theta\int_0^t V_{\cY_t}(W_s)\dd s\right) \right].
\end{align}
Let now $K, c_1, c_2$ as in Lemma~\ref{l:keyLB}, and set $c:=c_2 \wedge 1$. Write, for $0\le t_0 \le t$,
\begin{align}\label{e:prooflb2}
\E_0\left[\ee^{\int_0^t \theta V_{\cY_{t}}(W_s)\dd s}\right]
\geq \E_0\left[\ee^{\int_{t_0}^t \theta V_{\cY_t}(W_s)\dd s}\mathbbm{1}_{\{W_s \in B_{Kr(t)}(x_t) \, \forall s \in [t_0,t]\}}\right].
\end{align}
Denote by $p(0,y,t) = (2 \pi t)^{-d/2} \ee^{-|y|^2/(2t)}$ the probability density of Brownian motion at time $t$ started at $0$.
Applying the Markov property, we see that \eqref{e:prooflb2} equals
\begin{align}\label{e:prooflb3}
& \int_{B_{Kr(t)}(x_t)}p(0,y,t_0)\E_y\left[\ee^{\theta \int_0^{t-t_0}V_{\cY_{t}}(W_s)\dd s}\mathbbm{1}_{\{\tau_{B^\cc_{Kr(t)}(x_t)} > t-t_0\}}\right]\dd y\nonumber\\
\geq \, & (2\pi t)^{-\frac{d}{2}}\ee^{-\frac{R(t)^2}{t_0}} \int_{B_{Kr(t)}(x_t)}\E_y\left[\ee^{\int_0^{t-t_0} \theta V_{\cY_{t}}(W_s)\dd s}\mathbbm{1}_{\{\tau_{B^\cc_{Kr(t)}(x_t)} > t-t_0\}}\right]\dd y,
\end{align}
where we used $|y|^2 \leq (|x_t|+K r(t))^2 \leq 2 R(t)^2$ for large $t$.
The integral above can be identified with the integral in \eqref{e:keyLB} with $a=r(t),x=x_t$, 
implying that \eqref{e:prooflb3} is at least
\begin{equation}
c_1 (2\pi t)^{-\frac{d}{2}} r(t)^d \exp \left\{-\frac{R(t)^2}{t_0}+c_2 (t-t_0) r(t)^{-2} \right\}. 
\end{equation}
Maximizing the exponent over $t_0\in(0,t)$ we obtain
$t_0= R(t)r(t) / \sqrt{c_2} \le t$,
which yields
\begin{align}\label{e:prooflb4}
\log\E_0\left[\ee^{\theta \int_0^tV_{\cY_{t}}(W_s)\dd s}\right]
\geq -2 \sqrt{c_2}\frac{R(t)}{r(t)}+\frac{c_2 t}{r(t)^2} + \log\left(c_1(2\pi t)^{-\frac{d}2} r(t)^{d}\right).
\end{align}
Now \eqref{e:lemlb} follows from \eqref{e:prooflb1}, \eqref{e:prooflb4} and our assumptions on $R(t)$, $r(t)$.
\end{proof}
With Lemma~\ref{l:lemlb} at hand, we are ready to complete the proofs of Theorems~\ref{t:tightnessva}, \ref{t:limsupva} and \ref{t:liminfva}
in the special case of the truncated kernels $\frK = \frK_a$.

\begin{Lem}\label{l:LBtightnessVa}
For any $a \in (0,\infty)$, 
\eqref{e:tightnesslbva} holds with $\frK(x) = \frK_a(x) = |x|^{-2} \mathbbm{1}_{\{|x| \leq a\}}$.
\end{Lem}

\begin{proof}
We may assume that $g(t)=t^{o(1)}$ as $t\to\infty$.
Take $c$ as in Lemma~\ref{l:lemlb}, let $A>0$ and
\begin{equation*}\label{e:prLBtightness1}
R(t)=  \tfrac12 \sqrt{ \tfrac12 A g(t)^{-1}}t^{\frac{k}{k-1}}, \quad r(t)= \sqrt{\tfrac12 c A^{-1} g(t)}t^{-\frac{1}{k-1}}.
\end{equation*} 
Lemma~\ref{l:lemlb} implies that, on the event $\cA_t$ defined in \eqref{e:defAt},
\begin{equation*}\label{e:prLBtightness2}
g(t)t^{-\frac{k+1}{k-1}}\log\E_0\left[\exp\left(\theta\int_0^tV^{(a)}(W_s)\dd s\right)\right] \geq A - o(1).
\end{equation*}
Since $\lim_{t\rightarrow\infty}\Prob(\cA_t)=1$ by Corollary~\ref{c:lbppp} and $A$ is arbitrary,
we conclude \eqref{e:tightnesslbva}.
\end{proof}

\begin{Lem}\label{l:LBlimsupVa}
For any $a \in (0,\infty)$, 
the lower bound in \eqref{e:limsupva} holds with $\frK(x) = \frK_a(x) = |x|^{-2} \mathbbm{1}_{\{|x| \leq a\}}$.
\end{Lem}

\begin{proof}
Let $\ell(t) \ge 1$ be slowly varying with $\int_{1}^{\infty}\frac{\dd r}{r\ell(r)}=\infty$.
Fix $A>0$ and set $R(t) := \sqrt{\tfrac{A}{8}} t^{\frac{k}{k-1}} \ell(t)^{\frac{1}{d(k-1)}}$, 
$r(t):= \sqrt{\tfrac12 c A^{-1}} t^{-\frac{1}{k-1}} \ell(t)^{-\frac{1}{d(k-1)}}$
with $c$ as in Lemma~\ref{l:lemlb}.
On $\cA_t$,
\[
\ell(t)^{-\frac{2}{d(k-1)}} t^{-\frac{k+1}{k-1}}\log\E_0\left[\exp\left(\theta\int_0^tV^{(a)}(W_s)\dd s\right)\right] \geq A - o(1).
\]
Now Lemma~\ref{l:pplimsup} and \eqref{e:integSV} provide
a sequence $t_j \to \infty$ as $j \to \infty$ such that $\cA_{t_j}$ occurs,
and we conclude by taking $A\uparrow\infty$.
\end{proof}

%
\begin{Lem}\label{l:LBliminfVa}
For any $a \in (0,\infty)$, 
the lower bound in \eqref{e:liminfva} holds with $\frK(x) = \frK_a(x) = |x|^{-2} \mathbbm{1}_{\{|x| \leq a\}}$.
\end{Lem}

\begin{proof}
Let $c$ as in Lemma~\ref{l:lemlb} and pick $\mu, \nu>0$ satisfying
\begin{equation}\label{e:condmunu}
\mu \nu < \sqrt{c} \qquad \text{ and } \qquad (\mu \nu^k)^d > \frac{2^d (k+1)!}{|B_1|}.
\end{equation}
Set $R(t) := \mu t^{\frac{k}{k-1}} (\log \log t)^{-\frac{1}{d(k-1)}}$ and $r(t) := \nu t^{-\frac{1}{k-1}} (\log \log t)^{\frac{1}{d(k-1)}}$.
By Lemma~\ref{l:lemlb},  on $\cA_t$,
\begin{align}\nonumber
(\log \log t)^{\frac{2}{d(k-1)}} t^{\frac{k+1}{k-1}} \log\E_0\left[\ee^{\theta \int_0^t V^{(a)}(W_s)\dd s}\right]
& \geq \frac{c}{\nu^2} - 2 \sqrt{c} \frac{\mu}{\nu}-o(1).
\end{align}
On the other hand, by Lemma~\ref{l:PPPliminf1},
$\cA_t$ occurs for all large enough $t$, 
and thus we may take $C_{\inf}$ as the maximum of $c\nu^{-2} - 2 \sqrt{c} \mu/\nu$ over $\mu, \nu>0$ satisfying \eqref{e:condmunu},
which turns out to be
\[
C_{\inf}:= \frac{c^{\frac{k}{k-1}} (k-1)}{(k+1)^{\frac{k+1}{k-1}}} \left(\frac{|B_1|}{2^d (k+1)!}\right)^{\frac{2}{d(k-1)}}.
\qedhere
\]
\end{proof}

\subsection{Proof of Theorems~\ref{t:limitfrac}, \ref{t:tightnessva}, \ref{t:limsupva}, \ref{t:liminfva}, and \ref{t:tightness}}
\label{ss:prooflimitfrac}
%
\begin{proof}[Proof of Theorem~\ref{t:limitfrac}]
Let $I^{(\frK)}_t:=\exp \int_0^t \theta V^{(\frK)}(W_s)\dd s$ and  
$I^{(a)}_t:=\exp \int_0^t \theta V^{(a)}(W_s)\dd s$.
Let $R(t) := (\log t)^\gamma t^{k/(k+1)}$ with $\gamma$ as in Lemma~\ref{l:macrobox},
and put $S_t := \sup_{x \in B_{R(t)} \setminus \cP} |V^{(\frK)}(x) - V^{(a)}(x)|$.
Then
\begin{align*}
\log \E_0 \left[ I^{(a)}_t \mathbbm{1}\{\tau_{B_{R(t)}^\cc} \ge t \}\right] - \theta t S_t \le \log \E_0 \left[ I^{(\frK)}_t \mathbbm{1}\{\tau_{B_{R(t)}^\cc} \ge t \}\right] \le \log \E_0 \left[ I^{(a)}_t \mathbbm{1}\{\tau_{B_{R(t)}^\cc} \ge t \}\right] + \theta t S_t.
\end{align*}
Now, by Lemma~\ref{l:LBliminfVa} and \eqref{e:errortruncation} with $a_R \equiv a$,
$t S_t / \log \E_0\left[ I^{(a)}_t \right]$ tends a.s.\ to $0$.
To obtain \eqref{e:limitfrac}, note that, by Lemma~\ref{l:macrobox}, $\E_0 \left[I^{(a)}_t \mathbbm{1}\{\tau_{B^{\cc}_{R(t)}} \ge t\} \right] \sim \E_0 \left[I^{(a)}_t \right]$, and the same can be concluded for $I^{(\frK)}_t$ by taking into account the first inequality above.
The proofs for $|V^{(\frK)}|$, $\overbar{V}$ or $|\overbar{V}|$ are identical.
\end{proof}

As anticipated, this allows us to finally give the:

\begin{proof}[Proof of Theorems~\ref{t:tightnessva}, \ref{t:limsupva} and \ref{t:liminfva}]
For $\frK = \frK_a$, the results follow from the upper bounds in Section~\ref{ss:proofsupperbounds} 
together with Lemmas~\ref{l:LBtightnessVa}--\ref{l:LBliminfVa}.
The other cases then follow from Theorem~\ref{t:limitfrac}.
\end{proof}

\begin{proof}[Proof of Theorem~\ref{t:tightness}]
Follows from Theorems~\ref{t:limitfrac}, \ref{t:finiteness}, \ref{t:finitenessva}, \ref{t:tightnessva}, \ref{t:limsupva} and \ref{t:liminfva}.
\end{proof}

\subsection{Proof of Theorems~\ref{t:solva} and~\ref{t:solvbar}}
\label{ss:proofssol}
%
\begin{proof}
We follow \cite[Proposition~1.6]{CK12}.
We start with Theorem~\ref{t:solva}. Fix $\frK \in \scrK$. 
Let $m \in \N$ and $F_m:=|\theta V^{(\frK)}|\wedge m \in L^{\infty}(\R^d)$.
By Proposition~\ref{p:FKrep}, 
$w_m(t,x):=\E_x\left[\ee^{\int_0^t F_m(W_s) \dd s} \1{\{ \tau_{B_m^\cc} > t \}}\right]$ 
satisfies
\begin{equation}\label{e:proofmild1}
w_m(t,x)=1+\int_0^t\int_{\R^d}p_{t-s}(x-y)F_m(y) w_m(s,y)\dd y\dd s, \quad (t,x)\in(0,\infty)\times B_m \setminus \cP
\end{equation}
with $p_t(x)$ as in \eqref{e:Gaussiandensity}.
Letting $m\uparrow\infty$ and applying the monotone convergence theorem, we see that \eqref{e:proofmild1} 
still holds true with $F_m$ and $v_m$ replaced by $|\theta V^{(\frK)}|$ and  $v^{(\frK)}_\theta(t,x)= \E_x \left[\exp \int_0^t \theta|V^{\frK}|(W_s) \dd s \right]$, 
both sides being finite almost surely by Theorem~\ref{t:finitenessva}.
In particular, 
\begin{equation}\label{proofmild2}
\int_0^t\int_{\R^d}p_{t-s}(x-y)|V^{(\frK)}|(y) |v^{(\frK, u_0)}_\theta(s,y)| \dd y\dd s <\infty, \quad (t,x)\in(0,\infty)\times \R^d\setminus \cP.
\end{equation}
Noting that $t \mapsto \ee^{\pm\int_0^t \theta V^{(\frK)}(W_s) \dd s }$ are absolutely continuous,  the fundamental theorem of calculus gives
\[
\exp\left(\int_0^t \theta V^{(\frK)}(W_s) \dd s\right)=1+\int_0^t \theta V^{(\frK)}(W_s)\exp\left(\int_s^t \theta V^{(\frK)}(W_s) \dd u\right) \dd s,
\]
which we use to write, for all $(t,x)\in(0,\infty)\times \R^d\setminus \cP$ and all $u_0 \in L^\infty(\R^d)$,
\begin{align*}
u^{(\frK, u_0)}_\theta(t,x) 
& = \int p_t(x-y) u_0(y) \dd y +\int_0^t\E_x\left[\theta V^{(\frK)}(W_s) \ee^{\int_s^t \theta V^{(\frK)}(W_u)\dd u }\right] \dd s \\
& =\int p_t(x-y) u_0(y) \dd y+\int_0^t\E_x\left[\theta V^{(\frK)}(W_s)u^{(\frK)}_\theta(t-s,y)\right] \dd s \\
& =\int p_t(x-y) u_0(y) \dd y+\int_0^t\int_{\R^d}p_{t-s}(x-y)\theta V^{(\frK)}(y) u^{(\frK)}_\theta(s,y) \, \dd y \, \dd s,
\end{align*}
where we used Fubini's theorem (which is justified by \eqref{proofmild2}), 
 the Markov property of $W$, and time reversal.
The same argument works for $|V^{(\frK)}|$, and thus we complete the proof of Theorem~\ref{t:solva}.
Theorem~\ref{t:solvbar} is proved analogously.
\end{proof}

\vspace{10pt}
\noindent
{\bf Acknowledgements.}
The authors are thankful to Achim Klenke for suggesting the problem, 
and to Wolfgang K\"onig for fruitful discussions.
The research of RdS was supported by the DFG projects KO 2205/13, KO 2205/11
and by the DFG Research Unit FOR2402.
We thank WIAS Berlin and JGU Mainz 
for hospitality and financial support during several research visits.


\end{document}